\documentclass[11pt]{aims}
\usepackage{amsmath, amssymb, mathrsfs}
  \usepackage{paralist}
  \usepackage{graphics} 
  \usepackage{epsfig} 
\usepackage{graphicx}  \usepackage{epstopdf}
 \usepackage[colorlinks=true]{hyperref}
 \hypersetup{urlcolor=blue, citecolor=red}
 \usepackage[toc,page]{appendix}
\usepackage{chngcntr}
\usepackage{hyperref}

\usepackage{titlesec}
\titleformat{\section}{\Large\bfseries}{\thesection.}{4pt}{}
\titleformat{\subsection}{\large\bfseries}{\thesection.\arabic{subsection}.}{4pt}{}
\titleformat{\subsubsection}{\bfseries}{\thesection.\arabic{subsection}.\arabic{subsubsection}.}{4pt}{}
\titleformat*{\paragraph}{\bfseries}
\titleformat*{\subparagraph}{\bfseries}
  \def\currentyear{}
   \def\currentmonth{}
    \def\ppages{}
     \def\DOI{}

\usepackage[margin=1in]{geometry}

\newtheorem{theorem}{Theorem}[section]

\newtheorem{lemma}[theorem]{Lemma}
\newtheorem{proposition}[theorem]{Proposition}
\theoremstyle{definition}
\newtheorem{definition}[theorem]{Definition}
\newtheorem{remark}[theorem]{Remark}

\newcommand{\Rb}{\mathbb{R}}
\newcommand{\Sb}{\mathbb{S}}
\newcommand{\NN}{\mathbb{N}}

\newcommand{\Cc}{\mathcal{C}}
\newcommand{\Dc}{\mathcal{D}}
\newcommand{\Ec}{\mathcal{E}}
\newcommand{\Fc}{\mathcal{F}}

\newcommand{\Oc}{\mathcal{O}}
\newcommand{\Mc}{\mathcal{M}}

\newcommand{\Sc}{\mathcal{S}}

\newcommand{\Uc}{\mathcal{U}}
\newcommand{\Vc}{\mathcal{V}}

\newcommand{\pd}{\partial}
\newcommand{\py}{\partial_y}
\newcommand{\ps}{\partial_s}
\newcommand{\pr}{\partial_r}
\newcommand{\pt}{\partial_t}

\newcommand{\Ls}{\mathscr{L}}
\newcommand{\Hs}{\mathscr{H}}
\newcommand{\As}{\mathscr{A}}
\newcommand{\Es}{\mathscr{E}}

\numberwithin{equation}{section}

\title[1-corotational energy supercritical wave maps] 
      {Construction of type II blowup solutions for the 1-corotational energy supercritical wave maps}

\author[T. Ghoul, S. Ibrahim,  and V. T. Nguyen]{}

\email[S. Ibrahim]{ibrahim@math.uvic.ca}
 \email[T. Ghoul]{teg6@nyu.edu}
 \email[V. T. Nguyen]{Tien.Nguyen@nyu.edu}
 
\thanks{\today}
\begin{document}
\maketitle


\centerline{\scshape T. Ghoul$^a$, S. Ibrahim$^{a,b}$  and V. T. Nguyen$^a$}
\medskip
{\footnotesize
 \centerline{$^a$Department of Mathematics, New York University in Abu Dhabi,}
   \centerline{Saadiyat Island, P.O. Box 129188, Abu Dhabi, United Arab Emirates.}
}
\medskip
{\footnotesize
 \centerline{$^b$Department of Mathematics and Statistics, University of Victoria,}
   \centerline{PO Box 3060 STN CSC, Victoria, BC, V8P 5C3, Canada.}
}

\bigskip

\begin{abstract} We consider the energy supercritical wave maps from $\Rb^d$ into the $d$-sphere $\mathbb{S}^d$ with $d \geq 7$. Under an additional assumption of 1-corotational symmetry, the problem reduces to the one dimensional semilinear wave equation
$$\partial_t^2 u = \partial^2_r u + \frac{(d-1)}{r}\partial_r u - \frac{(d-1)}{2r^2}\sin(2u).$$
We construct for this equation a family of $\mathcal{C}^{\infty}$ solutions which blow up in finite time  via concentration of the universal profile
$$u(r,t) \sim Q\left(\frac{r}{\lambda(t)}\right),$$
where $Q$ is the stationary solution of the equation and the speed is given by the quantized rates
$$\lambda(t) \sim c_u(T-t)^\frac{\ell}{\gamma}, \quad \ell \in \mathbb{N}^*, \;\; \ell > \gamma = \gamma(d) \in (1,2].$$
The construction relies on two arguments: the reduction of the problem to a finite-dimensional one thanks to a robust universal energy method and modulation techniques developed by Merle, Rapha\"el and Rodnianski \cite{MRRcjm15} for the energy supercritical nonlinear Schr\"odinger equation, then we proceed by contradiction to solve the finite-dimensional problem and conclude using the Brouwer fixed point theorem. 

\end{abstract}

\section{Introduction.} 
Let $(N,h)$ be a complete smooth Riemannian manifold of dimension $d$ with $\pd N=\emptyset$. We denote spacetime coordinates on $\Rb^{1 + d}$ as $(t,x) = (x_\alpha)$ with $0 \leq \alpha \leq d$. A wave map $\Phi : \Rb^{1+d} \mapsto N$ is formally defined as a critical point of the Lagrangian 
$$\mathcal{L}(\Phi,\pd \Phi)=\int_{\Rb^{1 + d}}g^{\alpha\mu}\Big<\pd_\alpha \Phi,\pd_\mu \Phi\Big>_h dtdx,$$
where $g = \textup{diag}(-1, 1, \cdots, 1)$ is the Minkowski metric on $\Rb^{1 + d}$ and $\pd_\alpha = \frac{\pd}{\pd x_\alpha}$. In the local coordinates on $(N,h)$, the critical points of $\mathcal{L}$ satisfy the equation
\begin{equation}\label{eq:Wavemaps}
\square_g \Phi^k+g^{\alpha\mu}\Gamma^k_{ij}(\Phi)\pd_\alpha \Phi^i\pd_\mu \Phi^j=0, \quad 1 \leq k \leq d,
\end{equation}
where $\Gamma^k_{ij}$ are Christoffel symbols associated to the metric $h$ of the target manifold $N$, and $\square_g$ stands for the Laplace-Beltrami operator on $(\Rb^{1 + d},g)$ defined by 
$$
\square_g u = \pd_{tt} - \Delta.
$$
A special case is when the target manifold $N = \Sb^d \hookrightarrow \Rb^{1 + d}$, equation \eqref{eq:Wavemaps} becomes
\begin{equation}\label{eq:WM}
\partial_{t}^2\Phi - \Delta \Phi = \Phi(|\nabla \Phi|^2 - |\pt \Phi|^2).
\end{equation}
Under the assumption of 1-corotational symmetry, namely that the solution takes the form
$$\Phi(x,t) = \binom{\cos(u(|x|,t))}{\frac{x}{|x|}\sin(u(|x|,t))},$$ 
equation \eqref{eq:WM} reduces to the semilinear wave equation 
 \begin{equation}\label{Pb}
 \arraycolsep=1.4pt\def\arraystretch{1.6}
\left\{\begin{array}{rl}
\pt^2 u &= \partial^2_r u + \frac{(d-1)}{r}\partial_r u - \frac{(d-1)}{2r^2}\sin(2u),\\
(u,\pt u)\mid_{t = 0} &= (u_0, u_1),
\end{array}\right.
\end{equation}
where $u(t): r \in \Rb_+ \to u(r,t) \in \Rb_+$. The set of solutions to \eqref{Pb} is invariant by the scaling symmetry
$$\vec u(r,t) := \big(u, \pt u\big)(r,t) \longmapsto \vec u_\lambda(r,t):=\left(u, \frac{1}{\lambda}\pt u\right)\left(\frac{r}{\lambda}, \frac{t}{\lambda} \right), \quad \forall\lambda > 0.$$
The problem  \eqref{Pb} exhibits a conserved energy 
\begin{equation}\label{def:Enut}
\Ec(\vec u)(t) = \int_0^{+\infty} \left( |\pt u|^2 + |\partial_r u|^2 + \frac{(d-1)}{r^2}\sin^2(u)\right)r^{d-1}dr = \textup{const.},
\end{equation}
which satisfies
$$\Ec(\vec u_\lambda) = \lambda^{d-2}\Ec(\vec u).$$
This means that the wave map problem \eqref{Pb} is energy subcritical if $d = 1$, critical if $d =2$ and supercritical if $d \geq 3$.

The Cauchy problem  for wave maps has been extensively studied, see for examples, Shatah and Shadi Tahvildar-Zadeh \cite{SScpam94},  Shatah and Struwe \cite{SSbook98, SSimrn02}, Struwe \cite{Sndea97}, Tataru \cite{Tcpde98,Tajm01}. It is well understood that the Cauchy problem is locally well posed for initial data in $H^s \times H^{s - 1}$ with $s > \frac{d}{2}$ (see Klainerman and Machedon \cite{KMdm95} for $d \geq 3$, Klainerman and Selberg \cite{KScpde97} for $d = 2$, Keel and Tao \cite{KTimrn98} for $d = 1$) and the solution can be continued as long as the $H^s$-norm remains bounded.  We refer the reader to the paper by Krieger \cite{Ksdg08} for a survey on these results and a detailed list of references. It is well known that the solution $u(r,t)$ may develop singularities in some finite time (see for example, \cite{CSTihp98} and \cite{Scpam88}). In this case, we say that $u(r,t)$ blows up in a finite time $T < +\infty$ in the sense that 
$$\lim_{t \to T}\|\nabla u(t)\|_{L^\infty} = + \infty.$$
Here we call $T$ the blowup time. In this paper, a blowup solution is called Type I if 
\begin{equation}\label{def:Typ1blowup}
\limsup_{t \to T} (T-t)\|\nabla u(t)\|_{L^\infty} < +\infty,
\end{equation}
otherwise, it is called Type II. \\

In the energy critical case $d = 2$, Struwe \cite{Smz02,Scpam03} proved that blowup cannot be self-similar. A solution $u$ is said to be self-similar if it is of the form 
$$u(r,t) = \varphi(y), \quad y = \frac{r}{T-t},$$
where $T$ is a positive constant and $\varphi$ is a smooth function solving the ordinary differential equation
\begin{equation}\label{eq:phi}
(1 - y^2)\varphi_{yy} + \left(\frac{d - 1}{y} - 2y\right)\varphi_y - \frac{(d-1)}{2y^2}\sin(2\varphi) = 0.
\end{equation}
Note that Struwe's result does not imply that blowup actually occurs. However, numerical evidences by Bizon, Chmaj and Tabor \cite{BCTnon01}, Isenberg and Liebling \cite{ILjmp02} strongly suggest singularity development for certain positively curved targets. Later, the existence of finite time blowup solutions for equivariant wave maps from $(2+1)$ Minkowski space to the $\mathbb{S}^2$-sphere  has been constructively proved by Krieger, Schlag and Tataru \cite{KSTim08},  C\^arstea \cite{Ccmp10}, Rodnianski and Sterbenz \cite{RSam10}, Rapha\"el and Rodnianski \cite{RRmihes12}. It is worth mentioning the work by C\^ote \textit{et al.} \cite{CKLSajm15I, CKLSajm15II} where the authors establish a classification of blowup solutions of topological degree one \footnote{Following the definition in \cite{CKLSajm15I}, a solution $\vec u$ is of degree $n$ if $\Ec(\vec u)$ is finite and $u(r=0, t) = 0$, $u(r = \infty, t) = n\pi$.} with energies less than $3\Ec(Q,0)$, where $Q(r) = 2\arctan(r)$ is the unique (up to scaling) non trivial solution to the equation 
$$Q_{rr} + \frac{1}{r}Q_r= \frac{\sin(2Q)}{2r^2}.$$  
In particular, they show that a blowup solution of degree one is essentially a decomposition of the form
$$\vec u(t) = \vec h + \left(Q\left(\frac{\cdot}{\lambda(t)}\right), 0\right)  + \vec \epsilon(t), \quad \lambda(t) = o(T-t),$$
where $\vec h$ and $\vec \epsilon$ are of topological degree zero, $\Ec(\vec h)$ is less than $2\Ec(Q,0)$ and $\Ec(\vec \epsilon)(t)$ goes to zero as $t \to T$.  This result reveals the universal character of the known blowup constructions for degree one of \cite{KSTim08} and \cite{RRmihes12}.\\

In the supercritical energy case $d \geq 3$, we have the following explicit solution of \eqref{eq:phi}
\begin{equation}\label{sol:phi}
\varphi_0(y) = 2 \arctan \left(\frac{y}{\sqrt{d - 2}}\right).
\end{equation}
This self-similar solution was found by Turok and Spergel \cite{TSprl90} for $d = 3$ (see also Shatah \cite{Scpam88} for an earlier result) and by Bizon and Biernat \cite{BBcmp15} for $d \geq 4$. For $d = 3$, the solution \eqref{sol:phi} is proved to be stable by Donninger \cite{Dcpam11}, Donninger, Sch\"orkhuber and Aichelburg \cite{DSAihp12}, Costin, Donninger and Xia \cite{CDXnon16}. This stability is recently proved for all odd dimensions by Chatzikaleas, Donninger and Glogic \cite{CDGar17}. This selfsimilar solution is expected to be generic through numerical simulations in \cite{BCTnon00} and \cite{BBcmp15}. When $3 \leq d \leq 6$, we note that there exists an infinite sequence of globally regular solutions $\varphi_n$ for \eqref{eq:phi} (see \cite{BBMnon17}) where the index $n$ denotes the number of zeros of $\varphi_n'$ in $(0,1)$.  

When $d \geq 7$, Biernat \cite{BIEnon2015} shows the existence of a stationary solution $Q$ for equation \eqref{Pb}, namely that $Q$ solves
\begin{equation}\label{eq:Qr}
Q'' + \frac{(d-1)}{r}Q' - \frac{(d-1)}{2r^2}\sin(2Q) = 0, \quad Q(0) = 0, \; Q'(0) = 1,
\end{equation} 
The solution $Q$ is unique (up to scaling) and admits the behavior for $r$ large,
\begin{equation}\label{exp:Qr}
Q(r) = \dfrac{\pi}{2} - \dfrac{a_0}{ r^{\gamma}} + o\left(\dfrac 1{r^{\gamma}}\right),
\end{equation}
for some $a_0 = a_0(d) > 0$ and and $\gamma = \gamma(d)$ is given by
\begin{equation}\label{def:gamome}
\gamma(d) = \frac{1}{2}(d - 2 - \tilde{\gamma}) \in (1,2] \quad \text{for}\;\; d \geq 7,
\end{equation}
where 
$$\tilde{\gamma}=\sqrt{d^2-8d+8}.$$
It happens that the asymptotic behavior of the stationary solution $Q$ given by \eqref{exp:Qr} plays an important role in the construction of Type II blowup solutions for an analogous problem for the heat flow
\begin{equation}\label{Plheat}
\pt u = \partial^2_r u + \frac{(d-1)}{r}\partial_r u - \frac{(d-1)}{2r^2}\sin(2u).
\end{equation}
In \cite{GINapde18}, we construct for equation \eqref{Plheat} a family of $\mathcal{C}^\infty$ solutions which blow up in finite time via concentration of the profile
$$u(r,t) \sim Q\left(\frac{r}{\lambda(t)}\right),$$
where $\lambda$ is given by the quantized rates 
$$\lambda(t) \sim (T-t)^\frac{\ell}{\gamma} \quad \text{as}\;\; t \to T,$$
for $\ell \in \mathbb{N}^*$ satisfying $2\ell > \gamma$. Note that the same blowup rate was obtained by Biernat and Seki \cite{BSarxiv2016} through a matched asymptotic method. More precisely, we have successfully adapted the strategy developed by Merle, Rapha\"el and Rodnianski \cite{MRRcjm15} for the study of the energy supercritical nonlinear Schr\"odinger equation to construct for equation \eqref{Plheat} type II blowup solutions. The method relies on a two step procedure:
\begin{itemize}
\item Construction of a suitable approximate blowup profile through iterated resolutions of elliptic equations. The \textit{tail computation} allows us to formally derive the blowup speed. 

\item Implementation of a robust universal energy method to control the solution in the blowup regime through the derivation of suitable Lyapunov functional, which relies on neither spectral estimates nor the maximum principle and may be easily applied to various settings.
\end{itemize}
The method of \cite{MRRcjm15} has been also proved to be success for the construction of type II blowup solutions for the energy supercritical semilinear heat and wave equations by Collot \cite{Car16, Car161}. 

\bigskip

In this paper, by considering the case when 
$$d \geq 7,$$
we ask whether we can carry out the analysis in \cite{GINapde18} to construct solutions for equation \eqref{Pb} which blow up in finite time via concentration of the profile $Q$. The following theorem is our main result. 
\begin{theorem}[Existence of type II blowup solutions to \eqref{Pb} with prescibed behavior] \label{Theo:1} Let $d \geq 7$ and $\gamma$ be defined as in \eqref{def:gamome}, we fix an integer 
$$\ell \in \mathbb{N}^* \quad  \text{with}\quad \ell > \gamma,$$
and two numbers $\sigma \in \Rb_+, \frak s \in \mathbb{N}$ such that
$$0 < \sigma - \frac{d}{2} \ll 1 \quad \text{and} \quad 1 \ll \frak{s} = \frak{s}(\ell) \to +\infty \quad \text{as}\;\; \ell \to +\infty.$$
Then there exists an open set of initial data of the form 
$$(u_0, u_1) = (Q,0) + (\varepsilon_0, \varepsilon_1), \quad (\varepsilon_0, \varepsilon_1) \in \mathcal{O}  \subset \left(\dot{H}^\sigma \cap \dot{H}^\frak s\right) \times \left(\dot{H}^{\sigma-1} \cap \dot{H}^{\frak s-1}\right), $$
such that the corresponding solution to equation \eqref{Pb} satisfies
\begin{equation}\label{eq:uQq}
u(r,t) = Q\left(\frac{r}{\lambda(t)}\right) + \varepsilon\left(\frac{r}{\lambda(t)}, t\right)
\end{equation}
where 
\begin{equation}\label{eq:quanblrate}
\lambda(t) = c(u_0, u_1)(T-t)^\frac{\ell}{\gamma} (1 + o_{t \to T}(1)), \quad c(u_0,u_1) > 0,
\end{equation}
and 
\begin{equation}\label{eq:asypqs}
\lim_{t \to T}\|(\varepsilon(t), \lambda \pt \varepsilon(t))\|_{\dot H^\mu \times \dot H^{\mu-1}} = 0, \quad \forall \mu \in \left[\sigma, \frak{s}\right].
\end{equation}
\end{theorem}

\begin{remark} Since  $\gamma \in (1, 2)$ for $d \geq 8$ and $\gamma = 2$ for $d = 7$, the condition $\ell > \gamma$ requests that $\ell \geq 2$ for $d \geq 8$ and $\ell \geq 3$ for $d = 7$. As for the case $\ell = \gamma$, which only happens in the case $d = 7$ with $\ell = \gamma = 2$, we expect that the blowup rate  \eqref{eq:quanblrate} would involve some logarithmic correction of the form
$$\lambda(t) \sim \frac{T-t}{|\log(T-t)|^\nu} \quad \text{for some} \quad \nu > 0.$$ 
This logarithmic gain would be related to the growth of the approximate profile at infinity. Although our analysis would be naturally extended  to this case, this seems to require some crucial modification in the construction of an approximate profile and this would be treated in a separate work. 
\end{remark}

\begin{remark} The proof of Theorem \ref{Theo:1} involves a detailed description of the the set of initial data leading to the type II blowup with the quantization of the blowup rate \eqref{eq:quanblrate}. In particular, given $\ell \in \mathbb{N}^*$, $L \gg 1$ and $\frak{s} \sim L$,  our initial data is of the form 
\begin{equation}\label{eq:u0des}
\vec u_0 = \vec Q_{b(0)} + \vec q_0,
\end{equation}
where $\vec Q_{b}$ is a deformation of the ground state $\vec Q = (Q,0)$, and $b = (b_1, \cdots, b_L)$ correspond to possible unstable directions of the flow in the $\dot{H}^\frak{s}\times \dot{H}^{\frak{s} - 1}$ topology in a suitable neighborhood of $\vec Q$. We show that for all $\vec q_0 \in \mathcal{O} \subset \Big(\dot{H}^\sigma \cap \dot{H}^\frak{s}\Big) \times \Big(\dot{H}^{\sigma - 1} \cap \dot{H}^{\frak{s} - 1}\Big)$, where the set $\mathcal{O}$ is built on the linearized operator (see Definition \ref{def:1} for its precise description of $\mathcal{O}$) and for all $\big(b_1(0), b_{\ell + 1}(0), \cdots, b_L(0)\big)$ small enough, there exists a choice of unstable directions $\big(b_2(0), \cdots, b_\ell(0)\big)$ such that the solution of \eqref{Pb} with initial data \eqref{eq:u0des} satisfies the conclusion of Theorem \ref{Theo:1}. The control of $(\ell - 1)$ unstable modes is done through a topological argument based on Brouwer's fixed point theorem. In some sense, the set of blowup solutions we construct lies on a $(\ell - 1)$ codimension manifold in the radial class whose proof would require some Lipschitz regularity of the set of initial data we consider and it would be addressed separately in detail.
\end{remark}

\begin{remark}
It is worth mentioning that our analysis relies only on the study of supercritical Sobolev norms built on the linearized operator, thus, the finiteness of the $H^1$ norm of the initial data is not requested.
Roughly speaking, the initial data $(u_0,u_1)$ can be taken smooth and compactly supported, namely that
if $u=Q+\varepsilon$, we take $\varepsilon(r) \sim -Q(r)$ for $r \gg 1$. Since the energy is conserved, our constructed solution can be taken to be of finite energy or even compactly supported. As a matter of fact, the finite energy together with the constructed manifold mentioned in the previous remark ensures that the original solution $\Phi$ to the wave map equation \eqref{eq:WM} has the same the regularity as for the 1-corotational symmetric solution $u$ described in Theorem \ref{Theo:1}.
\end{remark}

\begin{remark} We note from \eqref{eq:uQq} that 
$$\partial_r u(0,t) \sim \lambda^{-1}(t) \sim (T-t)^{-\frac{\ell}{\gamma}} \gg (T-t)^{-1} \quad \textup{as} \;\; t \to T.$$ 
This implies that our constructed solution is of Type II blowup in the sense of \eqref{def:Typ1blowup}.

\end{remark}

\begin{remark} Following the work by C\^ote \textit{et al.} \cite{CKLSajm15I, CKLSajm15II} where the question of the classification of the flow near the special class of stationary solution $Q$ are considered in the energy critical setting, i.e. $d  = 2$, we would address the same question for the energy supercritical case $d \geq 7$. In Theorem \ref{Theo:1}, the constructed blowup solutions exhibit the decomposition of the form \eqref{eq:uQq}. Here we ask for a converse problem, namely that if blowup does occur for a solution $\vec u$, in which energy regime and in what sense does such the decomposition \eqref{eq:uQq} always hold?
\end{remark}

\begin{remark} It is worth mentioning the work of Krieger-Schlag-Tataru \cite{KSTim08}, where the authors constructed for equation \eqref{Pb} in the critical case $d = 2$ blowup solutions of the form 
$$u(r,t) = Q(r\lambda(t)) + \varepsilon(r,t), \quad r \leq t,$$ 
where $\varepsilon$ has local energy going to zero as $t \to 0$ and $\lambda(t) = t^{-1-\nu}$ with $\nu > \frac{1}{2}$ arbitrary. Analogous results are also established in \cite{KSTam09,KSTduke09} (see also \cite{DHKSmmj14}) for the critical semilinear wave equation and the critical Yang-Mills problem.  The existence of the continuum of blowup rates established in \cite{KSTim08, KSTam09, KSTduke09} is an interesting phenomena and it is different from our result where the blowup rate \eqref{eq:quanblrate} is discretely quantized. The discrete quantization of blowup rates has been previously derived in  \cite{RRmihes12},  \cite{RSapde2014}, \cite{MRRcjm15}, \cite{GINapde18}, \cite{Car161}, \cite{Car16}, ..., where the constructions of blowup solutions are based on the modulation theoretic approach. We suspect that such an existence of a continuum of blowup rates only happens in hyperbolic problems. A more evidence is due to the work by Collot-Ghoul-Masmoudi \cite{CGMarx18} for the Burger's equation with a transverse viscosity, where the authors observe that there also exist blowup solutions with a continuum of blowup rates if one does not impose smoothness on the solution before the blowup time. An interesting question after our work is that whether there exist blowup solutions to equation \eqref{Pb} in the case $d \geq 7$ with a continuum of blowup rates?
\end{remark}

\bigskip

Let us briefly explain the main steps of the proof of Theorem \ref{Theo:1}, which follows the strategy developed  in \cite{MRRcjm15} for the energy supercritical nonlinear Schr\"odinger equation. We would like to mention that this kind of method has been successfully applied for various nonlinear evolution equations. In particular in the dispersive setting for the nonlinear Schr\"odinger equation both in the mass critical \cite{MRgfa03,MRim04, MRam05, MRcmp05} and mass supercritical \cite{MRRcjm15} cases; the mass critical gKdV equation \cite{MMRam14, MMRasp15, MMRjems15};  the energy critical \cite{DKMcjm13}, \cite{HRapde12} and supercritical \cite{Car161} wave equation; the two dimensional critical geometric equations: the wave maps \cite{RRmihes12}, the Schr\"odinger maps \cite{MRRim13} and the harmonic heat flow \cite{RScpam13, RSapde2014} and \cite{GINapde18}; the semilinear heat equation in the energy critical \cite{Sjfa12} and supercritical \cite{Car16} cases; and the two dimensional Keller-Segel model \cite{RSma14}, \cite{GMarx16}. In all those works, the method relies on two arguments:
\begin{itemize}
\item Reduction of an infinite dimensional problem to a finite dimensional one, through the derivation of suitable Lyapunov functional and the robust energy method as mentioned in the two step procedure above.
\item The control of the finite dimensional problem thanks to a topological argument based on index theory.
\end{itemize}
Note that this kind of topological arguments has proved to be successful also for the construction of type I blowup solutions for the semilinear heat equation in \cite{BKnon94}, \cite{MZdm97}, \cite{NZens16} (see also \cite{NZsns16}, \cite{DNZtjm18} for the case of logarithmic perturbations, \cite{Breiumj90}, \cite{Brejde92} and \cite{GNZjde17} for the exponential source, \cite{NZcpde15} for the complex-valued case), the Ginzburg-Landau equation in \cite{MZjfa08} (see also \cite{ZAAihn98} for an earlier work), a non-variational parabolic system in \cite{GNZihp18, GNZjde18} and the semilinear wave equation in \cite{CZcpam13}.\\

For the reader's convenience and for a better explanation, let's first introduce notations used throughout this paper.\\
\noindent \textbf{- Notation.} The equation \eqref{Pb} can be put in the following first-order form:
\begin{equation}\label{eq:uvec}
\pt \vec{u} = \vec{F}(\vec u), \quad \vec u(t): \Rb^d \to \Rb \times \Rb,
\end{equation}
where we denote by
$$\vec{u} = \binom{u_1}{u_2}, \quad \vec F(\vec u) = \binom{u_2}{\pr^2u_1 + \frac{d-1}{r}\pr u_1 - \frac{d-1}{2r^2}\sin(2u_1)}.$$
In what follows the notation $\vec u$ always refers to a vector whose coordinates are $\binom{u_1}{u_2}$. The stationary solution of \eqref{eq:uvec} is denoted by 
$$\vec{Q} = \binom{Q}{0},$$
where $Q$ is introduced in \eqref{eq:Qr} and \eqref{exp:Qr}. \\
We denote by
$$\big \langle u,v \big \rangle = \int_{\Rb^d}uv \quad \text{and} \quad \big \langle \vec u, \vec v \big \rangle = \int_{\Rb^d} \vec u . \vec v =  \int_{\Rb^d}u_1 v_1 + \int_{\Rb^d}u_2 v_2.$$ 
For each $d \geq 7$, we define
\begin{equation}\label{def:kdeltaplus}
\left\{\begin{array}{ll}
\hbar &= \left\lfloor\frac{d}{2} - \gamma\right\rfloor \in \mathbb{N}^*,\\
& \\
\delta &=\left(\frac{d}{2} - \gamma\right) - \hbar, \quad  \delta \in (0,1),
\end{array} \right.
\end{equation}
where $\lfloor x \rfloor \in \mathbb{Z}$ stands for the integer part of $x$ which is defined by $\lfloor x \rfloor \leq x < \lfloor x \rfloor + 1$.\footnote{Note that $\delta \ne 0$. Indeed, if $\delta = 0$, then there is $m \in \mathbb{N}$ such that $2\gamma = d - 2m \in \mathbb{N}$. This only happens when $\gamma = 2$ or $\gamma = \frac{3}{2}$ because $\gamma \in (1, 2]$. The case $\gamma = 2$ gives $d = 7$ and $m = \frac{3}{2} \not \in \mathbb{N}$. The case $\gamma = \frac{3}{2}$ gives $d = \frac{17}{2} \not \in \mathbb{N}$.}\\
For each $k \in \mathbb{N}$, we denote by 
$$k\wedge 2 := k \mod 2.$$
Given a large odd integer $L \gg 1$, we set
\begin{equation}\label{def:kbb}
\Bbbk = L + \hbar + 1.
\end{equation}
We fix $\sigma \in \Rb_+$ such that
\begin{equation}\label{def:sigma}
\sigma > \frac{d}{2} \quad \text{and}\quad \left|\sigma - \frac{d}{2} \right| \leq \frac{1}{L^2} \ll 1.
\end{equation}
Given $b_1 > 0$ and $\lambda > 0$, we define 
\begin{equation}\label{def:B0B1}
B_0 = \frac{1}{ b_1}, \quad B_1 = B_0^{1 + \eta}, \quad 0 < \eta \leq \frac{1}{L^2} \ll 1,
\end{equation}
and denote by
$$f_\lambda(r) = f(y) \quad \text{with} \quad y = \frac{r}{\lambda}.$$
Let $\chi \in \Cc_0^\infty([0, +\infty))$ be a positive non increasing cutoff function with $\text{supp}(\chi) \subset [0,2]$ and $\chi \equiv 1$ on $[0,1]$. For all $M > 0$, we define
\begin{equation}\label{def:chiM}
\chi_M(y) = \chi\left(\frac y M\right).
\end{equation}
We introduce the first order differential operators 
$$\Lambda f = y\partial_y f, \quad Df = f + y\py f, \quad \Lambda \vec f = \binom{\Lambda f_1}{Df_2}.$$
The linearized operator near the stationary solution $\vec Q$ is then defined by 
\begin{equation}\label{def:Hop}
\Hs = \begin{bmatrix}
0&-1\\\Ls & 0
\end{bmatrix},
\end{equation}
so that 
$$\vec F(\vec Q + \vec q) = - \Hs \vec q + \vec N(\vec q),$$
where 
\begin{equation}\label{def:Lc}
\Ls  = -\partial_{yy} - \frac{(d-1)}{y}\partial_y  + \frac{Z}{y^2}, \quad \text{with}\;\; Z(y)= (d-1)\cos(2Q(y)),
\end{equation}
and $\vec N$ is the purely nonlinear term
\begin{equation}\label{def:Nvec}
\vec{N}(\vec q)= \binom{0}{\frac{(d-1)}{2y^2}\left[\sin(2Q + 2 q_1) - \sin(2Q) - 2\cos(2Q)q_1\right]} = \binom{0}{N(q_1)}.
\end{equation}
We denote by $\Hs^*$ the adjoint of $\Hs$, 
$$\Hs^* = \begin{bmatrix}
 0 & \Ls \\ -1 & 0
\end{bmatrix} \quad \text{satisfying} \quad \big \langle \Hs \vec u, \vec v \big \rangle = \big \langle \vec u, \Hs^* \vec v \big \rangle.$$
We let the matrix 
\begin{equation}
J = \begin{bmatrix}
0 & -1 \\ 1 &0
\end{bmatrix}, 
\end{equation}
and define the adapted norm for $k \in \mathbb{N}^*$,
\begin{equation}\label{def:normk}
\|\vec u\|_{k}^2 = \int_{\Rb^d} u_1 \Ls^{k}u_1 + \int_{\Rb^d} u_2 \Ls^{k - 1}u_2.
\end{equation}
Note that the norm defined by \eqref{def:normk} is actually positive thanks to the factorization of $\Ls$ (see Lemma \ref{lemm:factorL} below), 
$$\Ls = \As^* \As.$$
For $k \in \mathbb{N}$, we define the suitable derivative for any smooth function $f$:
\begin{equation}\label{def:adapder}
f_{2k} = \Ls^kf, \quad f_{2k + 1} = \As \Ls^k f, \quad f_0 = f.
\end{equation}

\medskip

\noindent \textbf{- Strategy of the proof.}  We now summary the main ideas of the proof of Theorem \ref{Theo:1}, which follows the road map in \cite{GINapde18} and \cite{MRRcjm15}.\\

\noindent $(i)$ \underline{\textit{Renormalized flow.}} Following the scaling invariance of \eqref{Pb}, let us make the change of variables
\begin{equation*}
\vec w(y,s) := \binom{w_1}{w_2}(y,s) = \binom{u_1}{\lambda u_2}(r,t), \quad y = \frac{r}{\lambda(t)}, \quad \frac{ds}{dt} = \frac{1}{\lambda(t)},
\end{equation*}
which leads to the following renormalized flow:
\begin{equation}\label{eq:wys_i}
\ps \vec w + b_1 \Lambda \vec w= \vec F (\vec w),  \quad \text{with} \quad  b_1 = -\frac{\lambda_s}{\lambda}.
\end{equation}
As we will show later that $b_1 \to 0$ as $s \to +\infty$, the leading part of the solution $\vec w(y,s)$ is given by the ground state profile $\vec Q(y)$. That is why, we introduce 
$$\vec q(y,s) = \vec w(y,s) - \vec Q(y),$$
then $\vec q$ solves 
\begin{equation}\label{eq:qys_i}
\ps \vec{q} + \Hs \vec q  + b_1 \Lambda \vec q = -b_1 \Lambda \vec Q + \vec{N}(\vec q),
\end{equation}
where the nonlinear term is given by \eqref{def:Nq}.

\noindent $(ii)$ \underline{\textit{Properties of the linearized operators $\Ls$ and $\Hs$.}} The linear operator $\Ls$ admits the following factorization (see Lemma \ref{lemm:factorL} below)
\begin{equation}\label{eq:facL_i}
\Ls = \As^*\As, \quad \As f = -\Lambda Q \py \left(\frac{f}{\Lambda Q} \right), \quad \As^*f = \frac{1}{y^{d-1}\Lambda Q}\py\left(y^{d-1}\Lambda Q f\right), 
\end{equation}
which simplifies the computation of $\Ls^{-1}$ (see Lemma \ref{lemm:inversionL} below). The factorization  \eqref{eq:facL_i} immediately follows
\begin{equation}\label{eq:LlaQ}
\Ls(\Lambda Q) = 0.
\end{equation}
Note from \eqref{exp:Qr} that
$$\Lambda Q \sim \frac{c_0}{y^\gamma} \quad \text{as} \quad y \to +\infty,$$
with $\gamma$ defined in \eqref{def:gamome}. We can compute the kernel of $\Ls^k$ through the iterative scheme
\begin{equation}\label{def:phik}
\Ls \phi_{k+1} = -\phi_k, \quad \phi_0 = \Lambda Q,
\end{equation}
which displays a non trivial tail at infinity (see Lemma 2.9 in \cite{GINapde18})
\begin{equation}\label{asy:phik}
\phi_k(y) \sim c_k y^{2k - \gamma} \quad \text{for} \quad y \gg 1.
\end{equation}
The identity \eqref{eq:LlaQ} also yields
$$\Hs (\Lambda \vec Q) = \vec 0.$$
Furthermore, knowing $\Ls^{-1}$ we can define the inversion of $\Hs$ as follows
\begin{equation}
\Hs^{-1} = \begin{bmatrix}
0 & \Ls^{-1}\\ -1 &0
\end{bmatrix}.
\end{equation}
More generally, the kernel of $\Hs^k$ is computed  by
\begin{equation}\label{def:Tk_i}
\Hs \vec T_{k + 1} = - \vec T_k \quad \text{with} \quad \vec T_0 = \Lambda \vec Q = \binom{\phi_0}{0}.
\end{equation}
In particular, we have
\begin{equation}\label{eq:formTk_i}
\vec T_{2k} = \binom{\phi_k}{0}, \quad \vec T_{2k + 1} = \binom{0}{\phi_k}.
\end{equation}

\noindent $(iii)$ \underline{\textit{Tail dynamics.}} Following the approach in \cite{GINapde18} and \cite{MRRcjm15}, we look for a slowly modulated approximate solution to \eqref{eq:wys_i} of the form 
$$\vec w(y,s) = \vec Q_{b(s)}(y),$$
where 
\begin{equation}\label{def:Qb_i}
b = (b_1, \cdots, b_L), \quad \vec Q_{b(s)}(y) = \vec Q(y) + \sum_{k = 1}^Lb_k\vec T_k(y) + \sum_{k = 2}^{L+2}\vec S_k(y,b)
\end{equation}
with a priori bounds
$$b_k \sim b_1^k, \quad |\vec S_k(y,b)| \lesssim b_1^k y^{k - 2 -  k\wedge 2 - \gamma},$$
so that $\vec S_k$ is in some sense homogeneous of degree $k$ in $b_1$, and behaves better than $\vec T_k$ at infinity. The construction of $\vec S_k$ with the above a priori bounds is possible for a specific choice of the universal dynamical system which drives the modes $(b_k)_{1 \leq k \leq L}$. This is so called the \textit{tail computation}. Let us illustrate the procedure of the \textit{tail computation}. We plug the decomposition \eqref{def:Qb_i} into \eqref{eq:wys_i} and choose the law for $(b_k)_{1 \leq k \leq L}$ which cancels the  leading order terms at infinity.\\
- At the order $\Oc(b_1)$: we cannot adjust the law of $b_1$ for the first term \footnote{if $(b_1)_s = -c_1 b_1$, then $-\lambda_s/\lambda \sim b_1 \sim e^{-c_1 s}$, hence after an integration in time, $|\log \lambda| \lesssim 1$ and there is no blowup.} and obtain from \eqref{eq:qys_i}, 
$$b_1(\Hs \vec T_1 + \Lambda \vec Q) = 0.$$
- At the order $\Oc(b_1^{2k}, b_{2k})$, $k = 1, \cdots, (L+1)/2$: We obtain
$$(b_{2k - 1})_s\vec T_{2k - 1} + b_1 b_{2k - 1}\Lambda \vec T_{2k - 1} + b_{2k}\Hs \vec T_{2k} + \Hs \vec S_{2k} = b_1^{2k} \vec N_{2k - 1}(\vec Q, \vec T_1, \cdots, \vec T_{2k - 1}),$$
where $\vec N_{2k - 1}$ corresponds to nonlinear interaction terms. Note from \eqref{eq:formTk_i}, \eqref{asy:phik} and \eqref{def:Tk_i}, we have
$$\Lambda \vec T_{2k - 1} \sim (2k - 1 - \gamma)\vec T_{2k - 1} \quad \text{for}\quad y \gg 1, \quad \Hs \vec T_{2k} = - \vec T_{2k - 1},$$
and thus, 
$$(b_{2k - 1})_s \vec T_{2k - 1} + b_1 b_{2k - 1}\Lambda \vec T_{2k -1} + b_{2k}\Hs \vec T_{2k} \sim \big[(b_{2k - 1})_s + (2k - 1 - \gamma)b_1 b_{2k - 1} - b_{2k}\big]\vec T_{2k - 1}.$$
Hence the leading order growth for $y$ large is canceled by the choice 
$$(b_{2k-1})_s + (2k - 1 - \gamma)b_1 b_{2k - 1} - b_{2k} = 0.$$
We then solve for 
$$\Hs \vec S_{2k} = -b_1^{2k}(\Lambda \vec T_{2k - 1} - (2k - 1 - \gamma)\vec T_{2k - 1}) + b_1^{2k} \vec N_{2k - 1}(\vec Q, \vec T_1, \cdots, \vec T_{2k - 1}),$$
and check the improved decay 
$$|S_{2k}(y,b)| \lesssim b_1^{2k}y^{2k - 2 - \gamma} \quad \text{for} \quad y \gg 1.$$

\noindent - At the order $\Oc(b_1^{2k+1}, b_{2k+1})$, $k = 1, \cdots, (L+1)/2$: we obtain an elliptic equation of the form
$$(b_{2k})_s\vec T_{2k} + b_1b_{2k}\Lambda \vec T_{2k} + b_{2k+1}\Hs \vec T_{2k+1} + \Hs \vec S_{2k+1} = b_1^{2k + 1}\vec N_{2k}(\vec Q, \vec T_1, \cdots, \vec T_{2k}).$$ 
From \eqref{eq:formTk_i}, \eqref{asy:phik} and \eqref{def:Tk_i}, we have
$$(b_{2k})_s\vec T_{2k} + b_1b_{2k}\Lambda \vec T_{2k} + b_{2k + 1}\Hs \vec T_{2k + 1} \sim \big[(b_{2k})_s + (2k - \gamma)b_1b_{2k} - b_{2k+1}\big]\vec T_{2k},$$
which leads to the choice 
$$(b_{2k})_s + (2k - \gamma)b_1b_{2k} - b_{2k+1} = 0,$$
for the cancellation of the leading order growth at infinity. We then solve for the remaining $\vec S_{2k + 1}$ term and check that $|\vec S_{2k + 1}(y)|\lesssim b_1^{2k+1}y^{2k - 2 - \gamma}$ for $y$ large. We refer to Proposition \ref{prop:1} for all details of the \textit{tail computation}. Note that for $k$ large enough, the profile $\vec T_{k}$ and $\vec S_k$ have irrelevant growth at infinity. For this reason we cut $\vec T_{k}$ and $\vec S_k$ in the zone $y \sim B_1$ in order to obtain a suitable approximate profile, namely that the approximation \eqref{def:Qb_i} is replaced by 
\begin{equation*}
\vec Q_{b(s)}(y) = \vec Q(y) + \chi_{_{B_1}}\left(\sum_{k = 1}^Lb_k\vec T_k(y) + \sum_{k = 2}^{L+2}\vec S_k(y,b)\right).
\end{equation*}
All the computation is then done in the zone  $y \sim B_1$ or in the original variable $r \sim \lambda B_1 \sim (T-t)^{1 - \eta \left(\frac{\ell}{\gamma} - 1\right)}$, which is slightly beyond the light cone.

\noindent $(iii)$ \textit{The universal system of ODEs.} The above procedure leads to the following universal system of ODEs after $L$ iterations,
\begin{equation}\label{sys:bk_i}
\left\{ \begin{array}{l}
(b_k)_s + (k - \gamma)b_1b_k - b_{k+1} = 0, \quad 1 \leq k \leq L, \quad b_{L+1} = 0,\\
\quad \\
-\dfrac{\lambda_s}{\lambda} = b_1, \quad \dfrac{ds}{dt} = \dfrac{1}{\lambda}.
\end{array}\right.
\end{equation}
The set of solutions to \eqref{sys:bk_i} (see Lemma \ref{lemm:solSysb} below) is explicitly given by 
\begin{equation}\label{eq:solbk_i}
\left\{\begin{array}{l}
b_k^e(s) = \frac{c_k}{s^k}, \quad 1 \leq k \leq L,\\
c_1 = \frac{\ell}{\ell - \gamma}, \quad \ell \in \mathbb{N}, \; \ell > \gamma,\\
c_{k + 1} = -\frac{\gamma(\ell - k)}{\ell - \gamma}c_k, \quad 1 \leq k \leq \ell -1,\ell\geq 2\\
c_j = 0, \quad j \geq \ell + 1.\\
\lambda(s) \sim s^{-\frac{\ell}{\ell - \gamma}}.
\end{array}\right.
\end{equation}
In the original time variable $t$, this implies that $\lambda(t)$ goes to zero in finite time $T$ with the asymptotic 
$$\lambda(t) \sim (T-t)^\frac{\ell}{\gamma}.$$
Moreover, the linearized flow of \eqref{sys:bk_i} near the solution \eqref{eq:solbk_i} is explicit and displays $\ell - 1$ unstable directions (see Lemma \ref{lemm:lisysb} below).\\

\noindent $(iv)$ \textit{Decomposition of the flow and modulation equations.} Let the approximate solution $Q_b$ be given by \eqref{def:Qb_i} which by construction generates an approximate solution to the renormalized flow \eqref{eq:wys_i},
$$\vec \Psi_b = \ps \vec Q_b  + b_1 \Lambda \vec Q_b - \vec F(\vec Q_b) = \vec{\textup{Mod}} + \Oc(b_1^{2L + 2}),$$
where the modulation equation term is roughly of the form
$$\vec{\textup{Mod}} = \sum_{k = 1}^L \big[(b_k)_s + (k - \gamma)b_1b_k - b_{k+1}\big]\vec T_k.$$
We localize $\vec Q_b$ in the zone $y \leq B_1$ to avoid the irrelevant growing tails for $y \gg \frac{1}{b_1}$. We then take initial data of the form
$$\vec u_0(y) = \vec Q_{b(0)}(y) + \vec q_0(y),$$
where $\vec q_0$ is small in some suitable sense and $b(0)$  is chosen to be close to the exact solution \eqref{eq:solbk_i}. By a standard modulation argument, we introduce the decomposition of the flow
\begin{equation}\label{eq:dec_i}
\vec w(y,s) = \big(\vec Q_{b(s)} + \vec q\big)(y,s),
\end{equation}
where $L+1$ modulation parameters $(b(t), \lambda(t))$ are chosen in order to manufacture the orthogonality conditions:
\begin{equation}\label{eq:orh_i}
\left<\Hs^k \vec q, \vec \Phi_M \right> = 0, \quad 0 \leq k \leq L,
\end{equation}
where $\vec \Phi_M$ (see \eqref{def:PhiM}) is some fixed direction depending on some large constant $M$, generating an approximation of the kernel of the powers of $\Hs$. This orthogonal decomposition \eqref{eq:dec_i}, which follows from the implicit function theorem, allows us to compute the modulation equations governing the parameters $(b(t), \lambda(t))$ (see Lemmas \ref{lemm:mod1} and \ref{lemm:mod2} below),
\begin{equation}\label{eq:mod_i}
\left|\frac{\lambda_s}{\lambda} + b_1\right| +  \sum_{k = 1}^L \big|(b_k)_s + (k - \gamma)b_1b_k - b_{k+1}\big| \lesssim \|\vec q\|_{\textup{loc}} + b_1^{L + 1 + \nu(\delta, \eta)},
\end{equation}
where $\|\vec q\|_{\textup{loc}}$ measures a spatially localized norm of the radiation $\vec q$ and $\nu(\delta, \eta) > 0$.\\

\noindent $(v)$ \textit{Control of Sobolev norms.} According to \eqref{eq:mod_i}, we need to show that local norms of $\vec{q}$ are under control and do not perturb the dynamical system \eqref{sys:bk_i}. This is achieved via high order mixed energy estimates which provide controls of the Sobolev norms adapted to the linear flow and based on the powers of the linear operator $\Hs$. In particular, we have the following coercivity of the high energy under the orthogonality conditions \eqref{eq:orh_i} (see Lemma \ref{lemm:coerEk}),
$$\Es_{\Bbbk}(s) := \|\vec q(s)\|_{\Bbbk}^2  \gtrsim \|\vec q(s) \|^2_{\dot H^\Bbbk \times \dot H^{\Bbbk - 1}},$$
where $\Bbbk$ is given by \eqref{def:kbb} and the norm is defined by \eqref{def:normk}. The energy estimate is of the form 
\begin{equation}\label{eq:Ek_i}
\frac{d}{ds} \left\{\frac{\Es_{\Bbbk} + b_1\Mc }{\lambda^{2\Bbbk - d}}\right\} \lesssim \frac{b_1^{2L + 1 + 2\nu(\delta, \eta)}}{\lambda^{2\Bbbk - d}} \quad \text{for some}\quad \nu(\delta, \eta) > 0,
\end{equation}
where the right hand side is the size of the error $\vec \Psi_b$ in the construction of the approximate profile $\vec Q_b$ above, and $\Mc$ corresponds to an additional Morawetz type term (see \eqref{def:Mofun} for a precise definition of $\Mc$) which is needed to control $\Es_{\Bbbk}$ locally (see Proposition \ref{prop:Mcon}). Note that the successful key in deriving such a Morawetz type control is due to the fact that the linear operator $\Ls$ is positive in $\dot{H}^1$ for $d \geq 7$.   An integration of \eqref{eq:Ek_i} in time by using initial smallness assumptions, $b_1 \sim b_1^e$ and $\lambda(s) \sim b_1^{\frac{\ell}{\ell - \gamma}}$ yields the estimate
$$
\|\vec q \|^2_{\dot H^\Bbbk \times \dot H^{\Bbbk - 1}} \lesssim \Es_{\Bbbk}(s) \lesssim b_1^{2L + 2\nu(\delta, \eta)},$$
which is good enough to control the local norms of $\vec q$ and close the modulation equations \eqref{eq:mod_i}.

Note that when establishing the formula \eqref{eq:Ek_i}, we need to deal with a nonlinear term which is roughly of the form $\frac{q_1^2}{y^2}$. In order to archive the control of this term, we derive the following mononicity formula for the low Sobolev norm
$$\Es_\sigma = \|\vec q\|_{\dot{H}^\sigma \times \dot{H}^{\sigma - 1}}, \quad \frac{d}{ds}\left\{\frac{\Es_\sigma}{\lambda^{2\sigma - d}} \right\} \lesssim \frac{b_1^{1 + \frac{\ell}{\ell -\gamma}\left(2\sigma - d \right) + \epsilon}}{\lambda^{2\sigma - d}} \quad \text{for some} \quad \epsilon > 0.$$
Integrating in time yields the bound 
$$\Es_\sigma(s) \lesssim b_1^{\frac{\ell}{\ell - \gamma}(2\sigma - d)}.$$
which is enough to close the estimate for the nonlinear term.

The above scheme designs a bootstrap regime (see Definition \ref{def:Skset} for a precise definition) which traps blowup solution with speed \eqref{eq:quanblrate}. According to Lemma \ref{lemm:solSysb} and \ref{lemm:lisysb}, such a regime displays $(\ell - 1)$ unstable modes $(b_2, \cdots, b_\ell)$ which we can control through a topological argument based on the Brouwer fixed point theorem (see the proof of Proposition \ref{prop:exist}), and the proof of Theorem \ref{Theo:1} follows.\\

\bigskip 

The paper is organized as follows. In Section \ref{sec:2}, we give the construction of the approximate solution $\vec Q_b$ of \eqref{Pb} and derive estimates on the generated error term $\vec \Psi_b$ (Proposition \ref{prop:1}) as well as its localization (Proposition \ref{prop:localProfile}). We also give in this section some elementary facts on the study of the system \eqref{sys:bk_i} (Lemmas \ref{lemm:solSysb} and \ref{lemm:lisysb}). Section \ref{sec:3} is devoted to the proof of Theorem \ref{Theo:1} assuming a main technical result (Proposition \ref{prop:redu}). In particular, we give the proof of the existence of the solution trapped in some shrinking set to zero (Proposition \ref{prop:exist}) such that the constructed solution satisfies the conclusion of Theorem \ref{Theo:1}. Readers not interested in technical details may stop there. In Section \ref{sec:4}, we give the proof of Proposition \ref{prop:redu} which gives the reduction of the problem to a finite-dimensional one; and this is the heart of our analysis. \\

\noindent{\textbf{Acknowledgment:}} The authors would like to thank C. Collot for his helpful discussion concerning this work and the anonymous referee for a careful reading and suggestions to improve the presentation of the paper.

\section{Construction of an approximate profile.}\label{sec:2}
This section is devoted to the construction of a suitable approximate solution to \eqref{Pb}  by using the same approach developed in \cite{MRRcjm15}. Similar approachs can also be found in \cite{RScpam13}, \cite{HRapde12}, \cite{RSma14}, \cite{Sjfa12}, \cite{Car16}, \cite{Car161} and \cite{GINapde18}. The key to this construction is the fact that the linearized operator $\Hs$ around $\vec Q$ is completely explicit in the radial setting thanks to the explicit formulas of the kernel elements.

Following the scaling invariance of \eqref{Pb}, we introduce the following change of variables:
\begin{equation}\label{def:simiVars}
\vec w(y,s) = \binom{w_1}{w_2}(y,s) = \binom{u_1}{\lambda u_2}(r,t), \quad y = \frac{r}{\lambda(t)}, \quad \frac{ds}{dt} = \frac{1}{\lambda(t)},
\end{equation}
which leads to the following renormalized flow:
\begin{equation}\label{eq:wys}
\partial_s \vec w +b_1 \Lambda \vec w = \vec F(\vec w), \quad \text{with} \quad b_1 = -\frac{\lambda_s}{\lambda}.
\end{equation}

Let us assume that the leading part of the solution of \eqref{eq:wys} is given by the harmonic map $\vec Q = \binom{Q}{0}$, where $Q$ is the unique  solution (up to scaling) of the equation
\begin{equation}\label{eq:Qy}
Q'' + \frac{(d-1)}{y}Q' - \frac{(d-1)}{2y^2}\sin(2Q) = 0, \quad Q(0) = 0, \; Q'(0) = 1.
\end{equation}
We aim at constructing an approximate solution of \eqref{eq:wys} close to $\vec Q$. The natural way is to linearize equation \eqref{eq:wys} around $\vec Q$, which generates the operator defined by \eqref{def:Hop}. Let us now recall the properties of $\Hs$ in the following subsection.

\subsection{Structure of the linearized operator.}
In this subsection, we recall the main properties of the linearized operator $\Hs$ close to $\vec Q$, which is the heart of both construction of the approximate profile and the derivation of the coercivity properties serving for the high Sobolev energy estimates. Let us start by recalling the following result from Biernat \cite{BIEnon2015}, which gives the asymptotic behavior of the harmonic map $Q$:
\begin{lemma}[Development of the harmonic map $Q$] Let $d \geq 7$, there exists a unique solution $Q$ to equation \eqref{eq:Qy}, which admits the following asymptotic behavior: For any $k \in\NN^*$,\\
$(i)$ (Asymptotic behavior of $Q$) 
\begin{equation}\label{eq:asymQ}
Q(y) = \left\{\begin{array}{ll}
y + \sum \limits_{i = 1}^kc_iy^{2i + 1} + \Oc(y^{2k + 3}) &\text{as}\quad y \to 0, \\
&\\
\dfrac{\pi}{2} - \dfrac{a_0}{ y^{\gamma}}\left[1 + \Oc\left(\dfrac 1{y^{2}}\right) + \Oc\left(\dfrac 1 {y^{ \tilde{\gamma}}}\right)\right]\quad &\text{as} \quad y \to + \infty,
\end{array}
\right.
\end{equation}
where  $\gamma$ is defined in \eqref{def:gamome}, $\tilde{\gamma} = \sqrt{d^2 - 8d + 8}$ and the constant $a_0 = a_0(d) > 0$.\\
$(ii)$ (Degeneracy)
\begin{equation}\label{eq:asymLamQ}
\Lambda Q > 0, \quad \Lambda Q(y) = \left\{\begin{array}{ll}
y + \sum\limits_{i = 1}^kc_i' y^{2i + 1} + \Oc(y^{2k + 3}) &\text{as}\quad y \to 0, \\
&\\
\dfrac{a_0 \gamma}{ y^{\gamma}}\left[1 + \Oc\left(\dfrac 1{y^{2}}\right) + \Oc\left(\dfrac 1 {y^{ \tilde{\gamma}}}\right)\right]\quad &\text{as} \quad y \to + \infty.
\end{array}
\right.
\end{equation}
\end{lemma}
\begin{proof} The proof can be found at pages 184-185 in \cite{BIEnon2015}.
\end{proof}

A remarkable fact is that the linearized operator $\Ls$ admits the following factorization.
\begin{lemma}[Factorization of $\Ls$] \label{lemm:factorL}  Let $d \geq 7$ and define the first order operators
\begin{align}
\As w  &= -\partial_y w + \frac{V}{y}w = - \Lambda Q \partial_y \left(\frac{w}{\Lambda Q}\right), \label{def:As}\\ 
\As^* w &= \frac{1}{y^{d-1}}\partial_y \big(y^{d-1}w\big) + \frac{V}{y}w  =  \frac{1}{y^{d-1}\Lambda Q} \partial_y \left(y^{d-1} \Lambda Q w\right),\label{def:Astar}
\end{align}
where
\begin{equation}\label{eq:asympV}
V(y) := \Lambda \log(\Lambda Q) =  \left\{\begin{array}{ll}
1 + \Oc(y^2)\quad &\text{as}\quad y \to 0, \\
&\\
-\gamma + \Oc\left(\dfrac 1{y^{2}}\right)+ \Oc\left(\dfrac 1{y^{\tilde{\gamma}}}\right)\quad &\text{as} \quad y \to + \infty,
\end{array}
\right.
\end{equation}
We have
\begin{equation}\label{eq:reLAAst}
\Ls = \As^* \As, \quad \tilde{\Ls} = \As \As^*,
\end{equation}
where $\tilde{\Ls}$ stands for the conjugate Hamiltonian.
\end{lemma}

\begin{remark} The adjoint operator $\As^*$ is defined with respect to the Lebesgue measure
$$\int _0^{+\infty}(\As u) w y^{d-1}dy = \int_{0}^{+\infty} u(\As^* w)y^{d-1}dy.$$
\end{remark}

\begin{remark} The factorization \eqref{eq:reLAAst} immediately implies that 
$$\Ls w = 0 \quad \text{if and only if} \quad w \in \textup{span}\{\Lambda Q, \Gamma\},$$
where 
$$\Gamma = -\Lambda Q \int_1^y \frac{dx}{x^{d-1}(\Lambda Q)^2},$$
which admits the asymptotic behavior
\begin{equation}\label{eq:asymGamma}
\Gamma(y) = \left\{\begin{array}{ll}
\dfrac{1}{d y^{d-1}} + \Oc(y)\;\; &\text{as}\;\; y \to 0, \\
&\\
\dfrac{1}{a_0 \gamma (d - 2 - 2\gamma) y^{d - 2 - \gamma}} + \Oc\left(\dfrac 1{y^{d - \gamma}}\right)\;\; &\text{as} \;\; y \to + \infty,
\end{array}
\right.
\end{equation}
\end{remark}

\begin{remark} We have 
\begin{equation}\label{eq:relLsLam}
\Ls(\Lambda w) = \Lambda (\Ls w) + 2 \Ls w - \frac{\Lambda Z}{y^2}w.
\end{equation}
Since $\Ls(\Lambda Q) = 0$, one can express the definition of $Z$ through the potential $V$ as follows: 
\begin{equation}\label{def:ZbyV}
Z(y) = V^2 + \Lambda V + (d-2)V.
\end{equation}
Let $\tilde{Z}$ be defined by 
\begin{equation}\label{def:LstilbyZtil}
\tilde{\Ls} = -\partial_{yy} - \frac{d-1}{y}\py + \frac{\tilde{Z}}{y^2},
\end{equation}
then, a direct computation yields
\begin{equation}\label{def:ZtilbyV}
\tilde{Z}(y) = (V + 1)^2 + (d-2)(V+1) - \Lambda V.
\end{equation}
\end{remark}
The factorization of $\Ls$ allows us to compute $\Ls^{-1}$ in an elementary two step processes as follows:
\begin{lemma}[Inversion of $\Ls$] \label{lemm:inversionL} Let  $f$ be a $\Cc^\infty$ radially symmetric function and $\Ls w = f$, then 
\begin{equation}\label{eq:relaAL}
w = -\Lambda Q\int_0^y\frac{\As w(x)}{\Lambda Q(x)}dx \quad \text{with} \quad \As w = \frac{1}{y^{d-1}\Lambda Q} \int_0^yf(x) \Lambda Q(x) x^{d-1}dx.
\end{equation}
\end{lemma}
\begin{proof} See Lemma 2.5 in \cite{GINapde18}.
\end{proof}

Knowing $\Ls^{-1}$, we can easily defined the inversion of $\Hs$ as follows:
\begin{equation}
\Hs^{-1} = \begin{bmatrix}
0 & \Ls^{-1}\\ -1 & 0
\end{bmatrix}.
\end{equation}
By a direct check, we have 
\begin{equation}\label{def:H2k1}
\Hs^{2k} = (-1)^k \begin{bmatrix}
\Ls^k & 0 \\ 0 & \Ls^k
\end{bmatrix} \quad \text{and} \quad \Hs^{2k + 1} = (-1)^k \begin{bmatrix}
0 & - \Ls^k\\ \Ls^{k + 1} & 0
\end{bmatrix},
\end{equation}
and 
\begin{equation}\label{def:H2k1adj}
\Hs^{*2k} = (-1)^k \begin{bmatrix}
\Ls^k & 0 \\
0 & \Ls^k
\end{bmatrix} \quad \text{and} \quad \Hs^{*(2k + 1)} = (-1)^k \begin{bmatrix}
0 & \Ls^{k + 1}\\ - \Ls^k & 0
\end{bmatrix}.
\end{equation}

\subsection{Admissible functions.} We define a class of admissible functions which display a suitable behavior both at the origin and infinity.
\begin{definition}[Admissible function] \label{def:Admitfunc} Fix $\gamma > 0$, we say that a smooth vector function $\vec f \in  \Cc^\infty(\Rb_+, \Rb) \times \Cc^\infty(\Rb_+, \Rb)$ is admissible of degree $(p_1, p_2, \iota) \in \mathbb{N} \times \mathbb{R} \times \{0,1\}$ if \\
$(i)\;$ $\iota$ is the position: 
$$\vec f = \binom{f}{0} \; \text{if}\; \iota = 0, \quad \vec{f} = \binom{0}{f} \; \text{if}\; \iota = 1.$$ 
$(ii)\;$ $f$ admits a Taylor expansion to all orders around the origin, 
 $$f(y) = \sum_{k = p_1 - \iota, k\; \textup{even}}^p c_ky^{k + 1} + \Oc(y^{p + 2});$$
$(iii)\;$ $f$ and its derivatives admit the bounds, for $y \geq 1$,
$$\forall k \in \mathbb{N}, \quad |\partial^k_y f(y)| \lesssim y^{p_2 - \gamma - \iota - k}.$$
\end{definition}
\begin{remark} Note from \eqref{eq:asymLamQ} that $\Lambda \vec Q = \binom{\Lambda Q}{0}$ is admissible of degree $(0,0, 0)$.
\end{remark}

One note that $\Hs$ naturally acts on the class of admissible function in the following way:
\begin{lemma}[Action of $\Hs$ and $\Hs^{-1}$ on admissible functions] \label{lemm:actionH} Let $\vec f$ be an admissible function of degree $(p_1, p_2, \iota) \in \mathbb{N} \times \mathbb{R} \times \{0, 1\}$, then:\\
$(i)\;$ $\Lambda \vec f$ is admissible of degree $(p_1, p_2, \iota)$.\\
$(ii)\,$ $\Hs \vec f$ is admissible of degree $(\max\{\iota ,p_1 - 1\}, p_2 - 1, (\iota + 1)\wedge 2))$.\\
$(iii)\,$ $\Hs^{-1} \vec f$ is admissible of degree $(p_1 + 1, p_2 + 1, (\iota + 1)\wedge 2))$.
\end{lemma}
\begin{proof} The proof directly follows from the definitions of $\Lambda$, $\Hs$ and $\Hs^{-1}$, and we refer the reader to Lemma 2.8 in \cite{GINapde18} for a similar proof.
\end{proof}

The following lemma is a consequence of Lemma \ref{lemm:actionH}:
\begin{lemma}[Generators of the kernel of $\Hs^k$] \label{lemm:GenLk} Let the sequence of profiles 
\begin{equation}\label{def:Tk}
\Hs \vec T_{k+1} = - \vec T_k, \quad k \in \mathbb{N}, \quad \vec T_0 = \Lambda \vec Q,
\end{equation}
then\\
$(i)\;$ $\vec T_k$ is admissible of degree $(k,k, k\wedge 2)$ for $k \in \mathbb{N}$.\\
$(ii)\;$ $\Lambda \vec T_k - (k - \gamma)\vec T_k$ is admissible of degree $(k, k-1, k\wedge2)$ for $k \in \mathbb{N}$.
\end{lemma}
\begin{proof} $(i)$ We note from \eqref{eq:asymLamQ} that $\vec T_0 = \Lambda \vec Q$ is admissible of degree $(0,0, 0)$. By induction and part $(iii)$ of Lemma \ref{lemm:actionH}, the conclusion simply follows. For item $(ii)$, we refer to Lemma 2.9 in \cite{GINapde18} for an analogous proof.
\end{proof}
\begin{remark} From item $(i)$ of Lemma \ref{lemm:GenLk}, we see that the profile $T_k$ has only one null coordinate, which depends on the index $k$. For simplicity we make use the following notation
\begin{equation}\label{eq:formTk}
\vec T_{2i} = \binom{T_{2i}}{0}, \quad \vec T_{2i + 1} = \binom{0}{T_{2i + 1}}.
\end{equation}

\end{remark}

We end this subsection by introducing a simple notion of homogeneous admissible function.
\begin{definition}[Homogeneous admissible function] Let $L \gg 1$ be an integer and $b = (b_1, \cdots, b_L)$. We say that a vector function $\vec f(y,b)$ is homogeneous of degree $(p_1, p_2, \iota, p_3) \in \mathbb{N} \times \mathbb{\Rb}\times \{0,1\} \times \mathbb{N}$ if it is a finite combination of monomials
$$\vec{g}(y)\prod_{k = 1}^Lb_k^{m_k},$$
with $\vec{g}(y)$ admissible of degree $(p_1, p_2, \iota)$ in the sense of Definition \ref{def:Admitfunc} and
$$(m_1, \cdots, m_L) \in \mathbb{N}^L, \quad\sum_{k = 1}^L km_k = p_3.$$
We set 
$$\text{deg}(\vec f):= (p_1, p_2, \iota, p_3).$$  
\end{definition}

\subsection{Slowly modulated blow-up profile.}
In this subsection, we use the explicit structure of the linearized operator $\Hs$ to construct an approximate blow-up profile. In particular, we claim the following:
\begin{proposition}[Construction of the approximate profile] \label{prop:1}  Let $d \geq 7$ and $L \gg 1$ be an odd integer. Let $M > 0$ be a large enough universal constant, then there exist a small enough universal constant $b^*(M,L) > 0$ such that the following holds true. Let a $\Cc^1$ map
 $$b = (b_1, \cdots, b_L):[s_0,s_1] \mapsto (-b^*, b^*)^L,$$
with a priori bounds in $[s_0,s_1]$:
\begin{equation}\label{eq:relb1bk}
0 < b_1 < b^*, \quad |b_k| \lesssim b_1^k, \quad 2 \leq k \leq L, 
\end{equation}
Then there exist homogeneous profiles 
$$\vec S_k = \vec S_k(y,b), \quad 2 \leq k \leq L + 2,$$
such that
\begin{equation}\label{eq:Qbform}
\vec Q_{b(s)}(y) = \vec Q(y) + \sum_{k = 1}^L b_k(s)\vec T_k(y) + \sum_{k = 2}^{L+2}\vec S_k(y,b) \equiv \vec Q(y) + \vec\Theta_{b(s)}(y),
\end{equation}
generates an approximate solution to the remormalized flow \eqref{eq:wys}:
\begin{equation}\label{def:Psib}
\partial_s \vec Q_{b} + b_1 \Lambda \vec Q_b - \vec F(\vec Q_b) = \vec \Psi_b + \vec{\textup{Mod}},
\end{equation}
with the following property:\\
$(i)$ (Modulation equation)
\begin{equation}\label{eq:Modt}
\vec{\textup{Mod}} = \sum_{k = 1}^L\Big[(b_k)_s + (k - \gamma)b_1b_k - b_{k + 1}\Big] \left[\vec T_k + \sum_{j = k + 1}^{L+2}\frac{\partial \vec S_j}{\partial b_k}\right], 
\end{equation}
where we use the convention $b_{j} = 0$ for $j \geq L+1$.\\
$(ii)$ (Estimate on the profiles) The profiles $(S_k)_{2 \leq k \leq L+2}$ are homogeneous with
\begin{align*}
&\text{deg}(\vec S_k) = (k,k-1, k\wedge2 ,k) \quad \text{for} \quad 2 \leq k \leq L+2,\\
&\frac{\partial \vec S_k}{\partial b_m} = \vec 0 \quad \text{for}\quad 2 \leq k \leq m \leq L.
\end{align*}
$(iii)$ (Estimate on the error $\vec \Psi_b$)  The generated error term is of the form 
$$\vec \Psi_b = \binom{0}{\Psi_b},$$
where $\Psi_b$ satisfies for all $0 \leq m \leq L$,\\
- (global weight bound)
\begin{equation}\label{eq:estGlobalPsib}
\int_{y \leq 2B_1}|\nabla^{m + \hbar}\Psi_b |^2 + \int_{y \leq 2B_1}|\Psi_b \Ls^{m + \hbar}\Psi_b|\lesssim b_1^{2m + 4 + 2(1 - \delta) - C_L\eta},
\end{equation}
where $B_1$, $\hbar$, $\delta$ are defined in \eqref{def:B0B1} and \eqref{def:kdeltaplus}.\\
- (improved local bound)
\begin{equation}\label{eq:estlocalPsib}
\forall M \geq 1, \quad \int_{y \leq M}|\nabla^{m + \hbar}\Psi_b|^2 + \int_{y \leq M}|\Psi_b \Ls^{m + \hbar}\Psi_b| \lesssim M^Cb_1^{2L + 6}.
\end{equation}
\end{proposition}

\begin{remark} From item $(ii)$ of Proposition \ref{prop:1}, we make the abuse of notation 
\begin{equation}\label{eq:formSk}
\vec{S}_{2i} = \binom{S_{2i}}{0}, \quad \vec S_{2i +1} = \binom{0}{S_{2i  +1}}.
\end{equation}
\end{remark}

\begin{proof} We aim at constructing the profiles $(\vec S_k)_{2 \leq k \leq L+2}$ such that $\vec \Psi_b(y)$ defined from \eqref{def:Psib} has the \emph{least possible growth} as $y \to +\infty$. The key to this construction is the fact that the structure of the linearized operator $\Hs$ defined in \eqref{def:Hop} is completely explicit in the radial sector thanks to the explicit formulas of the elements of the kernel of $\Ls$. This procedure will lead to the leading-order modulation equation
\begin{equation}
(b_k)_s = -(k - \gamma)b_1b_k + b_{k+1} \quad \text{for}\quad 1 \leq k \leq L,
\end{equation}
which actually cancels the worst growth of $\vec S_k$ as $y \to +\infty$.

\paragraph{$\bullet$ Expansion of $\vec \Psi_b$}. From \eqref{def:Psib} and \eqref{eq:Qy}, we write
\begin{align*}
&\partial_s \vec Q_{b}  + b_1 \Lambda \vec Q_b - \vec F(\vec Q_b)\\
&= b_1\Lambda \vec Q + \partial_s \Theta_b + \Hs \vec \Theta_b + b_1 \Lambda \Theta_b - \vec N(\vec \Theta_b) := A_1 - \vec N(\vec \Theta_b),
\end{align*}
where $\vec N$ is defined as in \eqref{def:Nvec}. Using the expression \eqref{eq:Qbform} of $\vec \Theta_b$ and the definition \eqref{def:Tk} of $\vec T_k$ (recall that $\Hs \vec T_k = -\vec T_{k-1}$ with the convention $\vec T_0 = \Lambda \vec Q$), we write 
\begin{align*}
A_1&= b_1 \Lambda \vec Q + \sum_{k = 1}^L \Big[(b_k)_s \vec T_k + b_k \Hs \vec T_k + b_1b_k\Lambda \vec T_k \Big] + \sum_{k = 2}^{L+2}\Big[\partial_s \vec S_k + \Hs \vec S_k + b_1 \Lambda \vec S_k\Big]\\
&=  \sum_{k = 1}^L \Big[(b_k)_s\vec T_k - b_{k+1}\vec T_k + b_1b_k \Lambda \vec T_k \Big] + \sum_{k = 2}^{L+2}\Big[\partial_s  \vec S_k + \Hs \vec S_k + b_1 \Lambda \vec S_k\Big]\\
&= \sum_{k = 1}^L \Big[(b_k)_s  + (k - \gamma)b_1b_k- b_{k+1}\Big]\vec T_k\\
&\qquad + \sum_{k = 1}^L\Big[\Hs \vec S_{k + 1} + \partial_s \vec S_k + b_1b_k \big[\Lambda \vec T_k - (k - \gamma)\vec T_k\big] + b_1 \Lambda \vec S_k\Big]\\
& \qquad\qquad+ \Big[\Hs \vec S_{L+2} + \partial_s \vec S_{L+1} + b_1\Lambda \vec S_{L+1}\Big] + \Big[\partial_s \vec S_{L+2} + b_1\Lambda \vec S_{L+2}\Big].
\end{align*}
We now write
\begin{align*}
\partial_s\vec S_k = \sum_{j = 1}^L(b_j)_s\frac{\partial \vec S_k}{\partial b_j} &= \sum_{j = 1}^L\Big[(b_j)_s + (j - \gamma)b_1b_j - b_{j + 1} \Big]\frac{\partial \vec S_k}{\partial b_j} \\
&\qquad - \sum_{j = 1}^L\Big[(j - \gamma)b_1b_j - b_{j + 1} \Big]\frac{\partial \vec S_k}{\partial b_j}.
\end{align*} 
Hence, 
\begin{align*}
A_1 = \vec{\textup{Mod}} + \sum_{k = 1}^{L+1}\left[\Hs \vec S_{k + 1} + \vec E_k\right] + \vec E_{L+2},
\end{align*}
where for $k = 1, \cdots, L$,
\begin{equation}\label{def:Ek1}
\vec E_k= b_1b_k \big[\Lambda \vec T_k - (k - \gamma)\vec T_k\big] + b_1 \Lambda \vec S_k - \sum_{j = 1}^{k-1}\Big[(j - \gamma)b_1b_j - b_{j + 1} \Big]\frac{\partial \vec S_k}{\partial b_j},
\end{equation} 
and for $k = L+1, L+2$,
\begin{equation}\label{def:EkL12}
\vec E_k = b_1 \Lambda \vec S_k - \sum_{j = 1}^{L}\Big[(j - \gamma)b_1b_j - b_{j + 1} \Big]\frac{\partial \vec S_k}{\partial b_j}.
\end{equation}

Recall from \eqref{def:Nq} that the nonlinear term is given by
$$\vec N(\vec \Theta_b)= \binom{0}{N(\Theta_{b,1})}:= \binom{0}{A_2}.$$ 
Let us denote 
$$f(x) = \sin(2x)$$
and use a Taylor expansion to write (see pages 1740 in \cite{RSapde2014} for a similar computation)
\begin{align*}
A_2 &= \frac{(d-1)}{2y^2}\left[\sum_{i = 2}^{L+2}\frac{f^{(i)}(Q)}{i!}\left(\sum_{k = 2, \textup{even}}^{L-1}b_iT_i + \sum_{k = 2}^{L+2}S_{k,1} \right)^i + R_2\right]\\
&\quad = \frac{(d-1)}{2y^2}\left[\sum_{i = 2}^{L+2}P_i + R_1 + R_2\right],
\end{align*}
where 
\begin{equation}\label{def:Pj}
P_i = \sum_{j = 2}^{L+2}\frac{f^{(j)}(Q)}{j!}\sum_{|J|_1 = j, |J|_2 = i}c_J \prod_{k = 2, \textup{even}}^{L-1} b_k^{i_k}T_{k}^{i_k}\prod_{k=2}^{L+2}S_{k,1}^{j_k},
\end{equation}
\begin{equation}\label{def:R_1}
R_1 = \sum_{j = 2}^{L+2}\frac{f^{(j)}(Q)}{j!}\sum_{|J|_1 = j, |J|_2 \geq L+3}c_J \prod_{k = 2, \textup{even}}^{L-1}b_k^{i_k}T_{k}^{i_k}\prod_{k=2}^{L+2}S_{k,1}^{j_k},
\end{equation}
\begin{equation}\label{def:R2}
R_2 = \frac{\Theta_{b,1}^{L+3}}{(L+2)!}\int_0^1(1 - \tau)^{L+2}f^{(L+3)}(Q + \tau \Theta_{b,1})d\tau,
\end{equation}
with $J = (i_1, \cdots, i_L, j_2, \cdots, j_{L+2}) \in \mathbb{N}^{2L + 1}$ and 
\begin{equation}\label{def:J1J2}
|J|_1 = \sum_{k = 1}^L i_k + \sum_{k = 2}^{L+2}j_k, \quad |J|_2 = \sum_{k=1}^L ki_k + \sum_{k = 2}^{L+2}kj_k.
\end{equation}
In conclusion, we have 
\begin{equation}\label{eq:expanPsib}
\vec \Psi_b = \sum_{k = 1}^{L+1}\left[\Hs \vec S_{k + 1} + \vec E_k - \frac{(d-1)}{2y^2}\vec P_{k+1}\right] + \vec E_{L+2} - \frac{(d-1)}{2y^2}(\vec R_1 + \vec R_2),
\end{equation}
where we write 
$$\vec P_{k} = \binom{0}{P_k}, \quad \vec R_1 = \binom{0}{R_1}, \quad \vec R_2 = \binom{0}{R_2}.$$

\paragraph{$\bullet$ Construction of $\vec S_k$.} From the expression of $\vec \Psi_b$ given in \eqref{eq:expanPsib}, we construct iteratively the sequences of profiles $(\vec S_k)_{1 \leq k \leq L+2}$ through the scheme
\begin{equation}\label{def:Sk}
\left\{\begin{array}{ll}
\vec S_1 &= \vec 0, \\
\vec S_k &= - \Hs^{-1}\vec F_k, \quad 2 \leq k \leq L+2,
\end{array}\right.
\end{equation}
where 
$$\vec F_k = \vec E_{k-1} - \frac{(d-1)}{2y^2}\vec P_{k} \quad \text{for} \quad 2\leq k \leq L+2.$$
We claim by induction on $k$ that $\vec F_k$ is homogeneous with 
\begin{equation}\label{eq:degFk}
\text{deg}(\vec F_k) = (k-1, k-2, (k - 1)\wedge2, k) \quad \text{for} \quad 2 \leq k \leq L+2,
\end{equation}
and 
\begin{equation}\label{eq:estparFk}
\frac{\partial \vec F_k}{\partial b_m} = \vec 0 \quad \text{for}\quad 2 \leq k \leq m \leq L+2.
\end{equation}
From item $(iii)$ of Lemma \ref{lemm:actionH} and \eqref{eq:degFk}, we deduce that $\vec S_k$ is homogeneous of degree
$$\text{deg}(\vec S_k) = (k,k-1, k\wedge2, k) \quad \text{for}\quad 2 \leq k \leq L+2,$$
and from \eqref{eq:estparFk}, we get
$$\frac{\partial \vec S_k}{\partial b_m} = \vec 0 \quad \text{for} \quad 2 \leq k \leq m \leq L+2,$$
which is the conclusion of item $(ii)$.\\

Let us now give the proof of \eqref{eq:degFk} and \eqref{eq:estparFk}. We proceed by induction.\\
\noindent - Case $k = 2$: We compute explicitly from \eqref{def:Ek1} and \eqref{def:Pj},
$$\vec F_2 = \vec E_1 - \frac{(d-1)}{2y^2}\vec P_2 = b_1^2\left[\Lambda \vec T_1 - (1 - \gamma)\vec T_1 + \frac{(d-1)}{2y^2}\vec P_2\right],$$
which directly follows \eqref{eq:estparFk}. From Lemma \ref{lemm:GenLk}, we know that $\Lambda \vec T_1 - (1 - \gamma)\vec T_1$ are admissible of degree $(1,0,1)$. It remains to check that $\frac{1}{y^2}\vec P_2 = \binom{0}{\frac{P_2}{y^2}}$ is admissible of degree $(1,0,1)$. To do so, let us write from the definition \eqref{def:Pj}, 
$$\frac{P_2}{y^2} = \frac{f''(Q)}{y^2}T_1^2.$$
Using \eqref{eq:asymQ}, one can check the bound
\begin{equation}\label{eq:estfjm}
\forall m,j \in \mathbb{N}^2, \quad\left|\partial_y^m \left(\frac{f^{(j)}(Q)}{y^2}\right)\right| \lesssim y^{-\gamma - 2 - m} \quad \text{as}\quad y \to +\infty.
\end{equation}
Since $\vec T_1$ is admissible of degree $(1,1, 1)$, we have that
$$\forall m \in \mathbb{N}, \quad |\partial_y^m(T_1^2)| \lesssim y^{- 2\gamma - m}  \quad \text{as}\quad y \to +\infty.$$
By the Leibniz rule and the fact that $2\gamma > 2$, we get that
$$\forall m,j \in \mathbb{N}^2, \quad\left|\partial_y^m \left(\frac{f^{(j)}(Q)}{y^2} T_1^2\right)\right| \lesssim y^{- 2 - \gamma - m}.$$
We also have the expansion near the origin, 
$$\frac{f''(Q)}{y^2}T_1^2 = \sum_{i = 0, even}^k c_iy^{i + 1} + \Oc(y^{k + 2}), \quad k \geq 1.$$
Hence, $\frac{1}{y^2}\vec P_2$ is admissible of degree $(1,0,1)$, which concludes the proof of \eqref{eq:degFk} for $k = 2$.\\

\noindent - Case $k \to k+1$: Estimate \eqref{eq:estparFk} holds by direct inspection. Let us now assume that $\vec S_k$ is homogeneous of degree $(k, k-1, k\wedge2, k)$ and prove that $\vec S_{k + 1}$ is homogeneous of degree $(k+1, k, (k+1)\wedge2, k+1)$. In particular, the claim immediately follows from part $(iii)$ of Lemma \ref{lemm:actionH} once we show that $\vec F_{k+1}$ is homogeneous with
\begin{equation}\label{eq:Fk1EkPk1}
\text{deg}(\vec F_{k+1}) = \text{deg}\left(\vec E_k + \frac{\vec P_{k+1}}{y^2} \right) = (k, k-1, k\wedge2, k+1).
\end{equation}
From part $(ii)$ of Lemma \ref{lemm:GenLk} and the a priori assumption \eqref{eq:relb1bk}, we see that $b_1b_k(\Lambda T_k - (k - \gamma)T_k)$ is homogeneous of degree $(k, k-1, k\wedge2, k+1)$. From part $(i)$ of Lemma \ref{lemm:actionH} and the induction hypothesis, $b_1 \Lambda \vec S_k$ is also homogeneous of degree $(k, k-1, k\wedge2, k+1)$. By definition, $b_1\frac{\partial \vec S_k}{\partial b_1}$ is homogeneous and has the same degree as $\vec S_k$. Thus, 
$$\left((j - \gamma)b_1 - \frac{b_{2}}{b_{1}}\right)\left(b_1\frac{\partial \vec S_k}{\partial b_1}\right)$$
is homogeneous of degree $(k, k-1, k\wedge2, k+1)$. From definitions \eqref{def:Ek1} and \eqref{def:EkL12}, we derive  
$$\text{deg}(\vec E_k) = (k, k-1, k\wedge2, k+1), \quad k \geq 1.$$
It remains to check that the term $\frac{\vec P_{k+1}}{y^2}$ is homogeneous of degree $(k, k-1, k\wedge2, k+1)$. From the definition \eqref{def:Pj}, we see that if $k$ is even, then $P_{k+1} = 0$ and we are done. If $k$ is odd, then we see that $\frac{P_{k+1}}{y^2}$ is a linear combination of monomials of the form 
$$M_J(y) = \frac{f^{(j)}(Q)}{y^2}\prod_{m = 2, \textup{even}}^{L-1} b_m^{i_m}T_m^{i_m}\prod_{m = 2, \textup{even}}^{L+2}S_{m,1}^{j_m},$$
with 
$$J= (i_1, \cdots, i_L, j_2, \cdots, j_{L+2}), \quad |J|_1 = j, \; |J|_2 = k+1, \; 2 \leq j \leq k+1.$$
Recall from part $(i)$ of Lemma \ref{lemm:GenLk} that $\textup{deg}(\vec T_m) = (m, m, m\wedge2)$, we then have
$$\forall n \in \mathbb{N}, \quad |\partial_y^n T_m| \lesssim y^{m - m\wedge2 - \gamma - n} \quad \text{as} \quad y \to +\infty,$$
and from the induction hypothesis and the a priori bound \eqref{eq:relb1bk},
$$\forall n \in \mathbb{N}, \quad |\partial_y^n S_{m,1}| \lesssim b_1^{m}y^{m -1  - m\wedge2 - \gamma - n} \quad \text{as} \quad y \to +\infty. $$
Together with the bound \eqref{eq:estfjm}, we obtain the following bound at infinity,
$$| M_J|\lesssim b_1^{|J|_2} y^{|J|_2 - \gamma - |J|_1 \gamma - 2 - \sum_{m=2, \textup{even}}^{L+2}j_m} \lesssim b_1^{k+1}y^{k-1 - \gamma}.$$
The control of $\partial_y^n M_J$ follows by the Leibniz rule and the above estimates. The expansion near the origin can be checked by the same way. This concludes the proof of \eqref{eq:Fk1EkPk1} as well as part $(ii)$ of Proposition \ref{prop:1}.\\

\paragraph{$\bullet$ Estimate on $\vec \Psi_b$.} From \eqref{eq:expanPsib} and \eqref{def:Sk}, the expression of $\vec \Psi_b$ is now reduced to 
$$\vec \Psi_b = \vec E_{L+2} - \frac{(d-1)}{y^2}(\vec R_1 + \vec R_2),$$
where $\vec E_{L+2}$ is defined by \eqref{def:EkL12}, $\vec R_1 = \binom{0}{R_1}$ and $\vec R_2 = \binom{0}{R_2}$ with $R_1, R_2$ being given by \eqref{def:R_1} and \eqref{def:R2}. Note that the first coordinate of $\vec \Psi_b$ is null, so we can write for simplicity
\begin{equation}\label{eq:formPsib}
\vec\Psi_b = \binom{0}{\Psi_b} = \binom{0}{E_{L+2}- \frac{(d-1)}{y^2}(R_1 + R_2) }.
\end{equation}

We start by controlling $\vec E_{L+2}$ term. Since $\vec S_{L+2}$ is homogeneous of degree $(L+2, L+1, 1, L+2)$ and thus so are $\Lambda \vec S_{L+2}$ and $b_1 \frac{\partial \vec S_{L+2}}{\partial b_1}$. This follows that $\vec E_{L+2}$ is homogeneous of degree $(L+2, L+1, 1, L+3)$. Using part $(ii)$ of Lemma \ref{lemm:actionH} yields 
$$\textup{deg}(\Hs^{2m + 2\hbar +1}\vec E_{L+2}) = (\max\{1, L - 2m - 2\hbar + 1\}, L - 2m - 2\hbar, 0, L+3).$$
From the relation $d - 2\gamma - 2\hbar = 2\delta$ (see \eqref{def:kdeltaplus}), we estimate for all $0 \leq m \leq L$,
\begin{align*}
\int_{y \leq 2B_1}|E_{L+2} \Ls^{m + \hbar} E_{L+2}| &\lesssim b_1^{2L+6}\int_{y \leq 2B_1}y^{2L - 2\gamma - 2(\hbar + m)} y^{d-1}dy\\
&\quad\lesssim b_1^{2L+6}\int_{y \leq 2B_1}y^{2(L - m + \delta) - 1}dy\\
&\qquad \lesssim b_1^{(2L + 6) - 2(L - m + \delta)(1 + \eta)}\\
&\qquad \quad \lesssim b_1^{2m + 4 + 2(1 - \delta) - C_L\eta},
\end{align*}
where $\eta = \eta(L)$, $0 < \eta \ll 1$.

We now turn to the control of the term $\frac{R_1}{y^2}$, which is a linear combination of terms of the form (see \eqref{def:R_1})
$$\tilde{M}_J = \frac{f^{(j)}(Q)}{y^2}\prod_{n = 2, \textup{even} }^{L-1}b_n^{i_n}T_n^{i_n}\prod_{n = 2, \textup{even}}^{L+2}S_{n}^{j_n},$$
where we used the abuse notations \eqref{eq:formTk} and \eqref{eq:formSk}, and
$$J = (i_1, \cdots, i_L, j_2, \cdots, j_{L+2}), \; |J|_1 = j, \; |J|_2 \geq L+3, \; 2 \leq j \leq L+2.$$
Using the admissibility of $\vec T_n$ and the homogeneity of $\vec S_n$, we get the bounds
$$|\tilde{M}_J| \lesssim b_1^{L + 3}y^{|J|_2 + j - 1} \lesssim b_1^{L+3}y^{L + 4} \quad \text{as} \quad y \to 0,$$
and 
$$|\tilde{M}_J| \lesssim b_1^{|J|_2}y^{|J|_2 - j\gamma - 2 - \gamma}  \quad \text{as}\quad y \to +\infty,$$
where we used the fact that $j \geq 2$ and $2 - j\gamma < 0$, and similarly for higher derivatives by the Leibniz rule. Thus, we obtain the round estimate for all $0 \leq m \leq L$,
\begin{align*}
\int_{y \leq 2B_1}\left|\frac{R_1}{y^2} \Ls^{m + \hbar}\left(\frac{R_1}{y^2}\right)\right| &\lesssim b_1^{2|J|_2}\int_{y \leq 2B_1}y^{2|J|_2 - 2m - 2j \gamma - 4 + 2\delta - 1} dy\\
&\quad\lesssim b_1^{2m + 4 + 2(1 - \delta) - C_L\eta}.
\end{align*}
The term $\frac{R_2}{y^2}$ is estimated exactly as for the term $\frac{R_1}{y^2}$ using the definition \eqref{def:R2}. This concludes the proof of \eqref{eq:estGlobalPsib}. The local estimate \eqref{eq:estlocalPsib} directly follows from the homogeneity of $\vec S_k$ and the admissibility of $\vec T_k$. This concludes the proof of Proposition \ref{prop:1}.
\end{proof}

\bigskip

We now proceed to a simple localization of the profile $\vec Q_b$ to avoid the growth of tails in the region $y \geq 2B_1 \gg B_0$. More precisely, we claim the following:

\begin{proposition}[Estimates on the localized profile] \label{prop:localProfile} Under the assumptions of Proposition \ref{prop:1}, we assume in addition the a priori bound 
\begin{equation}\label{eq:apriorib1}
|(b_1)_s| \lesssim b_1^2.
\end{equation}
Consider the localized profile
\begin{equation}\label{def:Qbtil}
\vec{\mathbf{Q}}_{b(s)}(y) = \vec Q(y) + \sum_{k = 1}^Lb_k\vec{\mathbf{T}}_k + \sum_{k = 2}^{L+2}\vec{\mathbf{S}}_k \quad \text{with}\quad \vec{\mathbf{T}}_k = \chi_{B_1}\vec T_k, \; \vec{\mathbf{S}}_k = \chi_{B_1}\vec S_k,
\end{equation}
where $B_1$ and $\chi_{B_1}$ are defined as in \eqref{def:B0B1} and \eqref{def:chiM}. Then
\begin{equation}\label{def:Psibtilde}
\partial_s \vec{\mathbf{Q}}_{b} + b_1 \Lambda \vec{\mathbf{Q}}_b  - \vec F (\vec{\mathbf{Q}}_b) = \vec{\mathbf{\Psi}}_b + \chi_{B_1}\,\vec{\textup{Mod}},
\end{equation}
where $\vec{\mathbf{\Psi}}_{b}$ satisfies the bounds:\\

\noindent $(i)\;$ (Large Sobolev bound) For all $0 \leq m \leq L-1$,
\begin{align}
\|\vec{\mathbf{\Psi}}_b \|_{2m + 2\hbar +2}^2 + \int \left|\nabla^{m + \hbar + 1}(\mathbf{\Psi}_b)_1\right|^2 + \int \left|\nabla^{m + \hbar}(\mathbf{\Psi}_b)_2\right|^2  \lesssim b_1^{2m + 2 + 2(1 - \delta)  - C_L \eta},\label{eq:estPsibLarge1}
\end{align}
and 
\begin{align}
\|\vec{\mathbf{\Psi}}_b \|_{2L + 2\hbar +2}^2 \lesssim b_1^{2L + 2 + 2(1 - \delta)(1 + \eta)},\label{eq:estPsibLarge2}
\end{align}
where $\hbar$ and $\delta$ are defined by \eqref{def:kdeltaplus}. \\

\noindent $(ii)$ (Local bound) For all $M \leq \frac{B_1}{2}$ and $0 \leq m \leq L$, 
\begin{equation}\label{eq:estPsiblocalTilde}
\|\vec{\mathbf{\Psi}}_b \|_{2m + 2\hbar +2, (y \leq M)}^2 \lesssim M^Cb_1^{2L + 6}.
\end{equation}

\noindent $(iii)$ (Refined local bound near $B_0$) For all $0 \leq m \leq L$, 
\begin{equation}\label{eq:estPsiblocalB0}
\|\vec{\mathbf{\Psi}}_b \|_{2m + 2\hbar +2, (y \leq B_0)}^2\lesssim b_1^{2m + 4 + 2(1 - \delta) - C_L\eta}.
\end{equation}
\end{proposition}

\begin{proof} 
The proof is the same as  Proposition 2.12 in \cite{GINapde18} because the linear operator $\Ls$ is the same as the one defined in \cite{GINapde18}. Although the definition of parameters $\hbar, \delta, B_0, B_1$ are slightly different from the ones defined in \cite{GINapde18}, the reader will have absolutely no difficulty to adapt that proof to the new situation. For that reason, we refer the reader to \cite{GINapde18} for an analogous proof. We would like to mention the fact that the bound \eqref{eq:estPsibLarge1} is worse than \eqref{eq:estGlobalPsib} due to the localization effect of the approximate profile. In particular, replacing the profile $\vec{T}_i$ by $\chi_{B_1}\vec{T}_i$ and $\vec{S}_j$ by $\chi_{B_1}\vec S_j$ would give a worst estimate on $\vec {\mathbf{\Psi}}_b$ in the zone $B_1 \leq y \leq 2B_1$, where we loose $b_1^2$ approximately.  However, this localization will be necessary for our analysis. 
\end{proof}
\subsection{Study of the dynamical system for $b = (b_1, \cdots, b_L)$.}
The construction of the $\vec Q_b$ profile formally leads to the finite dimensional dynamical system for $b = (b_1, \cdots, b_L)$ by setting to zero the inhomogeneous $\vec{\textup{Mod}}$ term given in \eqref{eq:Modt}:
\begin{equation}\label{eq:systemb}
(b_k)_s + (k - \gamma)b_1b_k - b_{k+1} = 0, \;\; 1 \leq k \leq L, \;\; b_{L+1} = 0.
\end{equation}
The system \eqref{eq:systemb} admits explicit solutions and the linearized operator near these solutions is explicit. In particular, we have the following.

\begin{lemma}[Solution to the system \eqref{eq:systemb}]\label{lemm:solSysb} Let $\ell \in \mathbb{N}^*$ with $\gamma < \ell \ll L$, and the sequence
\begin{equation}\label{def:ck}
\left\{\begin{array}{ll}
c_1 &= \frac{\ell}{\ell - \gamma},\\
c_{k + 1} &= - \frac{\gamma(\ell - k)}{\ell - \gamma}c_k, \quad 1 \leq k \leq \ell-1,\\
c_{k + 1} &= 0, \quad k \geq \ell.
\end{array} \right.
\end{equation}
Then the explicit choice 
\begin{equation}\label{eq:solbe}
b_k^e(s) = \frac{c_k}{s^k}, \quad s > 0, \quad 1 \leq k \leq L,
\end{equation}
is a solution to \eqref{eq:systemb}.
\end{lemma}
The proof of Lemma \ref{lemm:solSysb} directly follows from an explicit computation which is left to the reader. We claim that the linearized flow of \eqref{eq:systemb} near the solution \eqref{eq:solbe} is explicit and displays $(\ell - 1)$ unstable directions. Note that the stability is considered in the sense that 
$$\sup_s s^k|b_k(s)| \leq C_k, \quad 1 \leq k \leq L.$$ 
In particular, we have the following result which was proved in \cite{MRRcjm15}:
\begin{lemma}[Linearization of \eqref{eq:systemb} around \eqref{eq:solbe}] \label{lemm:lisysb} Let 
\begin{equation}\label{eq:Ukbke}
b_k(s) = b_k^e(s) + \frac{\Uc_k(s)}{s^k}, \quad 1 \leq k \leq \ell,
\end{equation}
and note $\Uc = (\Uc_1, \cdots, \Uc_\ell)$. Then, for $1 \leq k \leq \ell-1$,
\begin{equation}\label{eq:bkk1}
(b_k)_s + (k - \gamma)b_1b_k - b_{k+1} = \frac{1}{s^{k+1}}\left[s(\Uc_k)_s - (A_\ell \Uc)_k + \Oc(|\Uc|^2)\right], 
\end{equation}
and 
\begin{equation}\label{eq:bell}
(b_\ell)_s + (\ell - \gamma)b_1b_\ell = \frac{1}{s^{k+1}}\left[s(\Uc_\ell)_s - (A_\ell \Uc)_\ell + \Oc(|\Uc|^2)\right], 
\end{equation}
where 
$$A_\ell = (a_{i,j})_{1\leq i,j\leq \ell} \quad \text{with}\quad \left\{\begin{array}{ll}
a_{1,1} &= \frac{\gamma(\ell - 1)}{\ell - \gamma} - (1 - \gamma)c_1,\\
a_{i,i} &= \frac{\gamma(\ell - i)}{\ell - \gamma}, \quad 2 \leq i \leq \ell,\\
a_{i, i+1}&= 1 , \quad 1 \leq i \leq \ell - 1,\\
a_{1,i} & = -(i - \gamma)c_i, \quad 2 \leq i \leq \ell,\\
a_{i,j} &= 0 \quad \textup{ortherwise}
\end{array} \right.$$
Moreover, $A_\ell$ is diagonalizable: 
\begin{equation}\label{eq:diagAlPl}
A_\ell = P_\ell^{-1}D_\ell P_\ell, \quad D_\ell = \textup{diag}\left\{-1, \frac{2\gamma}{\ell - \gamma}, \frac{3\gamma}{\ell - \gamma}, \cdots, \frac{\ell \gamma}{\ell - \gamma}\right\}.
\end{equation}
\end{lemma}
\begin{proof} Since we have an analogous system as the one in \cite{MRRcjm15} and the proof is essentially the same as written there, we kindly refer the reader to see Lemma 3.7 in \cite{MRRcjm15} for all details of the proof.  
\end{proof}

\section{Proof of Theorem \ref{Theo:1} assuming technical results.} \label{sec:3}
This section is devoted to the proof of Theorem \ref{Theo:1}. We hope that the explanation of the strategy we give in this section will be reader friendly. We proceed in 3 subsections:\\
- In the first subsection, we give an equivalent formulation of the linearization of the problem in the setting \eqref{eq:dec_i}.\\
- In the second subsection, we prepare the initial data and define the shrinking set $\Sc_K$ (see Definition \ref{def:Skset}) such that the solution trapped in this set satisfies the conclusion of Theorem \ref{Theo:1}.\\
- In the third subsection, we give all arguments of the proof of the existence of solutions trapped in $\Sc_K$ (Proposition \ref{prop:exist}) assuming an important technical result (Proposition \ref{prop:redu}) whose proof is left to the next section. Then we conclude the proof of Theorem \ref{Theo:1}. 

\subsection{Linearization of the problem.} 
Let $L \gg 1$ be an odd integer, $s_0 \gg 1$ and $\ell > \gamma$. 
We introduce the following notation $$\mathbf{f}= f\chi_{B_1}.$$
We introduce the renormalized variables:
\begin{equation}
\vec w(y,s) = \binom{u_1}{\lambda u_2}(r,t), \quad y = \frac{r}{\lambda(t)}, \quad s = s_0 + \int_0^t \frac{d\tau}{\lambda(\tau)},
\end{equation} 
and the decomposition
\begin{equation}\label{def:qys}
\vec w(y,s) = \big(\vec{\mathbf{Q}}_{b(s)} + \vec q \,\big)(y,s),
\end{equation}
where $\vec{\mathbf{Q}}_b$ is defined by \eqref{def:Qbtil} and the modulation parameters 
$$\lambda(s) > 0, \quad b(s) = (b_1(s), \cdots, b_L(s))$$
are determined from the $L+1$ orthogonality conditions:
\begin{equation}\label{eq:orthqPhiM}
\left<\vec q, \Hs^{*k} \vec \Phi_M \right> = 0, \quad 0 \leq k \leq L,
\end{equation} 
where $\vec \Phi_M$ is a fixed direction depending on some large constant $M$ defined by
\begin{equation}\label{def:PhiM}
\vec \Phi_M = \sum_{k = 0}^L c_{k,M}\Hs^{*k}(\chi_M \Lambda \vec Q),
\end{equation}
with 
\begin{equation}\label{def:ckM}
c_{0,M} = 1, \quad c_{k, M} = (-1)^{k+1}\frac{\sum_{j = 0}^{k - 1}c_{j,M}\left<\Hs^{*j}(\chi_M \Lambda \vec Q), \vec T_k\right> }{\left<\chi_M \Lambda \vec Q, \Lambda \vec Q\right>}, \quad 1 \leq k \leq L.
\end{equation}
Here, $\vec \Phi_M$ is build to ensure the nondegeneracy 
\begin{equation}\label{eq:PhiMLamQ}
\left<\vec \Phi_M, \Lambda \vec Q\right> = \left<\chi_M \Lambda \vec Q, \Lambda \vec Q\right> \gtrsim M^{d - 2\gamma},
\end{equation}
and the cancellation 
\begin{equation}\label{id:TkPhiM0}
\left<\vec \Phi_M, \vec T_k\right> = \sum_{j = 0}^{k - 1}c_{j,M}\left<\Hs^{*j}(\chi_M\Lambda \vec Q), \vec T_k\right> + c_{k,M} (-1)^k \left<\chi_M \Lambda \vec Q, \Lambda \vec Q\right> = 0,
\end{equation}
In particular, we have 
\begin{equation}\label{id:TkPhiMi}
\left<\Hs^i\vec T_k, \vec\Phi_M \right> = (-1)^k\left<\chi_M \Lambda \vec Q, \Lambda \vec Q\right>\delta_{i,k}, \quad 0 \leq i,k\leq L.
\end{equation}

\medskip

From \eqref{eq:wys}, we see that $\vec q$ satisfies
\begin{equation}\label{eq:qys}
\partial_s \vec q - \frac{\lambda_s}{\lambda} \Lambda \vec q + \Hs \vec q = - \vec{\mathbf{\Psi}}_b  - \vec{\mathbf{M}} + \vec L(\vec q) - \vec N(\vec q) \equiv \vec\Fc,
\end{equation}
where 
\begin{align}
\vec{\mathbf{M}} & =  \chi_{B_1}\vec{\textup{Mod}} - \left(\frac{\lambda_s}{\lambda} + b_1\right)\Lambda \vec{\mathbf{Q}}_b = \binom{\mathbf{M}_1}{\mathbf{M}_2}, \label{def:M1M2}\\
\vec L(\vec q) &= \binom{0}{\frac{(d-1)}{y^2}\big[\cos(2Q) - \cos(2\mathbf{Q}_{b,1})\big]q_1} = \binom{0}{L(q_1)},\label{def:Lq}\\
\vec N(\vec q) &= \binom{0}{\frac{(d-1)}{2y^2}\big[\sin(2\mathbf{Q}_{b,1} + 2q_1) - \sin(2\mathbf{Q}_{b,1}) - 2q_1 \cos(2\mathbf{Q}_{b,1}) \big]} = \binom{0}{N(q_1)}.\label{def:Nq}
\end{align}
We also need to write the equation \eqref{eq:qys} in the original variables. To do so, let the rescaled linearized operator:
\begin{equation}\label{def:Llambda}
\Ls_\lambda = - \partial_{rr} - \frac{(d-1)}{r}\partial_r + \frac{Z_\lambda}{r^2} \quad \text{with} \quad Z_\lambda(r) = Z\left(\frac{r}{\lambda}\right),
\end{equation}
and the renormalized vector function
$$\vec v(r,t) = \binom{q_1}{\frac{1}{\lambda}q_2}(y,s), \quad r = \lambda y, \quad \frac{dt}{ds} = \lambda.$$
We compute 
$$\partial_t \vec v = \frac{1}{\lambda^2}\binom{\lambda \left(\ps q_1 - \frac{\lambda_s}{\lambda}\Lambda q_1\right)}{\ps q_2 - \frac{\lambda_s}{\lambda}Dq_2},$$
then from \eqref{eq:qys}, $\vec v$ satisfies the equation
\begin{equation}\label{eq:vrt}
\partial_t \vec v + \Hs_\lambda \vec v = \frac{1}{\lambda^2}\vec \Fc_\lambda,
\end{equation}
where 
$$\Hs_\lambda = \begin{bmatrix}
0 & -1\\ \Ls_\lambda & 0
\end{bmatrix}, \quad \vec \Fc_\lambda(r,t) = \binom{\lambda \Fc_1}{\Fc_2}(y,s).$$
Note that 
\begin{equation}\label{eq:reLLlambda}
\Ls_\lambda v_1(r,t) = \frac{1}{\lambda^2}\Ls q_1(y,s),
\end{equation}
and by the factorization of $\Ls$, we can write
$$\Ls_\lambda  = \As^*_\lambda \As_\lambda,$$
where 
$$\As^*_\lambda f = \frac{1}{r^{d-1}}\pr(r^{d-1}f) + \frac{V_\lambda}{r}f  \quad \text{and} \quad \As_\lambda f= -\pr f + \frac{V_\lambda}{r}f \quad \text{with} \quad V_\lambda(r) = V\left(\frac{r}{\lambda}\right).$$
The reader should keep in mind that $\Hs_\lambda$, $\Ls_\lambda$, $\As^*_\lambda$ and $\As_\lambda$ act on functions depending on variable $r$, while $\Hs, \Ls, \As^*$ and $\As$ act on functions depending on variable $y$. 

\subsection{Preparation of the initial data.}
We describe in this subsection the set of initial data $\vec u_0 = (u(x,0), \pt u(x,0))$ of the problem \eqref{Pb} as well as the initial data for $(b, \lambda)$ leading to the blowup scenario of Theorem \ref{Theo:1}. Our construction is build on a careful choice of the initial data for the modulation parameter $b$ and the radiation $\vec q$ at time $s = s_0$. In particular, we will choose them in the following way:

\begin{definition}[Choice of the initial data] \label{def:1}  Given $\eta$, $\sigma$ and $\delta$ as in \eqref{def:B0B1}, \eqref{def:sigma} and \eqref{def:kdeltaplus}. Consider the change of variable 
\begin{equation}\label{def:VctoUc}
\Vc = P_\ell \Uc,
\end{equation}
where $\Uc = (\Uc_1, \cdots, \Uc_\ell)$ is introduced in the linearization \eqref{eq:Ukbke} and $P_\ell$ refers to the diagonalization \eqref{eq:diagAlPl} of $A_\ell$.

We assume that 
\begin{itemize}
\item Smallness of the initial perturbation for the $b_k$ unstable modes: 
\begin{equation}
|s_0^{\frac{\eta}{2}(1 - \delta)}\Vc_k(s_0)| < 1 \quad \text{for}\;\; 2 \leq k \leq \ell.
\end{equation}
\item Smallness of the initial perturbation for the $b_k$ stable modes: 
\begin{equation}\label{eq:initbk}
|s_0^{\frac{\eta}{2}(1 - \delta)}\Vc_1(s_0)| < 1, \quad |b_k(s_0)| <s_0^{-\frac{5\ell(k - \gamma)}{\ell - \gamma}} \text{for}\;\; \ell+1 \leq k \leq L.
\end{equation}
\item Smallness of the data:
\begin{equation}\label{eq:intialbounE2m}
\|\vec q(s_0)\|^2_{\Hs^\Bbbk \times \Hs^{\Bbbk - 1}} + \|\vec q(s_0)\|^2_{\dot{H}^\sigma \times \dot{H}^{\sigma - 1}}  < s_0^{-\frac{10L\ell}{\ell - \gamma}},
\end{equation}
where 
\begin{align*}
\|\vec q\|^2_{\Hs^\Bbbk \times \Hs^{\Bbbk - 1}} &= \int |(q_1)_\Bbbk|^2 + \int |(q_2)_{\Bbbk - 1}|^2 \\
& + \sum_{k = 0}^{\Bbbk - 1}\int \frac{|(q_1)_k|^2}{y^2(1 + y^{2\Bbbk -2 - 2k})} + \sum_{k = 0}^{\Bbbk - 2}\int \frac{|(q_2)_k|^2}{y^2(1 + y^{2\Bbbk - 4 - 2k)}}.
\end{align*}

\item Normalization: up to a fixed rescaling, we may always assume
\begin{equation}
\lambda(s_0) = 1.
\end{equation}
\end{itemize}
\end{definition}

In particular, the initial data described in Definition \ref{def:1} belongs to the following set which shrinks to zero as $s \to +\infty$:

\begin{definition}[Definition of the shrinking set]\label{def:Skset} Given $\eta$, $\sigma$ and $\delta$ as in \eqref{def:B0B1}, \eqref{def:sigma} and \eqref{def:kdeltaplus}. For all $K \geq 1$ and $s \geq 1$, we define $\Sc_K(s)$ as the set of all $(b_1(s), \cdots, b_L(s), \vec q(s))$ such that
\begin{align*}
&\left|\Vc_k(s) \right| \leq 10s^{-\frac{\eta}{2}(1 - \delta)} \quad \text{for}\;\; 1 \leq k \leq \ell,\\
&|b_k(s)| \leq s^{-k} \quad \text{for}\;\; \ell + 1 \leq k \leq L,\\
&\|\vec q(s)\|^2_{\Hs^\Bbbk \times \Hs^{\Bbbk - 1}}  \leq Ks^{-(2L + 2(1 - \delta)(1 + \eta))},\\
&\|\vec q(s)\|^2_{\dot{H}^\sigma \times \dot H^{\sigma - 1}} \leq Ks^{-\frac{\ell(2\sigma - d)}{\ell - \gamma}}.
\end{align*}
\end{definition}

\begin{remark} Note from \eqref{eq:Ukbke} that the bounds given in Definition \ref{def:Skset} imply that for $\eta$ small enough, 
$$b_1(s) \sim \frac{c_1}{s}, \quad |b_k(s)| \lesssim |b_1(s)|^k,$$
hence, the choice of the initial data $(b(s_0), q(s_0))$ belongs in $\Sc_K(s_0)$ if $s_0$ is large enough.
\end{remark}

\begin{remark} Note from the coercive property given in Lemma \ref{lemm:coerEk}, the $\|\vec q(s)\|^2_{\Hs^\Bbbk \times \Hs^{\Bbbk - 1}}$ is controlled by the adapted Sobolev norm $\Es_\Bbbk = \|\vec{q}\|_{\Bbbk}^2$ defined in \eqref{def:normk}. 
\end{remark}

\begin{remark} The introduction of the high Sobolev norm $\Es_{\Bbbk}$ is reflected on the following relation:
\begin{equation}\label{eq:modeq}
\left|\frac{\lambda_s}{\lambda} + b_1\right| + \sum_{k = 1}^L \left|(b_k)_s + (k - \gamma)b_1b_{k} - b_{k+1} \right| \lesssim C(M)\sqrt{\Es_{\Bbbk}} + l.o.t, 
\end{equation}
which is computed thanks to the $(L+1)$ orthogonality conditions \eqref{eq:orthqPhiM} (see lemmas \ref{lemm:mod1} and \ref{lemm:mod2} below).
\end{remark}

\subsection{Existence of solutions trapped in $\Sc_K(s)$ and conclusion of Theorem \ref{Theo:1}.}

We claim the following proposition:
\begin{proposition}[Existence of solutions trapped in $\Sc_K(s)$] \label{prop:exist} There exists $K_1 \geq 1$ such that for $K \geq K_1$, there exists $s_{0,1}(K)$ such that for all $s_0 \geq s_{0,1}$, there exists initial data for the unstable modes 
$$(\Vc_2(s_0), \cdots, \Vc_\ell(s_0)) \in \left[-s_0^{-\frac{\eta}{2}(1 - \delta)},s_0^{-\frac{\eta}{2}(1 - \delta)}\right]^{\ell - 1},$$
such that the corresponding solution $(b(s),q(s)) \in \Sc_K(s)$ for all $s \geq s_0$.
\end{proposition}
Let us briefly give the proof of Proposition \ref{prop:exist}. Let us consider $K \geq 1$ and $s_0 \geq 1$ and $(b_1(s_0),\cdots, b_L(s_0), \vec q(s_0))$ as in Definition \ref{def:1}. We introduce the exit time 
$$s_* = s_*(b_1(s_0), \cdots, b_L(s_0), \vec q(s_0)) = \sup\{s \geq s_0 \; \text{such that}\; (b_1(s),\cdots, b_L(s), \vec q(s)) \in \Sc_K(s)\},$$
and assume that for any choice of 
$$(\Vc_2(s_0), \cdots, \Vc_\ell(s_0)) \in \left[-s_0^{-\frac{\eta}{2}(1 - \delta)},s_0^{-\frac{\eta}{2}(1 - \delta)}\right]^{\ell - 1},$$
the exit time $s_* < +\infty$ and look for a contradiction. By the definition of $\Sc_K(s_*)$, at least one of the inequalities in that definition is an equality. Owing the following proposition, this can happens only for the components $(\Vc_2(s_*), \cdots, \Vc_\ell(s_*))$. Precisely, we have the following result which is the heart of our analysis:

\begin{proposition}[Control of $(b_1(s), \cdots, b_L(s), \vec q(s))$ in $\Sc_K(s)$ by $(\Vc_2(s), \cdots, \Vc_\ell(s))$] \label{prop:redu} There exists $K_2 \geq 1$ such that for each $K \geq K_2$, there exists $s_{0,2}(K) \geq 1$ such that for all $s_0 \geq s_{0,2}(k)$, the following holds: Given the initial data at $s = s_0$ as in Definition \ref{def:1}, if $(b_1(s), \cdots, b_L(s), \vec q(s)) \in \Sc_K(s)$ for all $s \in [s_0, s_1]$, with $(b_1(s_1), \cdots, b_L(s_1), \vec q(s_1)) \in \partial \Sc_K(s_1)$ for some $s_1 \geq s_0$, then:\\
$(i)$ (Reduction to a finite dimensional problem)
$$(\Vc_2(s_1), \cdots, \Vc_\ell(s_1)) \in \partial\left[- \frac{K}{s_1^{\frac{\eta}{2}(1 - \delta)}}, \frac{K}{s_1^{\frac{\eta}{2}(1 - \delta)}} \right]^{\ell - 1}.$$
$(ii)$ (Transverse crossing)
$$\frac{d}{ds}\left(\sum_{i = 2}^\ell \left|s^{\frac{\eta}{2}(1 - \delta)}\Vc_i(s)\right|^2\right)_{\big|_{s = s_1} }> 0.$$
\end{proposition}

Let us assume Proposition \ref{prop:redu} and continue the proof of Proposition \ref{prop:exist}. From part $(i)$ of Proposition \ref{prop:redu}, we see that 
$$(\Vc_2(s_*), \cdots, \Vc_\ell(s_*)) \in \partial\left[- \frac{K}{s_*^{\frac{\eta}{2}(1 - \delta)}}, \frac{K}{s_*^{\frac{\eta}{2}(1 - \delta)}} \right]^{\ell - 1},$$
and the following mapping 
\begin{align*}
\Upsilon: [-1,1]^{\ell - 1} &\mapsto \partial \left([-1,1]^{\ell - 1}\right)\\
s_0^{\frac{\eta}{2}(1 - \delta)}\big(\Vc_2(s_0), \cdots, \Vc_\ell(s_0)\big) & \to \frac{s_*^{\frac{\eta}{2}(1 - \delta)}}{K} \big(\Vc_2(s_*), \cdots, \Vc_\ell(s_*)\big)
\end{align*}
is well defined. Applying the transverse crossing property given in part $(ii)$ of Proposition \ref{prop:redu}, we see that $(b_1(s), \cdots, b_L(s), \vec q(s))$ leaves $\Sc_K(s)$ at $s = s_0$, hence, $s_* = s_0$. This is a contradiction since $\Upsilon$ is the identity map on the boundary sphere and it can not be a continuous retraction of the unit ball. This concludes the proof of Proposition \ref{prop:exist}, assuming that Proposition \ref{prop:redu} holds.\\

\paragraph{Conclusion of Theorem \ref{Theo:1} assuming Proposition \ref{prop:redu}}. From Proposition \ref{prop:exist}, we know that there exist initial data $(b_1(s_0), \cdots, b_L(s_0), \vec q(s_0))$ with $s_0 \gg 1$ such that 
$$(b_1(s), \cdots, b_L(s), \vec q(s)) \in \Sc_K(s) \quad \text{for all} \quad s \geq s_0.$$ 
From \eqref{eq:lam10}, \eqref{eq:Lamdas}, we have 
\begin{align*}
-\lambda_t = c(\vec u_0)\lambda^{\frac{\ell - \gamma}{\ell}} \left[1+o(1)\right],
\end{align*}
which yields
$$-\lambda^{- \frac{\ell - \gamma}{\ell}}\lambda_t = c(\vec u_0)(1 + o(1)).$$
We easily conclude that $\lambda$ vanishes in finite time $T = T(\vec u_0) < +\infty$ with the following behavior near the blowup time:
$$\lambda(t) = c(\vec u_0)(1 + o(1))(T-t)^{\frac{\ell}{\gamma}},$$
which is the conclusion of item $(i)$ of Theorem \ref{Theo:1}. 

For the control of the Sobolev norms, we observe from \eqref{eq:interBound} and Definition \ref{def:Skset} that 
$$\|\vec q(s)\|^2_{\dot{H}^\Bbbk \times \dot{H}^{\Bbbk - 1}} \lesssim \Es_{\Bbbk}(s) \to 0 \quad \text{as}\;\;s \to +\infty,$$
$$\|\vec q(s)\|^2_{\dot{H}^\sigma \times \dot{H}^{\sigma - 1}}\to 0  \quad \text{as}\;\;s \to +\infty,$$
which concludes the proof of item $(ii)$ of Theorem \ref{Theo:1}. 

\section{Reduction of the problem to a finite dimensional one.} \label{sec:4}

In this section, we aim at proving Proposition \ref{prop:redu} which is the heart of our analysis. We proceed in three separate subsections:

- In the first subsection, we derive the laws for the parameters $(b_1, \cdots, b_L,\lambda)$ thanks to the orthogonality condition \eqref{eq:orthqPhiM} and the coercivity of the powers of $\Hs$.

- In the second subsection, we prove the main monotonicity tools for the control of the infinite dimensional part of the solution. In particular, we derive a suitable Lyapunov functional for the $\Es_{\Bbbk}$ energy as well as the monotonicity formula for the fractional Sobolev norm $\Es_\sigma$.

- In the third subsection, we conclude the proof of Proposition \ref{prop:redu} thanks to the identities obtained in the first two parts.

\subsection{Modulation equations.}
We derive here the modulation equations for $(b_1,\cdots, b_L, \lambda)$. The derivation is mainly based on the orthogonality \eqref{eq:orthqPhiM} and the coercivity of the powers of $\Hs$. Let us start with elementary estimates relating to the fixed direction $\vec \Phi_M$.
\begin{lemma}[Estimate for $\vec \Phi_M$]\label{lemm:estPhiM} Given $\vec \Phi_M$ as defined in \eqref{def:PhiM}, we have the followings:
\begin{equation}\label{eq:boundckM}
c_{2k + 1, M} = 0, \quad |c_{2k,M}| \lesssim M^{2k} \quad \text{for}\;\; 0 \leq k \leq \frac{L-1}{2},
\end{equation}
and 
\begin{equation}\label{eq:boundPhiM2}
\int |\vec \Phi_M|^2 \lesssim M^{d-2\gamma}, \quad \int |\Hs^{*k} \vec \Phi_M|^2 \lesssim M^C \quad \text{for}\; k \in \mathbb{N}.
\end{equation}
Moreover, we have the following orthogonality: 
\begin{equation}\label{eq:ortPhiMTj}
\big\langle \vec \Phi_M, \Hs^i \vec T_j \big\rangle = \big \langle \chi_M \Lambda \vec Q, \Lambda \vec Q \big \rangle\delta_{i,j}, \quad i \in \mathbb{N}, \; 1 \leq j \leq L.
\end{equation}

\begin{remark} Since the second coordinate of $\Phi_M$ is null, we write 
\begin{equation}\label{def:PhiM1}
\vec \Phi_M = \binom{\Phi_M}{0} \quad \text{with} \quad \int |\Phi_M|^2 \lesssim M^{d - 2\gamma}.
\end{equation}
\end{remark}

\end{lemma}
\begin{proof} Let us start with the proof of \eqref{eq:boundckM}. From definition \eqref{def:ckM}, \eqref{eq:formTk}, and the definition of $\Lambda \vec Q$ we have $$c_{1,M} = \frac{\big\langle \chi_M \Lambda \vec Q, \vec T_1\big \rangle}{\big\langle \chi_M \Lambda \vec Q, \Lambda \vec Q\big \rangle} = 0.$$
Arguing by induction, we assume that
\begin{equation*}
(P_k) \qquad c_{2j + 1, M} = 0, \quad |c_{2j,M}| \lesssim M^{2j}, \quad 0 \leq j \leq k,
\end{equation*}
and prove that $(P_{k+1})$ is true, namely that we prove 
$$c_{2k + 3,M} = 0, \quad |c_{2k + 2, M}| \lesssim M^{2k + 2}.$$
Indeed, by \eqref{def:ckM}, \eqref{def:H2k1} and \eqref{eq:formTk} we write 
\begin{align*}
c_{2k + 3, M} &= \frac{1}{\big\langle \chi_M \Lambda \vec Q, \Lambda \vec Q\big \rangle}\sum_{j = 0}^{k + 1}c_{2j,M} \big\langle \chi_M \Lambda \vec Q, \Hs^{2j}\vec T_{2k + 3}\big\rangle = 0.
\end{align*}
Similarly, we use $\big \langle \chi_M \Lambda \vec Q, \Lambda \vec Q \big \rangle \sim M^{d-2\gamma}$, the induction hypothesis and $(ii)$ of Lemma \ref{lemm:actionH} to estimate 
\begin{align*}
|c_{2k + 2, M}| &\lesssim  \frac{1}{M^{d - 2\gamma}} \sum_{j = 0}^k |c_{2j,M}|\left|\big \langle \chi_M \Lambda \vec Q, \Hs^{2j}\vec T_{2k + 2} \big \rangle \right|\\
&\lesssim \frac{1}{M^{d - 2\gamma}}\sum_{j = 0}^k M^{2j} \int_{y \leq M}y^{d-1}y^{-\gamma}y^{2(k+1 - j) - \gamma}dy \lesssim M^{2(k+1)}.
\end{align*}
Thus, the statement $(P_{k+1})$ holds true. 
Note from \eqref{def:H2k1} and \eqref{def:H2k1adj} that $\Hs^{*2j} = \Hs^{2j}$, we then estimate by using $(ii)$ of Lemma \ref{lemm:actionH},
\begin{align*}
\int |\vec \Phi_M|^2 &\lesssim \sum_{j = 0}^{\frac{L-1}{2}} \sum_{i = 0}^{\frac{L-1}{2}} |c_{2j,M}c_{2i,M}|\int|(\chi_M \Lambda Q)\Hs^{2(j+i)}(\chi_M \Lambda Q)|\\
&\lesssim \sum_{j = 0}^{\frac{L-1}{2}} \sum_{i = 0}^{\frac{L-1}{2}} M^{2(i + j)} \int_{y \leq M}y^{d-1}y^{-\gamma}y^{ - \gamma-2(j + i)}dy  \lesssim M^{d-2\gamma}.
\end{align*}
The estimate for $\int |\Hs^{*k} \vec \Phi_M|$ is obtained by a similar way. The orthogonality \eqref{eq:ortPhiMTj} is a direct consequence of \eqref{id:TkPhiM0}. This concludes the proof of Lemma \ref{lemm:estPhiM}.
\end{proof}

From the orthogonality conditions \eqref{eq:orthqPhiM} and equation \eqref{eq:qys}, we claim the following:
\begin{lemma}[Modulation equations] \label{lemm:mod1}  Given $\hbar$, $\delta$ and $\eta$ as defined in \eqref{def:kdeltaplus} and \eqref{def:B0B1}. For $K \geq 1$, we assume that there is $s_0(K) \gg 1$ such that $(b_1(s), \cdots, b_L(s), \vec q(s)) \in \Sc_K(s)$ for $s \in [s_0, s_1]$ for some $s_1 \geq s_0$. Then, the following estimates hold for $s \in [s_0, s_1]$:
\begin{equation}\label{eq:ODEbkl}
\sum_{k = 1}^{L-1}\left|(b_k)_s + (k - \gamma)b_1b_k - b_{k+1} \right| + \left|b_1 + \frac{\lambda_s}{\lambda}\right| \lesssim b_1^{L + 1 + (1 - \delta)(1 + \eta)},
\end{equation}
and 
\begin{equation}\label{eq:ODEbL}
\left|(b_L)_s + (L - \gamma)b_1b_L \right| \lesssim  C(M)\sqrt{\Es_{\Bbbk}} + b_1^{L + 1+ (1 - \delta)(1 + \eta)}.
\end{equation}
\end{lemma}
\begin{proof} Let 
$$D(t) = \left|b_1 + \frac{\lambda_s}{\lambda}\right| + \sum_{k = 1}^L \left|(b_k)_s + (k - \gamma)b_1b_k - b_{k+1}\right|,$$
where we use the convention $b_{k} \equiv 0$ if $k \geq L + 1$.

We take the scalar product of \eqref{eq:qys} with $\Hs^{*i} \vec \Phi_M, i = 0, \cdots, L$ and use the orthogonality \eqref{eq:orthqPhiM} to write
\begin{align}
&\left<\chi_{B_1}\vec{\textup{Mod}}, \Hs^{*i}\vec\Phi_M\right> - \left(\frac{\lambda_s}{\lambda} + b_1\right)\left<\Lambda \vec{\textbf{Q}}_b, \Hs^{*i}\vec \Phi_M \right> \nonumber\\
&\quad  = - \left<\Hs \vec q, \Hs^{*i}\vec \Phi_M\right>\delta_{i,L}  - \left< \vec{\mathbf{\Psi}}_b, \Hs^{*i}\vec \Phi_M \right> + \left<\frac{\lambda_s}{\lambda}\Lambda \vec q + \vec L(\vec q) - \vec N(\vec q), \Hs^{*i}\vec \Phi_M \right>.\label{eq:bL}
\end{align}
From the definition \eqref{def:PhiM}, we see that $\vec \Phi_M$ is localized in $y \leq 2M$. From definition \eqref{eq:Modt}, we compute the left hand side of \eqref{eq:bL} by using the identity \eqref{id:TkPhiMi}, 
\begin{align*}
&\left<\chi_{B_1}\vec{\textup{Mod}}, \Hs^{*i}\vec\Phi_M\right> - \left(\frac{\lambda_s}{\lambda} + b_1\right)\left<\Lambda \vec{\textbf{Q}}_b, \Hs^{*i}\vec \Phi_M \right>\\
& = (-1)^i\left<\Lambda \vec Q, \vec \Phi_M \right>\left(\left[(b_i)_s + (i - \gamma)b_1b_i - b_{i + 1}\right](1 - \delta_{0,i}) - \left(\frac{\lambda_s}{\lambda} + b_1\right)\delta_{0,i} \right) + \Oc(M^C b_1D(t)).
\end{align*}
We now estimate the terms on the right hand side of \eqref{eq:bL}. Recall that $L$ is an odd integer, we then use \eqref{def:H2k1}, the Cauchy-Schwartz inequality, \eqref{def:PhiM1} and \eqref{eq:qmbyE2k} to estimate
\begin{align*}
&\left|\left<\Hs \vec q, \Hs^{*L}\vec \Phi_M\right>\right| = \left|\left<\Hs^{L+1} \vec q, \vec \Phi_M\right>\right| = \left|\left<\Ls^{\frac{L+1}2} q_1, \Phi_M\right>\right|\\
& \quad \lesssim  M^{2\hbar}\left(\int \frac{|(q_1)_{L+1}|^2}{1 + y^{2\hbar}} \right)^\frac 12 \left(\int |\Phi_M|^2 \right)^\frac 12 \lesssim M^{\frac d2 - \gamma + 2\hbar}\sqrt{\Es_\Bbbk}.
\end{align*} 
The error term is estimated by using \eqref{eq:estlocalPsib} and\eqref{eq:boundPhiM2}, 
\begin{align*}
\left|\left<\Hs^i\vec{\mathbf{\Psi}}_b, \vec \Phi_M \right>\right| &\leq \left(\int_{y \leq 2M}|\Hs^i \vec{\mathbf{\Psi}_b}|^2 \right)^\frac 12 \left(\int_{y \leq 2M} |\vec \Phi_M|^2 \right)^\frac 12 \lesssim M^Cb_1^{L+3}.
\end{align*}
The remaining linear terms are estimated by using the following bound coming from \eqref{eq:qmbyE2k} and Lemma \ref{lemm:coerA},
\begin{align}
\int \frac{|q_1|^2}{y^4(1 + y^{2\Bbbk - 4})} + \int \frac{|q_2|^2}{y^4(1 + y^{2\Bbbk - 6})} +
\int \frac{|\py q_1|^2}{y^2(1 + y^{2\Bbbk - 4})} + \int \frac{|\py q_2|^2}{y^2(1 + y^{2\Bbbk - 6})}& \nonumber\\
\lesssim \int \frac{|(q_1)_1|^2}{y^2(1 + y^{2\Bbbk - 4})} + \int \frac{|(q_2)_1|^2}{y^2(1 + y^{2\Bbbk - 6})} &\quad \lesssim \Es_\Bbbk,\label{est:Es2K1}
\end{align}
from which and the Cauchy-Schwartz inequality and \eqref{eq:boundPhiM2}, we obtain
$$\left|\left<-\frac{\lambda_s}{\lambda} \Lambda \vec q + \vec L(\vec q), \Hs^{*i}\vec \Phi_M\right>\right| \lesssim M^Cb_1\left(\sqrt{\Es_{\Bbbk}} + D(t)\right). $$
Similarly, the nonlinear term $\vec N(\vec q) = \binom{0}{N(q_1)}$ is estimate by using \eqref{est:Es2K1} and the $L^\infty$ bound \eqref{eq:boundLinfq}, 
$$\left|\left<\vec N(\vec q), \Hs^{*i}\vec \Phi_M\right>\right| \lesssim M^Cb_1\left(\sqrt{\Es_{\Bbbk}} + D(t)\right).$$
Put all the above estimates into \eqref{eq:bL} and use \eqref{eq:PhiMLamQ} together the bootstrap bound on $\Es_\Bbbk$ given in Definition \ref{def:Skset}, we conclude the proof of Lemma \ref{lemm:mod1}.
\end{proof}

\medskip

From the bound for $\Es_{\Bbbk}$ given in Definition \ref{def:Skset} and the modulation equation \eqref{eq:ODEbL}, we only have the pointwise bound
$$|(b_L)_s + (L - \gamma)b_1b_L| \lesssim b_1^{L + (1 - \delta)(1 + \eta)},$$
which is not good enough to close the expected one
$$|(b_L)_s + (L - \gamma)b_1b_L| \ll b_1^{L + 1}.$$
We claim that the main linear term can be removed up to an oscillation in time leading to the improved modulation equation for $b_L$ as follows:

\begin{lemma}[Improved modulation equation for $b_L$] \label{lemm:mod2} Under the assumption of Lemma \ref{lemm:mod1}, the following bound holds for all $s \in [s_0, s_1]$:
\begin{align}
&\left|(b_L)_s + (L - \gamma)b_1b_L - \frac{d}{ds} \left\{ \frac{\left<\Hs^L \vec q, \chi_{B_0}\Lambda \vec Q\right>}{\left<\chi_{B_0}\Lambda Q,\Lambda Q  +  (-1)^{\frac{L - 1}2} \Ls^{\frac{L-1}2}\left(\frac{\partial S_{L+2}}{\partial b_L} \right)\right>}\right\}\right|\nonumber\\
& \qquad \qquad \qquad \lesssim  b_1^\delta \left[C(M)\sqrt{\Ec_{\Bbbk}} + b_1^{L + (1 - \delta)(2 + \eta)}\right].\label{eq:ODEbLimproved}
\end{align}
\end{lemma}
\begin{proof} We commute  \eqref{eq:qys} with $\Hs^L$ and take the scalar product with $\chi_{B_0}\Lambda \vec Q$ and write
\begin{equation}\label{eq:HLeq}
\frac{d}{ds} \left<\Hs^L \vec q, \chi_{B_0} \Lambda \vec Q\right> = \left<\Hs^L \vec q_s, \chi_{B_0} \Lambda \vec Q\right> + \left<\Hs^L \vec q, b_{1,s} y\py \chi_{B_0}\Lambda \vec Q \right>.
\end{equation}
Recall that $L\gg 1$ is an odd integer, we estimate the last term in \eqref{eq:HLeq} by using \eqref{eq:qmbyE2k} as follows:
\begin{align*}
&\left| \left<\Hs^L \vec q, b_{1,s} y\py \chi_{B_0}\Lambda \vec Q \right>\right| = \left|(-1)^{\frac{L - 1}{2}}\int \Ls^{\frac{L-1}{2}}q_2 y\py \chi_{B_0}\Lambda Q \right| \lesssim b_1^2 \int_{y \sim B_0}|(q_2)_{L-1}| y^{1 - \gamma}\\
& \qquad \lesssim b_1^2 \left(\int \frac{|(q_2)_{L-1}|^2}{1  + y^{2 + 2\hbar}} \right)^\frac{1}{2}\left(\int_{y \sim B_0}y^{2 + 2\hbar}y^{2- 2\gamma} \right)^\frac{1}{2}\lesssim b_1^2\sqrt{\Es_\Bbbk}B_0^{4 + 2\hbar - 2\gamma + d} = b_1^{-(2\hbar  + \delta)}\sqrt{\Es_\Bbbk}.
\end{align*}
For the second term on the right hand side of \eqref{eq:HLeq}, we write from \eqref{def:H2k1} and \eqref{eq:qys}, 
\begin{align}
&(-1)^{\frac{L - 1}{2}}\left<\Hs^L \vec q_s, \chi_{B_0} \Lambda \vec Q\right> = \int \Ls^{\frac{L-1}{2}}q_{2,s}\chi_{B_0}\Lambda Q \label{eq:Hsq2} \\
&= \int \chi_{B_0}\Lambda Q \Ls^{\frac{L -1}2}\left[  D q_2 - \Ls q_1 - (\mathbf{\Psi}_{b})_2 - \chi_{B_1}(\textup{Mod})_2 + \left(\frac{\lambda_s}{\lambda}  +b_1\right)D(\textbf{Q}_b)_2 + L(q_1) - N(q_1) \right].\nonumber
\end{align}
We now estimate all the terms of \eqref{eq:Hsq2}.\\
- The term $Dq_2$: we use \eqref{est:Es2K1} to estimate 
\begin{align*}
&\left|\frac{\lambda_s}{\lambda} \int \chi_{B_0}\Lambda Q \Ls^{\frac{L -1}2}Dq_2 \right| \lesssim b_1 \left|\int \Ls^{\frac{L - 1}2} (\chi_{B_0} \Lambda Q) (q_2 + y\py q_2)\right| \\
& \lesssim b_1\left(\int \left| \Ls^{\frac{L - 1}2} (\chi_{B_0} \Lambda Q)( 1 + y^{\Bbbk - 1})  \right|^2\right)^\frac{1}{2} \left[ \left(\int \frac{|q_2|^2}{1 + y^{2\Bbbk - 2}} \right)^\frac 12 +  \left(\int \frac{|\py q_2|^2}{1 + y^{2\Bbbk - 4}} \right)^\frac 12\right] \\
& \lesssim b_1 \sqrt{\Es_\Bbbk}\left(\int_{y \sim B_0} y^{-2\gamma + 2\hbar + 2}\right)^\frac 12 \lesssim b_1\sqrt{\Es_\Bbbk}B_0^{ -2\gamma + 2\hbar  + 2 + d}:= b_1^{-(2\hbar + \delta)} \sqrt{\Es_\Bbbk}.
\end{align*}
- The term $\Ls q_2$: we use \eqref{eq:qmbyE2k} to estimate
$$\left|\int \chi_{B_0}\Lambda Q \Ls^{\frac{L + 1}{2}} q_2\right| \lesssim \left(\int |\Ls (\chi_{B_0}\Lambda Q) (1 + y^{\hbar   +1})|^2 \right)^\frac{1}{2} \left(\int \frac{|(q_2)_{L-1}|}{1 + y^{2\hbar  + 2}} \right)^\frac{1}{2} \lesssim b_1^{-(2\hbar + \delta)} \sqrt{\Es_\Bbbk}.$$
- The error term $(\mathbf{\Psi}_b)_2$: we use \eqref{eq:estPsiblocalB0} with $m = L - \hbar - 1$ to estimate 
$$\left|\int \chi_{B_0}\Lambda Q \Ls^{\frac{L - 1}2} (\mathbf{\Psi}_{b})_2 \right| \lesssim \|\chi_{B_0} \Lambda Q \|_{L^2(y \leq 2B_0)} \|\Ls^{\frac{L - 1}2} (\mathbf{\Psi}_{b})_2\|_{L^2(y \leq 2B_0)} \lesssim b_1^{-(2\hbar + \delta)}b_1^{L + 1+ (1 - \delta) - C_L\eta}.$$ 
- The terms $L(q_1)$ and $N(q_1)$ are estimated similarly by using \eqref{est:Es2K1} and the $L^\infty$ bound \eqref{eq:boundLinfq} for the nonlinear term, which results in 
$$\left|\int \Ls^{\frac{L - 1}2}(\chi_{B_0}\Lambda Q) (L(q_1) - N(q_1))\right| \lesssim  b_1b_1^{-(2\hbar + \delta)}\sqrt{\Es_\Bbbk}.$$
- The terms $\chi_{B_1}\textup{Mod}_2$ and $D(\mathbf{Q}_b)_2$: By \eqref{eq:Modt}, we write 
\begin{align*}
& \int \chi_{B_0}\Lambda Q \Ls^{\frac{L - 1}2}(\chi_{B_1}\textup{Mod}_2  - \left(\frac{\lambda_s}{\lambda} + b_1\right) D(\mathbf{Q}_b)_2)\\
& = \int \Ls^{\frac{L-1}{2}}(\chi_{B_0}\Lambda Q) \left(\sum_{k = 1}^{L-1}\big[ b_{k, s} + (k - \gamma)b_1b_{k} - b_{k}\big] \left(T_k \delta_{k\wedge 2, 1} + \sum_{j = k + 1, \textup{odd}}^{L + 2} \frac{\partial S_j}{\partial b_k} \right) \right) \\
& \quad  +\big[b_{L,s}+ (L - \gamma)b_1b_L\big]\int \Ls^{\frac{L-1}{2}}(\chi_{B_0}\Lambda Q)\left(T_L + \frac{\partial S_{L +2}}{\partial b_L}\right)- \left(\frac{\lambda_s}{\lambda} + b_1\right)\int  \Ls^{\frac{L-1}{2}}(\chi_{B_0}\Lambda Q)D(\mathbf{Q}_b)_2.
\end{align*}
Note that $\Ls^{\frac{L - 1}{2}}T_k = 0$ for $k < L$ and $\Ls^{\frac{L-1}{2}}T_L = (-1)^{\frac{L - 1}{2}}\Lambda Q$. We then use the admissibility of $\vec T_k$ and the homogeneity of $\vec S_k$ and Lemma \ref{lemm:mod1} to estimate
\begin{align*}
&\left|\int \Ls^{\frac{L-1}{2}}(\chi_{B_0}\Lambda Q) \left(\sum_{k = 1}^{L-1}\big[ b_{k, s} + (k - \gamma)b_1b_{k} - b_{k}\big] \left(T_k \delta_{k\wedge 2, 1} + \sum_{j = k + 1, \textup{odd}}^{L + 2} \frac{\partial S_j}{\partial b_k} \right) \right) \right|\\
&\quad + \left|\left(\frac{\lambda_s}{\lambda} + b_1\right)\int  \Ls^{\frac{L-1}{2}}(\chi_{B_0}\Lambda Q)D(\mathbf{Q}_b)_2 \right|\\
& \lesssim b_1^{L+1 +(1 - \delta)(1 + \eta)}\left\{\sum_{k = 1}^{L - 1}\sum_{j = k+1, \textup{odd}}^{L+2} \left|\int \chi_{B_0}\Lambda Q \Ls^{\frac{L-1}{2}}\left(\frac{\partial S_j}{\partial b_k}\right)\right| + \left|\int  (\chi_{B_0}\Lambda Q) \Ls^{\frac{L-1}{2}}(D(\Theta_b)_2) \right|  \right\}\\
&\lesssim b_1^{L+1 +(1 - \delta)(1 + \eta)} \left\{\sum_{k = 1}^{L - 1}\sum_{j = k+1, \textup{odd}}^{L+2} b_1^{j - k}\int_{y \leq 2B_0}y^{-2\gamma + j - L -1} + b_1^L\int_{y \leq 2B_0}y^{-2\gamma} + \sum_{k = 2, \textup{odd}}^{L+2} b_1^k\int_{y \leq 2B_0}y^{k - 2\gamma - 1} \right\}\\
&\lesssim b_1^{L+1 +(1 - \delta)(1 + \eta)}\left\{b_1^{-2\hbar -2\delta} + b_1^{L - 2\hbar  - 2\delta} + b_1^{1 - 2\hbar - 2\delta}\right\} \quad \lesssim b_1^{-(2\hbar +\delta)}b_1^{L + (1 - \delta)(2 + \eta)}.
\end{align*}
We also write 
\begin{align*}
&\big[b_{L,s}+ (L - \gamma)b_1b_L\big]\int \Ls^{\frac{L-1}{2}}(\chi_{B_0}\Lambda Q)\left(T_L + \frac{\partial S_{L +2}}{\partial b_L}\right)\\
& = \big[b_{L,s}+ (L - \gamma)b_1b_L\big]\left<\chi_{B_0}\Lambda Q, (-1)^{\frac{L-1}{2}}\Lambda Q + \Ls^{\frac{L-1}2}\left(\frac{\partial S_{L +2}}{\partial b_L}\right) \right>.
\end{align*}
Injecting all the above estimates into \eqref{eq:HLeq} yields
\begin{align*}
\frac{d}{ds}\left<\Hs^L \vec q, \chi_{B_0}\Lambda Q\right> &= \big[b_{L,s} + (L -\gamma)b_1b_L\big]G(s) + b_1^{-(2\hbar + \delta)}\Oc\left(\sqrt{\Es_{\Bbbk}} + b_1^{L + (1 - \delta)(2 + \eta)}\right),
\end{align*}
where we write for short
\begin{equation}\label{def:Gs}
G(s) = \left<\chi_{B_0}\Lambda Q,\Lambda Q  +  (-1)^{\frac{L - 1}2} \Ls^{\frac{L-1}2}\left(\frac{\partial S_{L+2}}{\partial b_L} \right)\right> \sim b_1^{-2\hbar - 2\delta}.
\end{equation}
Thus, we have
\begin{align*}
&\frac{d}{ds} \left[ \frac{\left<\Hs^L \vec q, \chi_{B_0}\Lambda \vec Q\right>}{G(s)}\right] - \big[b_{L,s} + (L -\gamma)b_1b_L\big]\\
& =\Oc\left( b_1^{4\hbar + 4\delta}\left|\left<\Hs^L \vec q, \chi_{B_0}\Lambda \vec Q\right> \frac{d}{ds}G(s)\right|\right) +  b_1^\delta \Oc\left(\sqrt{\Es_{\Bbbk}} + b_1^{L + (1 - \delta)(2 + \eta)}\right).
\end{align*}
From \eqref{def:H2k1} and \eqref{eq:qmbyE2k}, we estimate 
\begin{equation}\label{eq:HsLchiB0}
\left|\left<\Hs^L \vec q, \chi_{B_0}\Lambda \vec Q\right> \right| = \left|\int\chi_{B_0}\Lambda Q \Ls^{\frac{L-1}2}(q_2)\right| \lesssim \sqrt{\Es_\Bbbk}\left(\int_{y \sim B_0}y^{-2\gamma + 2 + 2\hbar} \right)^\frac{1}{2} \lesssim b_1^{-2\hbar -\delta - 1}\sqrt{\Es_\Bbbk}.
\end{equation}
Note that $\left|\frac{d}{ds}\chi_{B_0}\right| \lesssim b_1$ and has support on $B_0 \leq y \leq 2B_0$, and that $\frac{\partial S_{L+2}}{\partial b_L}$ does not depend on $b_L$, we can use the bounds on $b_1, \cdots, b_{L-1}$ given in Lemma \ref{lemm:mod1} to obtain the estimate
$$\left|\frac{d}{ds}G(s)\right| \lesssim b_1b_1^{-2\hbar -2\delta}.$$
Hence, we obtain 
$$\frac{d}{ds} \left[ \frac{\left<\Hs^L \vec q, \chi_{B_0}\Lambda \vec Q\right>}{G(s)}\right] - \big[b_{L,s} + (L -\gamma)b_1b_L\big] = b_1^\delta \Oc\left(\sqrt{\Es_{\Bbbk}} + b_1^{L + (1 - \delta)(2 + \eta)}\right),$$
which concludes the proof of Lemma \ref{lemm:mod2}.
\end{proof}

\subsection{Monotonicity for $\Es_\Bbbk$.}

We derive in this subsection the main monotonocity formula for $\Es_{\Bbbk}$. We claim the following which is the heart of this paper:

\begin{proposition}[Lyapunov monotonicity for $\Es_{\Bbbk}$] \label{prop:E2k}Given $\hbar$, $\delta$ and $\eta$ as defined in \eqref{def:kdeltaplus} and \eqref{def:B0B1}. For $K \geq 1$, we assume that there is $s_0(K) \gg 1$ such that $(b_1(s), \cdots, b_L(s), \vec q(s)) \in \Sc_K(s)$ for $s \in [s_0, s_1]$ for some $s_1 \geq s_0$. Then, the following estimate holds for $s \in [s_0, s_1]$:
\begin{align}
&\frac{d}{dt}\left\{\frac{\Es_{\Bbbk}}{\lambda^{2\Bbbk - d}}\left[1 + \Oc\left(b_1^{\eta(1-\delta)}\right) \right]\right\}\nonumber\\
&\quad \lesssim \frac{b_1}{\lambda^{2\Bbbk - d + 1}}\left[ \sqrt{\Es_{\Bbbk}} b_1^{L + (1- \delta)(1 + \eta)} + \Es_\Bbbk \Es_\sigma^{\frac{1}{2} +\Oc\left(\frac 1L\right)}+ \Es_\Bbbk b_1^{\eta(1-\delta)} +  \frac{\Es_\Bbbk}{N^{2\gamma - 1}} + C(N)\Es_{\Bbbk, \textup{loc}} \right], \label{eq:Es2sLya}
\end{align}
where 
\begin{equation}\label{def:Ekloc}
\Es_{\Bbbk, \textup{loc}} = \int_{y \leq N}|(q_1)_\Bbbk|^2 + \int_{y \leq N}|(q_2)_{\Bbbk - 1}|^2,
\end{equation}
with $N \gg 1$ being a fixed constant.
\end{proposition}

\begin{proof} By the definition of $\Es_\Bbbk$, \eqref{eq:reLLlambda} and equation \eqref{eq:vrt}, we write 
\begin{align}
\frac{d}{dt}\left[\frac{\Es_{\Bbbk}}{\lambda^{2\Bbbk - d}}\right] & = \frac{d}{dt}\left[\frac{1}{\lambda^{2\Bbbk - d}}\int q_1\Ls^\Bbbk q_1 + \int q_2\Ls^{\Bbbk - 1}q_2 \right] = \frac{d}{dt}\left[\int v_1 \Ls_\lambda^{\Bbbk}v_1 + \int v_2\Ls_\lambda^{\Bbbk- 1}v_2 \right]\nonumber\\
& = 2\int \pt v_1\Ls_\lambda^{\Bbbk}v_1 + 2\int \pt v_2 \Ls_\lambda^{\Bbbk - 1}v_2 + \int v_1 [\pt, \Ls_\lambda^\Bbbk]v_1 + \int v_2 [\pt, \Ls_\lambda^{\Bbbk-1}]v_2\nonumber\\
& = 2\int \frac{1}{\lambda}(\Fc_1)_\lambda \Ls_\lambda^\Bbbk v_1 + 2 \int \frac{1}{\lambda^2}(\Fc_2)_\lambda\Ls_\lambda^{\Bbbk - 1}v_2 + \int v_1 [\pt, \Ls_\lambda^\Bbbk]v_1 + \int v_2 [\pt, \Ls_\lambda^{\Bbbk-1}]v_2, \label{eq:IDEk}
\end{align}
where the commutator is defined by
\begin{align}
[\pt, \Ls_\lambda^k]f &= \pt(\Ls_\lambda^k f) - \Ls^k_\lambda (\pt f)\nonumber\\ 
& = \sum_{m = 0}^{k - 1}\Ls_\lambda^m\left([\pt, \Ls_\lambda] \Ls_\lambda^{k - 1 - m}f\right) = \sum_{m = 0}^{k - 1} \Ls_\lambda^m \left(\frac{\pt Z_\lambda}{r^2}\Ls_\lambda^{k - 1- m}f\right), \label{def:commu}
\end{align}
and we recall from \eqref{eq:qys},
\begin{align}
\Fc_1 = -(\mathbf{\Psi}_b)_1 - \mathbf{M}_1, \quad \mathbf{M}_1 = \chi_{B_1}(\textup{Mod})_1 - \left(\frac{\lambda_s}{\lambda}  + b_1\right)\Lambda (\mathbf{Q}_b)_1,\label{def:M1}\\
\Fc_2 = -(\mathbf{\Psi}_b)_2 - \mathbf{M}_2 +  L(q_1) - N(q_1), \quad \mathbf{M}_2 = \chi_{B_1}(\textup{Mod})_2 - \left(\frac{\lambda_s}{\lambda}  + b_1\right)D(\mathbf{Q}_b)_2.\label{def:M2}
\end{align}
$\bullet$ \textbf{The error term $\vec{\mathbf{\Psi}}_b$:} we use \eqref{eq:estPsibLarge2} to estimate 
\begin{align}
&\left|\int \frac{1}{\lambda}\big((\mathbf{\Psi}_b)_1\big)_\lambda \Ls_\lambda^\Bbbk v_1 + \int \frac{1}{\lambda^2}\big((\mathbf{\Psi}_b)_1\big)_\lambda\Ls_\lambda^{\Bbbk - 1}v_2\right|\nonumber\\
&\quad = \frac{1}{\lambda^{2\Bbbk - d + 1}}\left|\int(\mathbf{\Psi}_b)_1 \Ls^{\Bbbk}q_1 +\int(\mathbf{\Psi}_b)_2 \Ls^{\Bbbk - 1}q_2  \right|\nonumber\\
& \qquad = \frac{1}{\lambda^{2\Bbbk - d + 1}}\left|\int\big((\mathbf{\Psi}_b)_1\big)_{\Bbbk} (q_1)_\Bbbk +\int\big((\mathbf{\Psi}_b)_2\big)_{\Bbbk - 1} (q_2)_{\Bbbk - 1}\right|\nonumber\\
& \qquad \quad \lesssim \frac{1}{\lambda^{2\Bbbk - d + 1}}\sqrt{\Es_\Bbbk}b_1^{L + 1 + (1 -\delta)(1 + \eta)}.\label{est:Psi}
\end{align}
$\bullet$ \textbf{The nonlinear term $N(q_1)$:} we write 
\begin{equation}\label{est:Nq1}
\left|\int \frac{1}{\lambda^2}\big(N(q_1)\big)_\lambda\Ls_\lambda^{\Bbbk - 1}v_2\right| = \frac{1}{\lambda^{2\Bbbk - d  +1}}\left|\int N(q_1)\Ls^{\Bbbk - 1}q_2\right| \lesssim \frac{\sqrt{\Es_\Bbbk}}{\lambda^{2\Bbbk -d + 1 }}\|(N(q_1))_{\Bbbk - 1} \|_{L^2}.
\end{equation}
- \textit{Estimate for $y < 1$:} By \eqref{def:Nq}, we can write
$$N(q_1) = \frac{q_1^2}{y}\Phi \quad \text{with} \quad \Phi = -\frac{d-1}{y}\int_0^1(1 - \tau)\sin\big((2\mathbf{Q}_b)_1 + 2\tau q_1\big) d\tau.$$
Using the expansion \eqref{eq:expqat0} of $\vec q$ near the origin, we write
\begin{equation}\label{eq:q1y2}
\frac{q_1^2}{y} = \frac{1}{y} \left(\sum_{i = 0, \textup{even}}^{\Bbbk - 1}c_iT_i + r_1\right)^2 = \sum_{i = 0, \textup{even}}^{\Bbbk - 1}\tilde c_iy^{i + 1} + \tilde r_1,
\end{equation}
where 
$$|\tilde{c}_i| \lesssim \Es_\Bbbk \quad \text{and} \quad \sum_{j = 0}^{\Bbbk - 1}y^j|\py^j \tilde{r}_1| \lesssim y^{\Bbbk - \frac{d}{2}}\Es_\Bbbk \quad \text{for}\quad  y < 1.$$ 
Let $\tau \in [0,1]$ and 
$$v_\tau = (\mathbf{Q}_b)_1 + \tau q_1.$$
We obtain from Proposition \ref{prop:1} and \eqref{eq:expqat0} the expansion
$$v_\tau = \sum_{i = 0, \textup{even}}^{\Bbbk - 1} \hat c_i y^{i  +1} + \hat r_1,$$
where
$$|\hat c_i| \lesssim 1 \quad \text{and} \quad \sum_{j = 0}^{\Bbbk - 1}|y^j \py^j \hat r_1| \lesssim y^{\Bbbk - \frac{d}{2}} \quad \text{for} \quad y < 1.$$
By the Taylor expansion of $\sin(x)$ at $x = 0$, we write
\begin{equation}\label{eq:Phiq}
\Phi = \sum_{i = 0, \textup{even}}^{\Bbbk - 1} \bar c_i y^{i} + \bar r_1,
\end{equation}
where
$$|\bar c_i| \lesssim 1 \quad \text{and} \quad \sum_{j = 0}^{\Bbbk - 1}|y^j \py^j \bar r_1| \lesssim y^{\Bbbk - \frac{d}{2} - 1} \quad \text{for} \quad y < 1.$$
Thus, we can write from \eqref{eq:q1y2} and \eqref{eq:Phiq} the expansion of $N(q_1)$ near the origin as follows:
\begin{equation}\label{eq:expanNq1at0}
N(q_1) = \sum_{i = 0, \textup{even}}^{\Bbbk -1}c_i'y^{i + 1} + r'_1,
\end{equation}
where
$$|c'_i| \lesssim \Es_\Bbbk \quad \text{and} \quad \sum_{j = 0}^{\Bbbk - 1}|y^j \py^j r'_1| \lesssim \Es_\Bbbk y^{\Bbbk - \frac{d}{2}} \quad \text{for} \quad y < 1.$$
From the definition of $\As$ and $\As^*$, one can check that 
$$|(r'_1)_{\Bbbk - 1}| \lesssim \sum_{i = 0}^{\Bbbk - 1}\frac{\py^i r_1'}{y^{\Bbbk - 1 - i}} \lesssim \Es_\Bbbk y^{-\frac{d}{2}  +1} \quad \text{for}\;\; y < 1.$$
Using the fact that $\As(y) = \Oc(y^2)$ for $y < 1$, we obtain 
$$\left|\left(\sum_{i = 0, \textup{even}}^{\Bbbk -1}c_i'y^{i + 1} \right)_{\Bbbk - 1}\right| \lesssim y^2 \Es_\Bbbk.$$
Hence, we derive the estimate
\begin{equation}\label{eq:controlNqi0}
\|\big(N(q_1)\big)_{\Bbbk - 1}\|_{L^2(y < 1)} \lesssim \Es_\Bbbk.
\end{equation}

\noindent - \textit{Estimate for $y > 1$:} Let us rewrite from the definition \eqref{def:Nq} of $N(q_1)$,
\begin{equation}\label{def:Zpsi}
N(q_1) = Z^2 \psi, \quad Z = \frac{q_1}{y}, \quad \psi = -(d-1)\int_0^1(1 - \tau)\sin(2(\mathbf{Q}_b)_1 + 2\tau q_1) d\tau.
\end{equation}
Note from the definitions of $\As$ and $\As^*$ that 
$$\forall k \in \mathbb{N}, \quad |f_k| \lesssim \sum_{i = 0}^{k} \frac{|\py^i f|}{y^{k- i}},$$
from which and the Leibniz rule, we write 
\begin{align*}
\int_{y \geq 1}\left|\big(N(q_1)\big)_{\Bbbk - 1}\right|^2 &\lesssim \sum_{k = 0}^{\Bbbk - 1}\int_{y \geq 1}\frac{|\py^kN(q_1)|^2}{y^{2\Bbbk - 2k - 2}}\\
&\lesssim \sum_{k = 0}^{\Bbbk - 1}\sum_{i = 0}^k \int_{y \geq 1} \frac{|\py^i Z^2|^2|\py^{k - i}\psi|^2}{y^{2\Bbbk - 2k - 2}}\\
&\lesssim \sum_{k = 0}^{\Bbbk - 1}\sum_{i = 0}^k \sum_{m = 0}^i \int_{y \geq 1} \frac{|\py^m Z|^2 |\py^{i - m}Z|^2 |\py^{k - i}\psi|^2}{y^{2\Bbbk - 2k - 2}}.
\end{align*}
We aim at using \eqref{eq:weipykq} and \eqref{eq:boundLinfq} to prove that for $0 \leq k \leq \Bbbk - 1$, $0 \leq i \leq k$ and $0 \leq m \leq i$,
\begin{equation}\label{def:Akmi}
A_{k,i,m} := \int_{y \geq 1} \frac{|\py^m Z|^2 |\py^{i - m}Z|^2 |\py^{k - i}\psi|^2}{y^{2\Bbbk - 2k - 2}} \lesssim b_1^{2}\Es_\Bbbk\Es_\sigma^{1 + \Oc\left(\frac{1}{L} \right)}.
\end{equation}
from which and \eqref{eq:controlNqi0}, we derive the estimate 
\begin{equation}\label{eq:controlNq1}
\|\big(N(q_1)\big)_{\Bbbk - 1}\|_{L^2} \lesssim b_1\sqrt{\Es_\Bbbk}\Es_\sigma^{\frac 12 + \Oc\left(\frac{1}{L} \right)}.
\end{equation}
Let us prove \eqref{def:Akmi}. We distinguish in 3 cases:\\
\noindent - \underline{\textit{The initial case $k = 0$, then $m = i = k = 0$.}}  From \eqref{def:Zpsi}, it is obvious to see that $|\psi|$ is uniformly bounded.  We estimate from \eqref{eq:weipykq} and \eqref{eq:boundLinfq1}, 
\begin{align*}
A_{0,0,0}  = \int_{y \geq 1} \frac{|q_1|^4|\psi|^2}{y^{2\Bbbk + 2}} y^{d-1}dy  \lesssim  \left\| \frac{q_1}{y}\right\|^2_{L^\infty(y \geq 1)} \int_{y \geq 1} \frac{|q_1|^2}{y^{2\Bbbk}} \lesssim b_1^2\Es_\Bbbk\Es_\sigma^{1 + \Oc\left(\frac{1}{L} \right)}.
\end{align*} 
\noindent  - \underline{\textit{Case $k \geq 1$ and $k - i = 0$}}. We first use the Leibniz rule to write
\begin{equation}\label{eq:pylZexp}
\forall l \in \mathbb{N}, \quad |\py^l Z|^2 \lesssim \sum_{j = 0}^l \frac{|\py^j q_1|^2}{y^{2 + 2l - 2j}},
\end{equation}
from which and the uniform bound of $\psi$, we have
\begin{align*}
A_{k,k,m} &\lesssim \sum_{j = 0}^m \sum_{l = 0}^{k-m} \int_{y \geq 1} \frac{|\py^j q_1|^2 |\py^l q_1|^2}{y^{2 \Bbbk - 2j - 2l + 2}} = \sum_{j = 0}^m \sum_{l = 0}^{k-m} B_{j,l,0}.
\end{align*}
where 
$$B_{j,l,0} = \int_{y \geq 1} \frac{|\py^j q_1|^2 |\py^l q_1|^2}{y^{2 \Bbbk - 2j - 2l + 2}} \quad \text{for} \quad 0 \leq j +l \leq \Bbbk -1.$$
We consider two cases: \\
- If $0 \leq j \leq \frac{L + 1}{2}$, we estimate from \eqref{eq:boundLinfq1} and  \eqref{eq:interBound}, 
\begin{align*}
B_{j,l,0}  &\lesssim \left\|\frac{\py^j q_1}{y} \right\|^2_{L^\infty(y \geq 1)}\int_{y \geq 1} \frac{|\py^lq_1|^2}{y^{2\Bbbk  -2l - 2j}} \\
&\quad \lesssim \Es_\sigma^{1 + \Oc\left(\frac{1}{L} \right)}b_1^{2j + 2 + \frac{2\gamma(j + 1)}{L} + \Oc\left(\frac{1}{L^2}\right)}\Es_{\sigma}^\frac{j}{\Bbbk - \sigma} \Es_{\Bbbk}^{1 - \frac{j}{\Bbbk - \sigma}} \lesssim b_1^2\Es_\Bbbk\Es_\sigma^{1 + \Oc\left(\frac{1}{L} \right)},
\end{align*}
where we used the following fact
$$2j + \frac{2\gamma(j + 1)}{L} - \frac{j(2L + 2(1 - \delta)(1 + \eta))}{\Bbbk - \sigma} = \frac{2\gamma}{L} + \Oc\left(\frac{1}{L^2}\right) > 0.$$
- If $j \geq \frac{L + 1}{2} + 1$, then $l \leq \Bbbk - 1 - j \leq \frac{L - 3}{2} + \hbar$. We simply change the role of $j$ and $l$ in the above estimate resulting in the same estimate.\\

\noindent  - \underline{\textit{Case $k \geq 1$ and $k - i \geq 1$.}} Let us write from \eqref{def:Akmi} and \eqref{eq:pylZexp}, 
\begin{equation}\label{eq:Akmi11}
A_{k,m,i} \lesssim \sum_{j = 0}^{m}\sum_{l = 0}^{i - m} \int_{y \geq 1} \frac{|\py^j q_1|^2 |\py^l q_1|^2}{y^{2\Bbbk - 2j - 2l + 2}} \frac{|\py^{k - i} \psi|^2}{y^{-2(k - i)}}.
\end{equation}
At this stage, we need to precise the decay of $|\py^n \psi|$ to archive the bound \eqref{def:Akmi}. To do so, let us recall that $\vec T_i$ is admissible of degree $(i,i, i\wedge2)$ (see Lemma \ref{lemm:GenLk}) and $\vec S_i$ is homogeneous of degree $(i,i-1,i \wedge 2, i)$ (see Proposition \ref{prop:1}). We estimate for $j \geq 1$ and $y \geq 1$,
\begin{align*}
|\py^j (\mathbf{Q}_b)_1 | &= \left|\py^j \left(\sum_{i = 0}^{\frac{L - 1}{2}}b_{2i}T_{2i} + \sum_{i = 1}^{\frac{L + 1}{2}}S_{2i}\right) \right|\\
& \lesssim \frac{1}{y^{\gamma + j}} + \sum_{i = 0}^{\frac{L - 1}{2}}\frac{b_1^{2i} y^{2i}}{y^{\gamma + j}} \mathbf{1}_{\{y \leq 2B_1\}} \lesssim \frac{b_1^{-C\eta}}{y^{\gamma + j}}.
\end{align*}
Let $\tau \in [0,1]$ and $v_\tau = (\mathbf{Q}_b)_1 + \tau q_1$. We use the Faa di Bruno formula to write for $1 \leq n \leq \Bbbk - 1$,
\begin{align}
|\py^n \psi|^2 &\lesssim \int_0^1 \sum_{m^* = n} |\partial_{v_\tau}^{m_1 + \cdots+m_n}\sin(v_\tau)|^2  \prod_{i = 1}^n |\py^i (\mathbf{Q}_b)_1 + \py^iq_1|^{2m_i} d\tau \nonumber\\
&\quad \lesssim \sum_{|m|_2 = n}  \prod_{i = 1}^n \left(\frac{b_1^{-C(L)\eta}}{y^{2\gamma + 2i}} +  |\py^iq_1|^2\right)^{m_i}, \quad |m|_2 = \sum_{i = 1}^n im_i. \label{est:psiq1}
\end{align}
Hence, we need to estimate terms of the form 
\begin{equation}
B_{i,j,n}:= \int \frac{|\py^j q_1|^2 |\py^l q_1|^2}{y^{2\Bbbk - 2j - 2l + 2 - 2n}}\prod_{i = 1}^n \left(\frac{b_1^{-C(L)\eta}}{y^{2\gamma + 2i}} +  |\py^iq_1|^2\right)^{m_i},
\end{equation}
where $(j,l,n) \in \mathbb{N} \times \mathbb{N} \times \mathbb{N}^*$ and $m_i \in \{0,1, \cdots, n\}$ satisfying  
$$1 \leq j +l +n \leq \Bbbk - 1, \quad |m|_2 = \sum_{i = 1}^n im_i = n.$$
We consider two cases:\\
- Case 1: $m_i = 0$ for $L - 2 \leq i \leq n$ (if $n < L-2$ then we are in this case as well). We now use \eqref{eq:boundLinfq2} with $p = 0$ to estimate 
\begin{align*}
&\prod_{i = 1}^n \left(\frac{b_1^{-C(L)\eta}}{y^{2\gamma + 2i}} +  |\py^iq_1|^2\right)^{m_i} \lesssim \prod_{i = 1}^n \left(\frac{b_1^{-C(L)\eta}}{y^{2\gamma + 2i}} +  b_1^{2i}\right)^{m_i}\\
& \lesssim \prod_{i = 1}^n\left(\frac{b_1^{-C(L)\eta}}{y^{2\gamma + 2i}} +  \frac{1}{y^{2i}}\right)^{m_i}\mathbf{1}_{\{1 \leq y \leq B_0\}} + \prod_{i = 1}^n \left(b_1^{-C\eta + 2\gamma + 2i} + b_1^{2i} \right)^{m_i}\mathbf{1}_{\{y \geq B_0\}}\\
& \lesssim \frac{1}{y^{2n}}\mathbf{1}_{\{1 \leq y \leq B_0\}} + b_1^{2n }\mathbf{1}_{\{y \geq B_0\}}.
\end{align*}
Thus, we have 
\begin{align*}
B_{j,l,n} \lesssim \int_{1 \leq y \leq B_0} \frac{|\py^j q_1|^2 |\py^l q_1|^2}{y^{2\Bbbk - 2j - 2l + 2 }} +  b_1^{2n}\int_{y \geq B_0} \frac{|\py^j q_1|^2 |\py^l q_1|^2}{y^{2\Bbbk - 2j - 2l + 2 - 2n}}
\end{align*}
By the similar estimate as for $B_{i,l,0}$, we derive the bound
$$B_{j,l,n} \lesssim b_1^2\Es_\Bbbk\Es_\sigma^{1 + \Oc\left(\frac{1}{L} \right)}.$$
- Case 2: there exists $i_* \in \{L - 2,\cdots,  n\}$ ($n \leq L - 3$, this case does not occur) such that $m_{i_*} \ne 0$. Since $L \gg 1$ and the fact that $$0\leq \sum_{i = 1, i \ne i_*}^n i m_i = n - i_*m_{i_*} \leq \Bbbk - 1 - (L-2)m_{i_*},$$
we deduce that
$$m_{i_*} = 1, \quad m_{i} = 0 \quad \text{for}\quad \hbar + 2 \leq i \ne i_* \leq n.$$
We then write 
\begin{align*}
&\prod_{i = 1}^n \left(\frac{b_1^{-C(L)\eta}}{y^{2\gamma + 2i}} +  |\py^iq_1|^2\right)^{m_i} = \left(\frac{b_1^{-C(L)\eta}}{y^{2\gamma + 2i_*}} +  |\py^{i_*}q_1|^2\right)\prod_{i = 1}^{\hbar + 2} \left(\frac{b_1^{-C(L)\eta}}{y^{2\gamma + 2i}} +  |\py^iq_1|^2\right)^{m_i} \\
&\lesssim \left(\frac{b_1^{-C(L)\eta}}{y^{2\gamma + 2i_*}} +  |\py^{i_*}q_1|^2\right)\prod_{i = 1}^{\hbar + 2} \left(\frac{b_1^{-C(L)\eta}}{y^{2\gamma + 2i}} +  b_1^{2i}\right)^{m_i}\\
&\lesssim \left(\frac{b_1^{-C\eta}}{y^{2\gamma + 2n}} + \frac{|\py^{i_*} q_1|^2}{y^{2n - 2i_* }}\right)\mathbf{1}_{\{1 \leq y \leq B_0\}} + \left(b_1^{2n + 2\gamma - C\eta} +  b_1^{2n - 2i^* } |\py^{i_*}q_1|^2  \right) \mathbf{1}_{\{y \geq B_0\}}.
\end{align*}
Thus, we have 
\begin{align*}
B_{j,l,n} &\lesssim \int_{1 \leq y \leq B_0}\frac{|\py^j q_1|^2 |\py^l q_1|^2}{y^{2\Bbbk - 2j - 2l + 2}}\left(\frac{b_1^{-C\eta}}{y^{2\gamma}} + \frac{|\py^{i_*} q_1|^2}{y^{- 2i_*}}\right)\\
& +b_1^{2n}\int_{y \geq B_0}\frac{|\py^j q_1|^2 |\py^l q_1|^2}{y^{2\Bbbk - 2j - 2l + 2 - 2n}} \left(b_1^{2\gamma - C\eta} +  b_1^{- 2i^* } |\py^{i_*}q_1|^2  \right).
\end{align*}
Since $i_* \geq L - 3$, we can use the interpolation bound \eqref{eq:interBound} to control $\int |\py^{i^*}q_1|^2$ directly, then the rest terms are controlled by the $L^\infty$ bound \eqref{eq:boundLinfq1} resulting in 
$$B_{j,l,n} \lesssim b_1^2 \Es_\Bbbk\Es_\sigma^{1 + \Oc\left(\frac{1}{L} \right)}.$$
This concludes the proof of \eqref{def:Akmi} as well as \eqref{eq:controlNq1}. \\

\noindent $\bullet$ \textbf{The small linear term $L(q_1)$:} we write 
$$\left|\int \frac{1}{\lambda^2}\big(L(q_1)\big)_\lambda\Ls_\lambda^{\Bbbk - 1}v_2\right| = \frac{1}{\lambda^{2\Bbbk - d  +1}}\left|\int L(q_1)\Ls^{\Bbbk - 1}q_2\right| \lesssim \frac{\sqrt{\Es_\Bbbk}}{\lambda^{2\Bbbk -d + 1 }}\|(L(q_1))_{\Bbbk - 1} \|_{L^2}.$$
We claim that 
\begin{equation}\label{eq:controlLq1}
\|(L(q_1))_{\Bbbk - 1} \|_{L^2} \lesssim b_1^{2(1 - C(L)\eta)} \sqrt{\Es_\Bbbk}.
\end{equation}
Let us rewrite from \eqref{def:Lq} the definition of $L(q_1)$, 
$$L(q_1) = \Phi q_1 \quad \text{with} \quad \Phi = \frac{(d-1)}{y^2}\left[\cos(2Q) - \cos(2Q + 2(\mathbf{\Theta}_b)_1)\right],$$
where 
$$(\mathbf{\Theta}_b)_1 = \sum_{i = 1, \textup{even}}^Lb_i T_i \chi_{B_1} + \sum_{i = 2, \textup{even}}^{L+2}S_i(b,y)\chi_{B_1}.$$
From the asymptotic behavior of $Q$ given in \eqref{eq:asymQ}, the admissibility of $\vec T_i$ and the homogeneity of $\vec S_i$, we deduce that $\Phi$ is a regular function both at the origin and at infinity. We then apply the Leibniz rule \eqref{eq:LeibnizLk} with  $\phi = \Phi$ to write 
$$(\Phi q_1)_{\Bbbk - 1} = \sum_{m = 0}^{\Bbbk - 1}\big(q_1 \big)_m\Phi_{\Bbbk - 1, \Bbbk - 1 - m},$$
where $\Phi_{\Bbbk - 1, \Bbbk - 1 -  m}$ with $0 \leq m \leq \Bbbk - 1$ are defined by the recurrence relation given in Lemma \ref{lemm:LeibnizLk}. In particular, we have the following estimate
$$|\Phi_{\Bbbk - 1, \Bbbk - 1 - m}| \lesssim \frac{b_1^{2(1 - C(L)\eta)}}{1 + y^{2\gamma + \Bbbk - 1 - m }} \quad \text{for}\quad 0 \leq m \leq \Bbbk - 1.$$
Hence, from  the coercivity bound \eqref{eq:qmbyE2k} and $2\gamma - 1 \geq 1$, we estimate
\begin{align*}
\int |L(q_1)_{\Bbbk - 1}|^2 &\lesssim b_1^{2(1 - \eta)}\sum_{m = 0}^{\Bbbk - 1}\int \frac{|(q_1)_m|^2}{1 + y^{4\gamma - 2 + 2(\Bbbk - m)}} \lesssim b_1^{4(1 - C(L)\eta)}\Es_\Bbbk,
\end{align*}
which concludes the proof of \eqref{eq:controlLq1}. Hence, we have
\begin{align}
&\left|\int \frac{1}{\lambda^2}\big(L(q_1)\big)_\lambda\Ls_\lambda^{\Bbbk - 1}v_2\right| \lesssim \frac{b_1^{2(1 - C(L)\eta)}}{\lambda^{2\Bbbk - d + 1}}\Es_\Bbbk. \label{est:Lq1}
\end{align}

\noindent $\bullet$ \textbf{The commutator term:} By \eqref{def:commu}, we write 
\begin{align*}
&\int v_1 [\pt, \Ls_\lambda^\Bbbk]v_1 + \int v_2 [\pt, \Ls_\lambda^{\Bbbk-1}]v_2 \\
&\quad = -\frac{\lambda_s}{\lambda^{2\Bbbk - d + 2}}\left[ \sum_{m = 0}^{\Bbbk - 1}\int q_1 \Ls^m \left(\frac{\Lambda Z}{y^2}\Ls^{\Bbbk - 1 - m}q_1 \right) +\sum_{m = 0}^{\Bbbk - 2}\int q_2 \Ls^m \left(\frac{\Lambda Z}{y^2}\Ls^{\Bbbk - 2 - m}q_2 \right)\right]\\
& \quad = -\frac{\lambda_s}{\lambda^{2\Bbbk - d + 2}} \left[ 2\sum_{m = 0}^{\left[\frac{\Bbbk + 1}{2}\right] - 1}\int (q_1)_\Bbbk \left(\frac{\Lambda Z}{y^2} (q_1)_{2m} \right)_{\Bbbk - 2 - 2m} + 2\sum_{m = 0}^{\left[\frac{\Bbbk - 1}{2}\right] - 1}\int (q_2)_{\Bbbk-1} \left(\frac{\Lambda Z}{y^2} (q_2)_{2m} \right)_{\Bbbk - 3 - 2m} \right]\\
& \quad \lesssim \frac{b_1}{\lambda^{2\Bbbk - d + 1}}\sqrt{\Es_\Bbbk}\left[\sum_{m = 0}^{\left[\frac{\Bbbk + 1}{2}\right] - 1}\left\|\left(\frac{\Lambda Z}{y^2} (q_1)_{2m} \right)_{\Bbbk - 2 - 2m} \right\|_{L^2} +  \sum_{m = 0}^{\left[\frac{\Bbbk - 1}{2}\right] - 1}\left\|\left(\frac{\Lambda Z}{y^2} (q_2)_{2m} \right)_{\Bbbk - 3 - 2m} \right\|_{L^2}\right],
\end{align*}
where we used in the last line the fact that $\left|\frac{\lambda_s}{\lambda}\right| \sim b_1$ from the modulation equation \eqref{eq:ODEbkl} and the Cauchy-Schwartz inequality. We note from \eqref{def:Lc} and \eqref{eq:asympV} that 
$$\frac{\Lambda Z}{y^2} = \sum_{i = 0}^k d_i y^{2i + 1} + \Oc(y^{2k + 3}) \quad \text{for} \quad y \to 0,$$
and 
$$\forall j \in \mathbb{N},\quad \left|\py^j\left(\frac{\Lambda Z}{y^2}\right)\right| \lesssim \frac{1}{ y^{2\gamma + 1 + j}} \quad \text{for} \quad y \to +\infty.$$
Applying Lemma \ref{lemm:LeibnizLk} with $\phi = \frac{\Lambda Z}{y^2}$, we write
\begin{align*}
\sum_{m = 0}^{\left[\frac{\Bbbk + 1}{2}\right] - 1 }\big(\phi (q_1)_2m\big)_{\Bbbk - 2  -2m} &= \sum_{m = 0}^{\left[\frac{\Bbbk + 1}{2}\right] - 1} \sum_{j = 0}^{\Bbbk - 2 - 2m}(q_1)_{2m + j} \phi_{\Bbbk - 2 - 2m, \Bbbk - 2 - 2m - j},\\
\sum_{m = 0}^{\left[\frac{\Bbbk - 1}{2}\right] - 1}\big(\phi (q_2)_2m\big)_{\Bbbk - 3  -2m} &= \sum_{m = 0}^{\left[\frac{\Bbbk - 1}{2}\right] - 1}\sum_{j = 0}^{\Bbbk - 3 - 2m}(q_2)_{2m + j} \phi_{\Bbbk - 3 - 2m, \Bbbk - 3 - 2m - j},
\end{align*}
where we compute from the recurrence formula of Lemma \ref{lemm:LeibnizLk}, 
$$\phi_{\Bbbk - 2 - 2m, \Bbbk - 2 - 2m - j} \lesssim \frac{1}{1 + y^{2\gamma + 1 + \Bbbk - 2 - 2m - j }}, \quad \phi_{\Bbbk - 3 - 2m, \Bbbk - 3 - 2m - j} \lesssim \frac{1}{1 + y^{2\gamma + 1 + \Bbbk - 3 - 2m - j}}.$$
We then use the coercivity bound of $\As$ and $\As^*$ given in Lemmas \ref{lemm:coerA} and \ref{lemm:coerA} to estimate 
\begin{align*}
&\sum_{m = 0}^{\left[\frac{\Bbbk + 1}{2}\right] - 1}\left\|\left(\frac{\Lambda Z}{y^2} (q_1)_{2m} \right)_{\Bbbk - 2 - 2m} \right\|^2_{L^2} +  \sum_{m = 0}^{\left[\frac{\Bbbk - 1}{2}\right] - 1}\left\|\left(\frac{\Lambda Z}{y^2} (q_2)_{2m} \right)_{\Bbbk - 3 - 2m} \right\|^2_{L^2}\\
& \quad \lesssim \sum_{m = 0}^{\left[\frac{\Bbbk + 1}{2}\right] - 1} \sum_{j = 0}^{\Bbbk - 2 - 2m} \int \frac{|(q_1)_{2m + j}|^2}{1 + y^{2 (2\gamma - 1 + \Bbbk - 2m - j)}} + \sum_{m = 0}^{\left[\frac{\Bbbk - 1}{2}\right] - 1}\sum_{j = 0}^{\Bbbk - 3 - 2m} \int \frac{|(q_2)_{2m + j}|^2}{1 + y^{2(2\gamma - 1 + \Bbbk - 2m - j - 1)}}\\
& \qquad \lesssim \sum_{m = 0}^{\left[\frac{\Bbbk + 1}{2}\right] - 1} \sum_{j = 0}^{\Bbbk - 2 - 2m} \int \frac{|(q_1)_{\Bbbk}|^2}{1 + y^{2 (2\gamma - 1)}} + \sum_{m = 0}^{\left[\frac{\Bbbk - 1}{2}\right] - 1}\sum_{j = 0}^{\Bbbk - 3 - 2m} \int \frac{|(q_2)_{\Bbbk - 1}|^2}{1 + y^{2(2\gamma - 1)}}\\
& \quad \qquad \lesssim \frac{\Es_\Bbbk}{N^{2(2\gamma - 1)}} + \left\|\frac{(q_1)_\Bbbk}{1 + y^{2\gamma - 1}} \right\|^2_{L^2(y \leq N)} + \left\|\frac{(q_2)_{\Bbbk - 1}}{1 + y^{2\gamma - 1}} \right\|^2_{L^2(y \leq N)}.
\end{align*}
Thus, we obtain
\begin{align}
&\left|\int v_1 [\pt, \Ls_\lambda^\Bbbk]v_1 + \int v_2 [\pt, \Ls_\lambda^{\Bbbk-1}]v_2 \right| \nonumber \\
& \quad \lesssim \frac{b_1}{\lambda^{2\Bbbk - d + 1}}\sqrt{\Es_\Bbbk}\left[\frac{\sqrt{\Es_\Bbbk}}{N^{2\gamma - 1}} + \left\|\frac{(q_1)_\Bbbk}{1 + y^{2\gamma - 1}} \right\|_{L^2(y \leq N)} + \left\|\frac{(q_2)_{\Bbbk - 1}}{1 + y^{2\gamma - 1}} \right\|_{L^2(y \leq N)}\right],\nonumber\\
& \qquad \lesssim \frac{b_1}{\lambda^{2\Bbbk - d + 1}} \left[\frac{\Es_\Bbbk}{N^{2\gamma - 1}} + C(N) \Es_{\Bbbk, \textup{loc}}\right],\label{est:commuk}
\end{align}
where $\Es_{\Bbbk, \textup{loc}}$ is defined by \eqref{def:Ekloc}.

\noindent $\bullet$ \textbf{The modulation term:} Let us introduce the vector function
\begin{equation}\label{def:chiL}
\vec \Upsilon = C_\Upsilon \left( \vec{T}_L + \frac{\partial \vec{S}_{L + 1}}{\partial b_L} + \frac{\partial \vec{S}_{L + 2}}{\partial b_L}\right)\chi_{B_1},
\end{equation}
where
\begin{equation}\label{def:CUp}
C_\Upsilon = \frac{\left< \Hs^L\vec q, \chi_{B_0}\Lambda \vec Q \right>}{\left<\chi_{B_0}\Lambda Q, \Lambda Q + (-1)^{\frac{L - 1} {2}} \left(\frac{\partial S_{L+2}}{\partial b_L}\right)_{L-1}\right>}, \quad |C_\Upsilon|\lesssim b_1^{\delta - 1}\sqrt{\Es_\Bbbk}.
\end{equation}
The size of the coefficient $C_\Upsilon$ is computed from \eqref{def:Gs} and \eqref{eq:HsLchiB0}. The introduction of $\chi_L$ is to take advantage of the improved modulation equation \eqref{eq:ODEbLimproved}. Let us write 
\begin{align*}
&\int \frac{1}{\lambda}(\mathbf{M}_1)_\lambda \Ls^\Bbbk_\lambda v_1 + \int \frac{1}{\lambda^2}(\mathbf{M}_2)_\lambda \Ls^{\Bbbk-1}_\lambda v_2  = \frac{1}{\lambda^{2\Bbbk - d + 1}} \left[\int \mathbf{M}_1 \Ls^\Bbbk q_1 + \int \mathbf{M}_2 \Ls^{\Bbbk-1} q_2\right]\\
& \quad = \frac{1}{\lambda^{2\Bbbk - d + 1}} \left[\int \ps \Upsilon_1 \Ls^\Bbbk q_1 + \int \ps \Upsilon_2 \Ls^{\Bbbk-1} q_2\right] \\
& \qquad + \frac{1}{\lambda^{2\Bbbk - d + 1}} \left[\int \left(\mathbf{M}_1 - \ps \Upsilon_1\right) \Ls^\Bbbk q_1 + \int \left(\mathbf{M}_2 - \ps \Upsilon_2\right) \Ls^{\Bbbk-1} q_2\right]\\
& \quad = \frac{d}{dt} \left[\frac{1}{\lambda^{2\Bbbk - d}} \Sigma_0 \right] + \frac{1}{\lambda^{2\Bbbk - d + 1}}\left[(2\Bbbk - d)\frac{\lambda_s}{\lambda}\Sigma_0 - \Sigma_1 + \Sigma_2\right]
\end{align*}
where $\mathbf{M}_1$ and $\mathbf{M}_2$ are introduced in \eqref{def:M1} and \eqref{def:M2}, 
\begin{align*}
\Sigma_0 &=  \int \Upsilon_1 \Ls^\Bbbk q_1 + \int \Upsilon_2 \Ls^{\Bbbk-1} q_2  +\frac{1}{2}\int \Upsilon_1 \Ls^\Bbbk \Upsilon_1 + \frac{1}{2}\int \Upsilon_2 \Ls^{\Bbbk -1}\Upsilon_2,\\
\Sigma_1 &= \int \left(\mathbf{M}_1 - \ps \Upsilon_1\right) \Ls^\Bbbk q_1 + \int \left(\mathbf{M}_2 - \ps \Upsilon_2\right) \Ls^{\Bbbk-1} q_2,\\
\Sigma_2 & =\int \Upsilon_1 \Ls^\Bbbk \big(\ps q_1 + \ps \Upsilon_1 \big) + \int \Upsilon_2 \Ls^{\Bbbk-1} \big(\ps q_2 + \ps \Upsilon_2\big).
\end{align*}
We claim that 
\begin{equation}\label{est:Sigma012}
|\Sigma_0| \lesssim \Es_\Bbbk b_1^{\eta(1 - \delta)}, \quad |\Sigma_1| + |\Sigma_2| \lesssim \Es_\Bbbk b_1^{1 + \eta(1 - \delta)} + \sqrt{\Es_\Bbbk}b_1^{L + 1 + (1 - \delta)(1 + \eta)},
\end{equation}
from which and $-\frac{\lambda_s}{\lambda} \sim b_1$ we obtain 
\begin{align}
&\int \frac{1}{\lambda}(\mathbf{M}_1)_\lambda \Ls^\Bbbk_\lambda v_1 + \int \frac{1}{\lambda^2}(\mathbf{M}_2)_\lambda \Ls^{\Bbbk-1}_\lambda v_2 \nonumber \\
&\quad =  \frac{d}{dt} \left[\frac{1}{\lambda^{2\Bbbk - d}}\Oc\left(\Es_\Bbbk b_1^{\eta(1 - \delta)}\right) \right] + \Oc\left(\Es_\Bbbk b_1^{1 + \eta(1 - \delta)} + \sqrt{\Es_\Bbbk}b_1^{L + 1 + (1 - \delta)(1 + \eta)}\right). \label{est:Modk}
\end{align}
Let us start the estimate of $\Sigma_0$. By the Cauchy-Schwartz inequality, we write
\begin{equation*}
|\Sigma_0|\lesssim \sqrt{\Es_\Bbbk} \left(\|(\Upsilon_1)_{\Bbbk}\|_{L^2} + \|(\Upsilon_2)_{\Bbbk - 1}\|_{L^2}\right) + \|(\Upsilon_1)_{\Bbbk}\|^2_{L^2} + \|(\Upsilon_2)_{\Bbbk - 1}\|^2_{L^2}.
\end{equation*}
We use the admissibility of $\vec T_k$, the homogeneity of $\vec{S}_{k}$ together with the fact that $\Ls^\frac{L+1}{2}T_L = 0$ to estimate 
\begin{align}
&\|(\Upsilon_1)_{\Bbbk}\|^2_{L^2} + \|(\Upsilon_2)_{\Bbbk - 1}\|^2_{L^2} \lesssim C_\Upsilon^2\left[\int \left|\left(\chi_{B_1} \frac{\partial S_{L+1}}{\partial b_L}\right)_{\Bbbk}\right|^2 +  \int \left|\left(\chi_{B_1} \left( T_L +  \frac{\partial S_{L+2}}{\partial b_L}\right)\right)_{\Bbbk - 1}\right|^2\right]\nonumber\\
&\quad \lesssim C_\Upsilon^2 \left[\int_{y \leq 2B_1}\frac{b_1^2}{1 + y^{2(- L + \gamma + \Bbbk)}} + \int_{B_1 \leq y \leq 2B_1}\frac{1}{1 + y^{2(- L + \gamma  + \Bbbk)}}  +\int_{y \leq 2B_1}\frac{b_1^4}{1 + y^{2(-L + \gamma + \Bbbk - 1)}} \right]\nonumber\\
& \qquad \lesssim b_1^{2\delta - 2}\Es_\Bbbk \left[b_1^2 b_1^{(2 - 2\delta)(1 + \eta)} + b_1^{(2 - 2\delta)(1 + \eta)} + b_1^4b_1^{-2\delta(1 + \eta)}\right] \quad \lesssim b_1^{2\eta(1 - \delta)}\Es_\Bbbk.\label{est:Up12}
\end{align} 
This concludes the proof of \eqref{est:Sigma012} for $\Sigma_0$.

We now prove the estimate \eqref{est:Sigma012} for $\Sigma_1$. From the Cauchy-Schwartz inequality, we write
$$|\Sigma_1| \lesssim \sqrt{\Es_\Bbbk}\left[ \left\| \big(\mathbf{M}_1 - \ps \Upsilon_1 \big)_\Bbbk\right \|_{L^2} + \left\| \big(\mathbf{M}_2 - \ps \Upsilon_2 \big)_{\Bbbk - 1}\right \|_{L^2} \right].$$
We only deal with the second coordinate because the first one is estimated in the same way. Let us write 
\begin{align*}
\mathbf{M}_2 - \ps \Upsilon_2 & = \sum_{k = 1, \textup{odd}}^{L - 1} \big[(b_k)_s + (k - \gamma)b_1b_k - b_{k - 1}\big]\left(T_k + \sum_{j = k + 1, \textup{odd}}^{L + 2}\frac{\partial S_{j}}{\partial b_k}\right) \chi_{B_1} - \left(\frac{\lambda_s}{\lambda} + b_1\right)D(\mathbf{Q}_b)_2\\
& + \left[(b_L)_s + (L - \gamma)b_1b_L - \frac{d}{ds}C_\Upsilon\right]\left(T_L + \frac{\partial S_{L+2}}{\partial b_L} \right)\chi_{B_1} - C_\Upsilon \frac{d}{ds}\left[\left(T_L + \frac{\partial S_{L+2}}{\partial b_L} \right)\chi_{B_1} \right].
\end{align*}
Since $\Ls^{\frac{L+1}{2}}T_k = 0$ for $1 \leq k \leq L$, we us the admissibility of $\vec T_k$ to estimate 
$$\sum_{k = 1, \textup{odd}}^L \left\|\big(\chi_{B_1}T_k\big)_{\Bbbk - 1}\right\|_{L^2}^2 \lesssim \sum_{k = 1, \textup{odd}}^L \int_{y \sim B_1}y^{2(k - \gamma - \Bbbk)} \lesssim \int_{y \sim B_1}y^{2(-\gamma - \hbar -1)} \lesssim b_1^{2(1 - \delta)(1 + \eta)}.$$
From the homogeneity of $\vec S_k$, we have
$$\sum_{k = 1, \textup{odd}}^{L}\sum_{j = k + 1, \textup{odd}}^{L + 2}\left\| \left(\frac{\partial S_{j}}{\partial b_k} \chi_{B_1}\right)_{\Bbbk - 1}\right\|_{L^2}^2 \lesssim \sum_{k = 1, \textup{odd}}^{L}\sum_{j = k + 1, \textup{odd}}^{L + 2} \int_{y \leq 2B_1}\frac{b_1^4}{1 + y^{2(-j + 1 + \gamma + \Bbbk)}}\lesssim b_1^{2(1 - \delta)(1 + \eta)}.$$
Similarly, since $\left|\frac{d}{ds}\chi_{b_1}\right| \lesssim b_1$ and $\frac{\partial S_{L+2}}{\partial b_L}$ does not depend on $b_L$, we have the estimates 
\begin{align*}
\left\|\big(D(\mathbf{Q}_b)_2 \big)_{\Bbbk - 1} \right\|^2_{L^2} \lesssim b_1^{2(1 - \delta)(1 + \eta)}, \quad \left\|\left(\frac{d}{ds}\left(T_L + \frac{\partial S_{L+2}}{\partial b_L} \right)\chi_{B_1} \right)_{\Bbbk - 1} \right\|_{L^2}^2 \lesssim b_1^2b_1^{2(1 - \delta)(1 + \eta)}.
\end{align*}
Gathering all these above estimates together with the modulation equations \eqref{eq:ODEbkl}, \eqref{eq:ODEbLimproved} and the estimate \eqref{def:CUp} yields
\begin{equation}\label{est:M2Up2}
\left\| \big(\mathbf{M}_2 - \ps \Upsilon_2 \big)_{\Bbbk - 1}\right \|_{L^2} \lesssim b_1^{1 + \eta(1 - \delta)}\sqrt{\Es_\Bbbk} + b_1^{L + 1 + (1 - \delta)(2 + \eta)},
\end{equation}
which follows the estimate \eqref{est:Sigma012} for $\Sigma_1$.

We now turn to the proof of \eqref{est:Sigma012} for $\Sigma_2$. We only deal with the second coordinate because the same proof holds for the first one. Let us write from equation \eqref{eq:qys}, 
\begin{align*}
&\int \Upsilon_2 \Ls^{\Bbbk - 1} \left(\ps q_2 + \ps \Upsilon_2\right)\\
&\quad  = \int \Ls^{\Bbbk - 1}\Upsilon_2 \left( \frac{\lambda_s}{\lambda}Dq_2 - \Ls q_1 - (\mathbf{\Psi}_b)_2  - \mathbf{M}_2 + \ps \Upsilon_2 + L(q_1) - N(q_1)\right).
\end{align*}
Using the admissibility of $\vec T_L$, the homogeneity of $\vec S_{L+2}$ and the fact that $\Ls^{\frac{L + 1}{2}}T_L = 0$, we have 
\begin{align}
\left|\Ls^{\Bbbk - 1}\Upsilon_2\right| &\lesssim |C_\Upsilon|\left( \frac{1}{1 + y^{- L + 1 + \gamma + 2\Bbbk - 2}}\mathbf{1}_{\{B_1 \leq y \leq 2B_1\}} + \frac{b_1^2}{1 + y^{- L + \gamma + 2\Bbbk - 2}}\mathbf{1}_{\{y \leq 2B_1\}} \right)\nonumber\\
&\quad \lesssim b_1^{\delta - 1} \sqrt{\Es_\Bbbk} \frac{b_1^{1 + \eta}}{1 + y^{\Bbbk - 1 + \gamma + \hbar}}\mathbf{1}_{\{y \leq 2B_1\}}. \label{est:Up2Lsk}
\end{align}
From \eqref{est:Up2Lsk}, the coercivity bound \eqref{eq:weipykq} and  $-\frac{\lambda_s}{\lambda} \sim b_1$, we estimate 
\begin{align*}
&\left| \frac{\lambda_s}{\lambda}\int \Ls^{\Bbbk - 1}\Upsilon_2 Dq_2\right|\\
 & \quad \lesssim b_1^{1 + \delta + \eta}\sqrt{\Es_\Bbbk} \left[\left(\int \frac{|q_2|^2}{1 + y^{2(\Bbbk - 1)}} \right)^\frac{1}{2} + \left(\int \frac{|\py q_2|^2}{1 + y^{2(\Bbbk - 2)}} \right)^\frac{1}{2} \right]\left(\int_{y \leq 2B_1} \frac{1}{1 + y^{2(\gamma + \hbar)}}\right)^\frac{1}{2}\\
& \qquad \lesssim b_1^{1 + \delta + \eta}\Es_\Bbbk b_1^{-\delta(1  +\eta)} \quad \lesssim b_1^{1 + \eta
(1 - \delta)}\Es_\Bbbk.
\end{align*}
Similarly, we have 
\begin{align*}
\left|\int \Ls^{\Bbbk - 1}\Upsilon_2 \Ls q_1\right| \lesssim b_1^{\delta - 1}\sqrt{\Es_\Bbbk} \left( \int \frac{|(q_1)_2|^2}{1 + y^{2(\Bbbk - 2)}}\right)^\frac{1}{2}\left(\int_{y \leq 2B_1} \frac{1}{1 + y^{2(\gamma + \hbar + 1)}}\right)^\frac{1}{2} \lesssim b_1^{1 + \eta(1 - \delta)} \Es_\Bbbk.
\end{align*}
Using \eqref{eq:estPsibLarge2} and \eqref{est:Up12} yields 
\begin{align*}
\left|\int \Ls^{\Bbbk - 1}\Upsilon_2 (\mathbf{\Psi}_b)_2 \right| \lesssim \left\| \big((\mathbf{\Psi}_b)_2\big)_{\Bbbk - 1} \right\|_{L^2} \left\|\big(\Upsilon_2\big)_{\Bbbk - 1} \right\|_{L^2} \lesssim b_1^{1 + \eta(1 - \delta)}\sqrt{\Es_\Bbbk}b_1^{L + (1 - \delta)(1 + \eta)}.
\end{align*}
From \eqref{est:M2Up2} and \eqref{est:Up12}, we have 
\begin{align*}
\left|\int \Ls^{\Bbbk - 1}\Upsilon_2 \big(\mathbf{M}_2 - \ps \Upsilon_2\big) \right| &\lesssim \left\| \big(\mathbf{M}_2 - \ps \Upsilon_2\big)_{\Bbbk - 1} \right\|_{L^2} \left\|\big(\Upsilon_2\big)_{\Bbbk - 1} \right\|_{L^2}\\
& \quad \lesssim b_1^{1 + 2\eta(1 - \delta)}\Es_\Bbbk + b_1^{1 + \eta(1 - \delta)}\sqrt{\Es_\Bbbk}b_1^{L + (1 - \delta)(1 + \eta)}.
\end{align*}
Note from the definition \eqref{def:Lq} of $L(q_1)$ that
$$|L(q_1)| \lesssim \frac{b_1^{2(1 - \eta)}|q_1|}{1 + y^{2\gamma}}.$$
Thus, by using \eqref{est:Up2Lsk} and the coercivity bound \eqref{eq:weipykq}, we derive
$$\left|\int \Ls^{\Bbbk - 1}\Upsilon_2 L(q_1)\right| \lesssim b_1^{1 + \eta(1 - \delta)} \Es_\Bbbk.$$
For the nonlinear term $N(q_1)$ defined in \eqref{def:Nq}, we note that $|N(q_1)| \lesssim \frac{q_1^2}{y^2}$, we then use \eqref{est:Up2Lsk}, the coercivity bound \eqref{eq:qmbyE2k} and the bootstrap bound given in Definition \ref{def:Skset} to estimate 
\begin{align*}
\left|\int \Ls^{\Bbbk - 1}\Upsilon_2 N(q_1)\right| &\lesssim b_1^{\delta + \eta} \sqrt{\Es_\Bbbk} \int_{y \leq 2B_1} \frac{q_1^2}{y^2(1 + y^{2\Bbbk - 2})} (1 + y^{\Bbbk - 1 - \gamma - \hbar})\\
& \quad \lesssim b_1^{\delta + \eta} \sqrt{\Es_\Bbbk} \Es_\Bbbk b_1^{(-L +\gamma)(1  +\eta)} \lesssim b_1^{1 + \eta(1 - \delta)}\Es_\Bbbk.
\end{align*}
This finishes the proof of \eqref{est:Sigma012} for $\Sigma_2$. 

\medskip

A collection of the estimates \eqref{est:Psi}, \eqref{est:Nq1}, \eqref{eq:controlNq1}, \eqref{est:Lq1}, \eqref{est:commuk} and \eqref{est:Modk} into the identity \eqref{eq:IDEk} yields the formula \eqref{eq:Es2sLya}. This concludes the proof of Proposition \ref{prop:E2k}.
\end{proof}

\subsection{Local Morawetz control.}
We establish in this subsection the so-called Morawetz type identity in order to control the local term $\Es_{\Bbbk, \text{loc}}$ involved in the formula \eqref{eq:Es2sLya}. In particular, we have the following:
\begin{proposition}[Local Morawetz control] \label{prop:Mcon} Let $0 < \nu \ll 1$ and $A \gg 1$ be small and large enough constants, we define
\begin{equation}\label{def:phiA}
\phi_A(y) = \int_0^y \chi_A(\xi)\xi^{1 - \nu}d\xi,
\end{equation}
and 
\begin{equation}\label{def:Mofun}
\Mc = - \int \big[\nabla \phi_A . \nabla (q_1)_{\Bbbk - 1} + \frac{(1 - \nu)}{2}\Delta \phi_A (q_1)_{\Bbbk- 1} \big](q_2)_{\Bbbk - 1}.
\end{equation}
Then the following bounds hold for all $s \in [s_0, s_1]$ for $s_0$ large enough,
\begin{equation}\label{est:Mbound}
|\Mc| \leq C(A,M)\Es_\Bbbk,
\end{equation}
and
\begin{equation}\label{est:Mcontrol}
\frac{d}{dt}\left[\frac{\Mc}{\lambda^{2\Bbbk - d}}\right] \geq \frac{1}{\lambda^{2\Bbbk - d + 1}}\left( \frac{\nu}{2N^\nu}\Es_{\Bbbk,\textup{loc}} - \frac{C(M)}{A^\nu} \Es_\Bbbk - C(A,M)\sqrt{\Es_\Bbbk}b_1^{L + 1 + (1 - \delta)(1 + \eta)}\right),
\end{equation}
where the large constants $M$, $N$ and $\Es_{\Bbbk,\textup{loc}}$ are introduced in \eqref{def:PhiM} and \eqref{def:Ekloc}.
\end{proposition}

\begin{proof} The estimate \eqref{est:Mbound} simply follows from the coercivity bound \eqref{eq:qmbyE2k}. We aim at proving the bound 
\begin{equation}\label{eq:Mcon}
\frac{d}{ds}\Mc \geq  \frac{\nu}{2N^\nu}\Es_{\Bbbk,\textup{loc}} - \frac{C(M)}{A^\nu} \Es_\Bbbk - C(A,M)\sqrt{\Es_\Bbbk}b_1^{L + 1 + (1 - \delta)(1 + \eta)},
\end{equation}
which immediately implies \eqref{est:Mcontrol}. Indeed, from $-\frac{\lambda_s}{\lambda} \sim b_1$ and the bound \eqref{est:Mbound}, we have 
\begin{align*}
&\frac{d}{dt}\left[\frac{\Mc}{\lambda^{2\Bbbk - d}}\right] = \frac{1}{\lambda^{2\Bbbk - d + 1}}\frac{d}{ds}\Mc - \frac{(2\Bbbk - d)\lambda_s}{\lambda^{2\Bbbk - d + 1}\lambda} \Mc\\
& \quad \geq \frac{1}{\lambda^{2\Bbbk - d + 1}}\left[ \frac{\nu}{2N^\nu}\Es_{\Bbbk,\textup{loc}} - \frac{C(M)}{A^\nu} \Es_\Bbbk - C(A,M)\sqrt{\Es_\Bbbk}b_1^{L + 1 + (1 - \delta)(1 + \eta)} \right] - C(A,M)\frac{b_1}{\lambda^{2\Bbbk - d + 1}}\Es_\Bbbk.
\end{align*}
Since $b_1(s) \leq b_1(s_0)\sim \frac{1}{s_0}$, we can take $s_0 = s_0(A)$ large such that $b_1(s_0) \leq \frac{1}{A}$, then the estimate \eqref{est:Mcontrol} follows. \\

Let us give the proof of \eqref{eq:Mcon}. We first claim the following:
\begin{lemma}[Morawetz type identity at the linear level] Let $A \gg 1$ and $0 < \nu \ll 1$, there holds the following:
\begin{align}
&\int \big[\nabla \phi_A . \nabla (q_1)_{\Bbbk - 1} + \frac{(1 - \nu)}{2}\Delta \phi_A (q_1)_{\Bbbk- 1} \big]\Ls (q_1)_{\Bbbk - 1}\nonumber\\
& \quad - \int \big[\nabla \phi_A . \nabla (q_2)_{\Bbbk - 1} + \frac{(1 - \nu)}{2}\Delta \phi_A (q_2)_{\Bbbk- 1} \big](q_2)_{\Bbbk - 1} \quad \gtrsim  \frac{\nu}{N^\nu}\Es_{\Bbbk, \textup{loc}} - \frac{1}{A^\nu}\Es_\Bbbk.\label{id:Mo}
\end{align}
\end{lemma}
\begin{proof} The proof follows exactly the same lines as Lemma 3.8 in \cite{Car161} because we have the same definition $\phi_A$ and a similar structure of the linear operator $\Ls$. Although the potential $\frac{Z}{y^2}$ is different from the one defined in \cite{Car161}, it still satisfies
$$\frac{1}{2}y\py\left(\frac{Z}{y^2}\right) \geq - \frac{1}{y^2}\left[\left(\frac{d-2}{2} \right)^2  - \kappa(d)\right] \quad \text{for some $\kappa(d) > 0$ and $d \geq 7$},$$
thanks to  the asymptotic behavior \eqref{eq:asymQ} and the fact that $\frac{(d-2)^2}{4} - (d-1) \geq \frac{1}{4}$ for $d \geq 7$. For that reason, we refer the interested reader to \cite{Car161} for details of the proof.
\end{proof}
We now use the identity \eqref{id:Mo} to derive the formula \eqref{eq:Mcon}. We compute from the definition of $\Mc$ and the equation \eqref{eq:qys}, 
\begin{align}
\frac{d}{ds}\Mc &= -\int \nabla \phi_A . \nabla \left( q_2 - \frac{\lambda_s}{\lambda}\Lambda q_1 - (\mathbf{\Psi}_b)_1 - \mathbf{M}_1 \right)_{\Bbbk - 1} (q_2)_{\Bbbk - 1} \nonumber\\
& -\int \frac{1 - \nu}{2}\Delta \phi_A \left( q_2 - \frac{\lambda_s}{\lambda}\Lambda q_1 - (\mathbf{\Psi}_b)_1 - \mathbf{M}_1 \right)_{\Bbbk - 1} (q_2)_{\Bbbk - 1}\nonumber\\
& + \int \nabla \phi_A . \nabla (q_1)_{\Bbbk - 1} \left(\Ls q_1 + \frac{\lambda_s}{\lambda}D q_2 + (\mathbf{\Psi}_b)_2 + \mathbf{M}_2  - L(q_1) + N(q_1)\right)_{\Bbbk - 1}\nonumber\\
& +  \int \frac{1 - \nu}{2}\Delta \phi_A (q_1)_{\Bbbk - 1} \left(\Ls q_1 + \frac{\lambda_s}{\lambda}D q_2 + (\mathbf{\Psi}_b)_2 + \mathbf{M}_2  - L(q_1) + N(q_1)\right)_{\Bbbk - 1}.\label{eq:IdM}
\end{align}
From the definitions of $\As$ and $\As^*$, the coercivity bound \eqref{eq:qmbyE2k}, $-\frac{\lambda_s}{\lambda} \sim b_1$ and the compactness of the support of $\nabla\phi_A$ and $\Delta \phi_A$, we have the estimate
\begin{align*}
&\left|\frac{\lambda_s}{\lambda} \int \big[\nabla \phi_A . \nabla \left(\Lambda q_1\right)_{\Bbbk - 1} + \frac{(1 - \nu)}{2}\Delta \phi_A \left(\Lambda q_1\right)_{\Bbbk- 1} \big](q_2)_{\Bbbk - 1}\right|\\
& +\left|\frac{\lambda_s}{\lambda}\int \big[\nabla \phi_A . \nabla (q_1)_{\Bbbk - 1} + \frac{(1 - \nu)}{2}\Delta \phi_A (q_1)_{\Bbbk- 1} \big]\left(Dq_2 \right)_{\Bbbk - 1}\right|\quad \lesssim C(A)b_1\Es_\Bbbk.
\end{align*}
Again from the compactness of the support of $\nabla\phi_A$ and $\Delta \phi_A$, we use the Cauchy-Schwartz inequality and the local bound \eqref{eq:estPsiblocalTilde} for $\vec{\mathbf{\Psi}}$ to obtain the estimate
\begin{align*}
&\left|\int \big[\nabla \phi_A . \nabla \left((\mathbf{\Psi}_b)_1\right)_{\Bbbk - 1} + \frac{(1 - \nu)}{2}\Delta \phi_A \left((\mathbf{\Psi}_b)_1\right)_{\Bbbk- 1} \big](q_2)_{\Bbbk - 1}\right|\\
& +\left|\int \big[\nabla \phi_A . \nabla (q_1)_{\Bbbk - 1} + \frac{(1 - \nu)}{2}\Delta \phi_A (q_1)_{\Bbbk- 1} \big]\left((\mathbf{\Psi}_b)_2 \right)_{\Bbbk - 1}\right|\quad \lesssim C(A)\sqrt{\Es_\Bbbk}b_1^{L + 3}.
\end{align*}
For the small linear term $L(q_1)$ and the nonlinear term $N(q_1)$, we use the Cauchy-Schwartz inequality, the bounds \eqref{eq:controlLq1} and \eqref{eq:controlNq1} to estimate
\begin{align*}
\left|\int \big[\nabla \phi_A . \nabla (q_1)_{\Bbbk - 1} + \frac{(1 - \nu)}{2}\Delta \phi_A (q_1)_{\Bbbk- 1} \big]\left(L(q_1) - N(q_1) \right)_{\Bbbk - 1}\right|\quad \lesssim C(A)b_1\Es_\Bbbk.
\end{align*}
It remains to control the modulation term.  We use the fact that $\chi_{B_1} T_i = T_i$ for $y \leq 2A \ll B_1$ and $(T_i)_{\Bbbk - 1} = 0$ for $1 \leq i \leq L$ to deduce that 
 \begin{align*}
& \sum_{i = 1, \textup{even}}^L\int \big[\nabla \phi_A . \nabla \left(T_i\right)_{\Bbbk - 1} + \frac{(1 - \nu)}{2}\Delta \phi_A \left(T_i\right)_{\Bbbk- 1} \big](q_2)_{\Bbbk - 1} \\
& + \sum_{i = 1, \textup{odd}}^L\int \big[\nabla \phi_A . \nabla (q_1)_{\Bbbk - 1} + \frac{(1 - \nu)}{2}\Delta \phi_A (q_1)_{\Bbbk- 1} \big]\left(\chi_{B_1}T_i\right)_{\Bbbk - 1}=  0.
\end{align*}
For the term $\frac{\partial S_j}{b_i}$ for $j \geq i + 1$, we recall that $S_j$ is homogeneous of degree $(j, j - 1, j \wedge 2, j)$. Thus, $\left|\frac{\partial S_j}{b_i}\right| \lesssim C(A)b_1$ for $y \leq 2A$. We then use Lemma \ref{lemm:mod1} to derive the bound for $1 \leq i \leq L$,
 \begin{align*}
& \sum_{j = i+1, \textup{even}}^{L+2}\left| \big[(b_i)_s + (i - \gamma)b_1b_i - b_{i + 1}\big]\int \left[\nabla \phi_A . \nabla \left(\chi_{B_1}\frac{\partial S_j}{\partial b_i}\right)_{\Bbbk - 1} + \frac{(1 - \nu)}{2}\Delta \phi_A \left(\chi_{B_1}\frac{\partial S_j}{\partial b_i}\right)_{\Bbbk - 1} \right](q_2)_{\Bbbk - 1} \right| \\
& + \sum_{j = i+1, \textup{odd}}^{L+2}\left| \big[(b_i)_s + (i - \gamma)b_1b_i - b_{i + 1}\big]\int \big[\nabla \phi_A . \nabla (q_1)_{\Bbbk - 1} + \frac{(1 - \nu)}{2}\Delta \phi_A (q_1)_{\Bbbk- 1} \big]\left(\chi_{B_1}\frac{\partial S_j}{\partial b_i}\right)_{\Bbbk - 1} \right|\\
& \quad \lesssim C(A,M)\left(b_1\Es_\Bbbk + \sqrt{\Es_\Bbbk}b_1^{L + 2 + (1 - \delta)(1 + \eta)}\right).
\end{align*}
Injecting all the above estimates and identity \eqref{id:Mo} to \eqref{eq:IdM} yields the formula \eqref{eq:Mcon}. This concludes the proof of Proposition \ref{prop:Mcon}.
\end{proof}

\subsection{Monotonicity for $\Es_\sigma$.}
We now in the position to derive the monotonicity formula for $\Es_\sigma$. We claim the following.
\begin{proposition}[Lyapunov monotonicity for $\Es_\sigma$] \label{prop:Esigma} Given $\hbar$, $\delta$ and $\eta$ as defined in \eqref{def:kdeltaplus} and \eqref{def:B0B1}. For $K \geq 1$, we assume that there is $s_0(K) \gg 1$ such that $(b_1(s), \cdots, b_L(s), \vec q(s)) \in \Sc_K(s)$ for $s \in [s_0, s_1]$ for some $s_1 \geq s_0$. Then, the followings hold for $s \in [s_0, s_1]$:
\begin{align}
&\frac{d}{dt}\left(\frac{\Es_{\sigma}}{\lambda^{2\sigma - d}} \right) \leq \frac{b_1}{\lambda^{2\sigma - d + 1}}\sqrt{\Es_\sigma}b_1^{\frac{\ell}{\ell - \gamma}\left(\sigma - \frac d2\right)}\left[ b_1^{\frac{1}{L}\left(\frac{\gamma}{2} - \frac 14\right) + \Oc\left(\frac{1}{L}\left|\sigma - \frac d2\right|\right)} + b_1^{\frac{\gamma - 1}{2}} \right].\label{eq:Esigma}
\end{align}
\end{proposition}
\begin{proof} We compute from definition of $\Es_\sigma$ and equation \eqref{eq:vrt}, 
\begin{align}
&\frac{d}{dt}\left[\frac{\Es_\sigma}{\lambda^{2\sigma -d }} \right] = \frac{d}{dt}\left[\int |\nabla^\sigma v_1|^2 + \int |\nabla^{\sigma- 1}v_2|^2\right]\nonumber\\
& = 2\int \nabla^\sigma v_1. \nabla^\sigma \left(v_2 + \frac{1}{\lambda}(\Fc_1)_\lambda\right) + 2\int \nabla^{\sigma - 1} v_2. \nabla^{\sigma - 1} \left(- \Ls_\lambda v_1 + \frac{1}{\lambda^2}(\Fc_2)_\lambda\right)\nonumber\\
& \quad = -2\int \nabla^{\sigma - 1}v_2.\nabla^{\sigma - 1}\left(\frac{Z_\lambda}{r^2}v_1\right) +   2\int \nabla^\sigma v_1. \nabla^\sigma \left(\frac{1}{\lambda}(\Fc_1)_\lambda\right) + 2\int \nabla^{\sigma - 1} v_2. \nabla^{\sigma - 1} \left(\frac{1}{\lambda^2}(\Fc_2)_\lambda\right)\nonumber\\
& \qquad \leq C\frac{\sqrt{\Es_\sigma}}{\lambda^{2\sigma -d + 1}}\left\{\left\|\nabla^{\sigma - 1}\left(\frac{Z}{y^2}q_1\right) \right\|_{L^2} + \left\|\nabla^\sigma \Fc_1\right\|_{L^2} + \left\|\nabla^{\sigma - 1} \Fc_2 \right\|_{L^2}    \right\},\label{id:Esigma}
\end{align}
where $Z$, $\Fc_1$ and $\Fc_2$ are defined in \eqref{def:Lc} and \eqref{eq:qys}.\\

\noindent \textit{- Estimate for the potential term:} From the expansion \eqref{eq:expqat0}, we have 
$$\int_{y \leq 1} \left| \nabla^{\sigma - 1}\left(\frac{Z}{y^2}q_1\right) \right|^2 \lesssim \Es_\Bbbk.$$
For $y \geq 1$, we note from the asymptotic behavior \eqref{eq:asymQ} that $\left|\py^j\left(\frac{Z}{y^2}\right)\right| \lesssim \frac{1}{y^{2 + j}}$ for $j \in \mathbb{N}$. We then use the Leibniz rule and interpolation bound \eqref{eq:interBound} with $\lfloor\sigma -1 \rfloor + 2 > \sigma$ to obtain the estimates 
\begin{align*}
\int_{y \geq 1} \left| \nabla^{\lfloor\sigma -1 \rfloor}\left(\frac{Z}{y^2}q_1\right) \right|^2 \lesssim \Es_\sigma^{\frac{\Bbbk - \lfloor\sigma -1 \rfloor - 2}{\Bbbk - \sigma}} \Es_\Bbbk^{\frac{\lfloor\sigma -1 \rfloor + 2 - \sigma}{\Bbbk - \sigma}},  \quad \int_{y \geq 1} \left| \nabla^{\Bbbk - 1}\left(\frac{Z}{y^2}q_1\right) \right|^2 \lesssim \Es_\Bbbk.
\end{align*}
By interpolation and the bootstrap bounds given in Definition \eqref{def:Skset}, we have 
$$\int\left| \nabla^{\sigma - 1}\left(\frac{Z}{y^2}q_1\right) \right|^2 \leq C \Es_\sigma^{\frac{\Bbbk - \lfloor\sigma -1 \rfloor - 2}{\Bbbk - 1 - \lfloor\sigma -1 \rfloor}}\Es_\Bbbk^{\frac{1}{\Bbbk - 1 - \lfloor\sigma -1 \rfloor}} \leq C K^2b_1^{\frac{2\ell}{\ell - \gamma}\left(\sigma - \frac d2\right) + \Oc\left(\frac{1}{L}\left|\sigma - \frac d2 \right|\right)}b_1^{\frac{2L + 2(1 - \delta)(1 + \eta)}{L + \hbar - \lfloor\sigma -1 \rfloor}}.$$
Since $\left|\sigma  - \frac d2\right| \ll 1$, we note that $\lfloor\sigma -1 \rfloor \geq \frac{d-3}{2} = \gamma + \hbar + \delta - \frac{3}{2}$. We then compute the exponent 
\begin{align*}
\frac{2L + 2(1 - \delta)(1 + \eta)}{L + \hbar - \lfloor\sigma -1 \rfloor} &= 2 \left[1 + \frac{(1 - \delta)(1 + \eta) + \lfloor\sigma -1 \rfloor - \hbar}{L} + \Oc\left(\frac{1}{L^2}\right) \right]\\
& \geq 2 \left[1 + \frac{\gamma - \frac 12 + \eta(1 - \delta)}{L} + \Oc\left(\frac{1}{L^2}\right) \right].
\end{align*}
Since $b_1(s) \sim \frac{1}{s} \leq \frac{1}{s_0}$, we can take $s_0 = s_0(K)$ large enough to obtain the bound
\begin{equation}\label{est:poEsigma}
\int\left| \nabla^{\sigma - 1}\left(\frac{Z}{y^2}q_1\right) \right|^2 \leq b_1^{\frac{2\ell}{\ell - \gamma}\left(\sigma - \frac d2\right)} b_1^{2 + \frac{2\gamma - 1}{2L} + \Oc \left(\frac{1}{L}\left| \sigma - \frac{d}{2}\right|\right)}.
\end{equation}

We now turn to the estimate for the last term in \eqref{id:Esigma}. We only deal with the $\Fc_2$ term since the same estimate holds for $\Fc_1$. Let us recall form \eqref{eq:qys}, 
$$\Fc_2 = - (\mathbf{\Psi}_b)_2 - \mathbf{M}_2 + L(q_1) - N(q_1).$$

\noindent \textit{- Estimate for the error term $(\mathbf{\Psi}_b)_2$:} we apply \eqref{eq:estPsibLarge1} with $j = \lfloor\sigma -1 \rfloor- \hbar$ and $j = \lfloor\sigma\rfloor - \hbar$ to find that 
\begin{align*}
 \|\nabla^{\lfloor\sigma -1 \rfloor}(\mathbf{\Psi}_b)_2\|_{L^2}^2 &\lesssim b_1^{2\lfloor\sigma -1 \rfloor - 2\hbar + 2 + 2(1 - \delta) - C_L\eta} \lesssim b_1^{2\gamma + 1 - C_L\eta},\\
\|\nabla^{\lfloor\sigma\rfloor}(\mathbf{\Psi}_b)_2\|_{L^2}^2 &\lesssim b_1^{2\lfloor\sigma\rfloor - 2\hbar + 2 + 2(1 - \delta) - C_L\eta} \lesssim b_1^{2\gamma + 3 - C_L\eta},
\end{align*}
where we used the fact that $2\lfloor\sigma -1 \rfloor\geq d- 3$ and $2\lfloor\sigma\rfloor \geq d - 1$ for $\left|\sigma - \frac{d}{2}\right| \ll 1$ and $d = 2\gamma +2\hbar + 2\delta$ from \eqref{def:kdeltaplus}.
Note that $\lfloor\sigma\rfloor - \lfloor\sigma -1 \rfloor = 1$ and $\sigma - 1 - \lfloor\sigma -1 \rfloor \geq \sigma - \frac{d}{2}$ for $\left|\sigma - \frac{d}{2}\right| \ll 1$, we have by interpolation
\begin{align}
\|\nabla^{\sigma - 1}(\mathbf{\Psi}_b)_2\|_{L^2}^2 &\lesssim b_1^{(2\gamma + 1 - C_L \eta) (\lfloor\sigma \rfloor - \sigma +1)}b_1^{(2\gamma + 3 - C_L\eta)(\sigma - 1 - \lfloor\sigma -1 \rfloor)}\nonumber\\
& = b_1^{2\gamma + 1 - C_L\eta}b_1^{2(\sigma - 1 - \lfloor\sigma -1 \rfloor)} \lesssim b_1^2 b_1^{\frac{2\ell}{\ell - \gamma}\left(\sigma - \frac d2\right)} b_1^{\frac{2\gamma - 1}{2}},\label{est:errEs}
\end{align}
for $\eta$ and $\left|\sigma - \frac{d}{2}\right|$ small enough.\\

\noindent \textit{- Estimate for the modulation term $\mathbf{M}_2$:} From Lemma \ref{lemm:mod1}, the admissibility of $\vec T_k$ and homogeneity of $\vec S_k$, we estimate 
\begin{align}
\left\|\nabla^{\sigma - 1}\mathbf{M}_2 \right\|_{L^2} &\lesssim \sqrt{\Es_\Bbbk}\left[\sum_{k = 1, \textup{odd}}^{L}\left(\left\|\nabla^{\sigma  -1} \left(\chi_{B_1}T_k \right) \right\|_{L^2} + \sum_{j = k + 1, \textup{odd}}^{L+2} \left\|\nabla^{\sigma - 1}\left(\chi_{B_1} \frac{\partial S_j}{\partial b_k}\right) \right\|_{L^2}\right)\right]\nonumber\\
& \quad  + b_1^{L + 1 + (1 - \delta)(1 + \eta)}\left\|\nabla^{\sigma - 1}\left(D (\mathbf{Q}_b)_2\right)\right\|_{L^2}\nonumber\\
& \qquad \lesssim \sqrt{\Es_\Bbbk}B_1^{L - \sigma + \hbar + \delta} + b_1^{L + 1 + (1 - \delta)(1 + \eta)}b_1^{2(1 - \eta)}B_1^{\hbar + \delta} \lesssim b_1 b_1^{\frac{\ell}{\ell - \gamma}\left(\sigma - \frac{d}{2}\right)} b_1^{\frac{\gamma - 1}{2}}, \label{est:modEs}
\end{align}
for $\eta$ and $\left|\sigma - \frac{d}{2}\right|$ small enough.\\

\noindent \textit{- Estimate for the small linear term $L(q_1)$:} From the asymptotic behavior \eqref{eq:asymQ} of $Q$ and the definition \eqref{def:Lq} of $L(q_1)$, we have by Leibniz rule
$$\forall j \in \mathbb{N}, \quad \left|\py^j(L(q_1))\right| \lesssim b_1^{2(1 - C(L)\eta)}\sum_{i = 0}^j \frac{|\py^i q_1|}{1 + y^{2\gamma + j - i}}.$$
From the bounds \eqref{eq:weipykq} and \eqref{eq:interBound}, we have the estimate 
$$\int |\nabla^{\lfloor\sigma -1 \rfloor} L(q_1)|^2 \lesssim \Es_\sigma^{\frac{\Bbbk - (2\gamma + \lfloor\sigma -1 \rfloor)}{\Bbbk - \sigma}}\Es_\Bbbk^{\frac{2\gamma + \lfloor\sigma -1 \rfloor - \sigma}{\Bbbk - \sigma}}, \quad \int |\nabla^{\Bbbk - 2}L(q_1)|^2 \lesssim \Es_\Bbbk.$$
By interpolation and the same computation of the exponent as for the potential term, we obtain the estimate
\begin{align}
\left\|\nabla^{\sigma - 1}L(q_1)\right\|^2_{L^2} \lesssim  b_1^{2 + \frac{2\ell}{\ell - \gamma}\left(\sigma - \frac{d}{2}\right)} b_1^{\frac{\gamma - 1}{2}}. \label{est:LEsigma}
\end{align}
\noindent \textit{- Estimate for the nonlinear term $N(q_1)$:} We claim that 
\begin{align}
\left\|\nabla^{\sigma - 1}N(q_1)\right\|^2_{L^2} \lesssim  b_1^{2 + \frac{2\ell}{\ell - \gamma}\left(\sigma - \frac{d}{2}\right)} b_1^{\frac{\gamma}{2L}}. \label{est:NEsigma}
\end{align}
For $y < 1$, we use the expansion \eqref{eq:expanNq1at0} to deduce that 
$$\left\|\nabla^{\sigma - 1}N(q_1)\right\|^2_{L^2(y < 1)} \lesssim \Es_\Bbbk. $$
For $y \geq 1$, we shall control $\left\|\nabla^{\lfloor\sigma -1 \rfloor}N(q_1)\right\|^2_{L^2(y \geq 1)}$ and $\left\|\nabla^{\lfloor\sigma\rfloor}N(q_1)\right\|^2_{L^2(y \geq 1)}$, then obtain the result by interpolation. 
From \eqref{def:Zpsi}, the Leibniz rule and estimate \eqref{est:psiq1}, we write   
\begin{align*}
\int_{y \geq 1}|\py^{\lfloor\sigma -1 \rfloor} N(q_1)|^2  &\lesssim \sum_{j = 0}^{\lfloor\sigma -1 \rfloor} \sum_{i = 0}^j \sum_{k = 0}^i \int_{y \geq 1} \frac{|\py^k q_1|^2 |\py^{i - k}q_1|^2 |\py^{\lfloor\sigma -1 \rfloor - j}\psi|^2}{y^{4 + 2(j - i)}}\nonumber\\
& \lesssim \sum_{j = 0}^{\lfloor\sigma -1 \rfloor} \sum_{i = 0}^j \sum_{k = 0}^i \sum_{m \in \mathcal{I}} \int_{y \geq 1} \frac{|\py^k q_1|^2 |\py^{i - k}q_1|^2}{y^{4 + 2(j - i)}}\prod_{l = 1}^{\lfloor\sigma -1 \rfloor - j} \left(\frac{b_1^{-C\eta}}{y^{2\gamma + 2l}} + |\py^l q_1|^2\right)^{m_l},
\end{align*}
where $ \mathcal{I} = \{m \in \mathbb{N}^{\lfloor\sigma -1 \rfloor - j}: \sum_{l = 1}^{\lfloor\sigma -1 \rfloor - j} lm_l = [\sigma - 1] - j\}$. \\
From the Hardy inequality \eqref{eq:Hardy1} and \eqref{eq:interBound}, we have the estimate
\begin{align}
\left\|\frac{\py^k q_1}{y^{-k + \sigma - \frac{d}{2}}} \right\|^2_{L^\infty(y \geq 1)} \lesssim \int \frac{|\py^{k + 1}q_1|^2}{y^{2(-k - 1 + \sigma)}} + |\py^k q_1(1)|^2\lesssim \Es_\sigma \quad \text{for} \quad k = 0, \cdots, [\sigma - 1].  \label{eq:Hardy4sigma}
\end{align} 
from which we derive for every $m \in \mathcal{I}$, 
\begin{align*}
\frac{y^{2(\lfloor\sigma -1 \rfloor - j)}}{y^{\sum_l m_l (2\sigma  - d)}}\prod_{l = 1}^{\lfloor\sigma -1 \rfloor - j} \left(\frac{b_1^{-C\eta}}{y^{2\gamma + 2l}} + |\py^l q_1|^2\right)^{m_l} = \prod_{l = 1}^{\lfloor\sigma -1 \rfloor - j} \left(\frac{b_1^{-C\eta}}{y^{2\gamma + 2\sigma - d}} + \frac{|\py^l q_1|^2}{y^{-2l + 2\sigma - d}} \right)^{m_l} \lesssim b_1^{-C\eta}.
\end{align*}
Hence, we have
\begin{align*}
\int_{y \geq 1}|\py^{\lfloor\sigma -1 \rfloor} N(q_1)|^2  &\lesssim \sum_{j = 0}^{\lfloor\sigma -1 \rfloor} \sum_{i = 0}^j \sum_{k = 0}^i \sum_{m \in \mathcal{I}} b_1^{-C\eta}\int_{y \geq 1} \frac{|\py^k q_1|^2 |\py^{i - k}q_1|^2}{y^{4  - 2i + 2\lfloor\sigma -1 \rfloor - C(2\sigma - d)}}\\
 &\lesssim \sum_{j = 0}^{\lfloor\sigma -1 \rfloor} \sum_{i = 0}^j \sum_{k = 0}^i \sum_{m \in \mathcal{I}} b_1^{-C\eta} \left\|\frac{|\py^k q_1|^2}{y^{-2k + 2\sigma - d }} \right\|_{L^\infty(y \geq 1)} \int_{y \geq 1} \frac{|\py^{i - k}q_1|^2}{y^{4  - 2i + 2k + 2\lfloor\sigma -1 \rfloor - C(2\sigma - d)}}\\
 & \lesssim b_1^{-C\eta}\Es_\sigma \Es_\sigma^{\frac{\Bbbk - (2 + \lfloor\sigma -1 \rfloor - C(2\sigma - d))}{\Bbbk - \sigma}} \Es_\Bbbk^{\frac{2 + \lfloor\sigma -1 \rfloor - \sigma - C(2\sigma - d)}{\Bbbk - \sigma}} \lesssim \Es_\sigma^{2 + \Oc\left(\frac{1}{L}\right)}b_1^{2 \big(2 + \lfloor\sigma -1 \rfloor - \sigma\big)\left(1 + \frac{\gamma}{2L}\right)},
\end{align*}
where in the last estimate, we used the bootstrap bound on $\Es_\Bbbk$ given in Definition \eqref{def:Skset}, the smallness $|2\sigma - d| + \eta = \Oc\left(\frac{1}{L^2}\right)$ from \eqref{def:sigma} and \eqref{def:B0B1} and the following algebra
\begin{align*}
&-C \eta + {\frac{2 + \lfloor\sigma -1 \rfloor - \sigma - C(2\sigma - d)}{\Bbbk - 1}}\big(2L + 2(1 - \delta)(1 + \eta)\big)\\
& \quad = 2\big(2 + \lfloor\sigma -1 \rfloor - \sigma \big)\left(1 +\frac{\gamma}{L} \right) + \Oc\left(\frac{1}{L^2}\right) \geq 2\big(2 + \lfloor\sigma -1 \rfloor - \sigma \big)\left(1 +\frac{\gamma}{2L} \right) \quad \text{for}\quad L \gg 1.
\end{align*}
Similarly, we have the estimate
$$\int_{y \geq 1}|\py^{\lfloor\sigma\rfloor} N(q_1)|^2  \lesssim \Es_\sigma^{2 + \Oc\left(\frac{1}{L}\right)}b_1^{2 \big(2 + \lfloor\sigma\rfloor - \sigma\big)\left(1 + \frac{\gamma}{2L}\right)}.$$
By interpolation and the fact that $\lfloor\sigma\rfloor - \lfloor\sigma -1 \rfloor = 1$ for $|\sigma - d/2| \ll 1$, we have 
\begin{align*}
\int|\py^{\sigma - 1} N(q_1)|^2  &\lesssim \Es_\sigma^{2 + \Oc\left(\frac{1}{L}\right)} b_1^{2\left( 1 + \frac{\gamma}{2L}\right)\Big( \big(2 + \lfloor\sigma -1 \rfloor - \sigma\big)\big(\lfloor\sigma \rfloor - \sigma + 1\big) + \big(2 + \lfloor\sigma\rfloor - \sigma\big)\big(\sigma - 1 - \lfloor\sigma -1 \rfloor\big) \Big)}\\
& \quad = \Es_\sigma^{2 + \Oc\left(\frac{1}{L}\right)} b_1^{2\left( 1 + \frac{\gamma}{2L}\right)} \lesssim b_1^{\frac{\ell}{\ell - \gamma}(2\sigma - d)}b_1^{2\left( 1 + \frac{\gamma}{2L}\right)},
\end{align*}
which concludes the proof of \eqref{est:NEsigma}. \\

We note that the bounds \eqref{est:errEs} and \eqref{est:modEs} also hold for $\| \nabla^\sigma (\mathbf{\Psi}_b)_1\|^2_{L^2}$ and $\|\nabla^\sigma \mathbf{M}_1\|^2_{L^2}$ by using the same computation. We then inject the estimates \eqref{est:poEsigma}, \eqref{est:errEs}, \eqref{est:modEs}, \eqref{est:LEsigma} and \eqref{est:NEsigma} into the identity \eqref{id:Esigma} to obtain the desired formula \eqref{eq:Esigma}. This concludes the proof of Proposition \ref{prop:Esigma}.
\end{proof}

\subsection{Conclusion of Proposition \ref{prop:redu}.}
We give the proof of Proposition \ref{prop:redu} in this subsection in order to complete the proof of Theorem \ref{Theo:1}. Note that this section corresponds to Section 6.1 of \cite{MRRcjm15}. Here we follow exactly the same lines as in \cite{MRRcjm15} and no new ideas are needed. We divide the proof into 2 parts:

- \textit{Part 1: Reduction to a finite dimensional problem.} Assume that for a given $K > 0$ large and an initial time $s_0 \geq 1$ large, we have $(b_1(s), \cdots, b_L(s), \vec q(s)) \in \Sc_K(s)$ for all $s \in [s_0, s_1]$ for some $s_1 \geq s_0$. By using \eqref{eq:ODEbkl}, \eqref{eq:ODEbLimproved}, \eqref{eq:Es2sLya} and \eqref{eq:Esigma}, we derive new bounds on $\Vc_1(s)$, $b_k(s)$ for $\ell + 1 \leq k \leq L$ and $\Es_\sigma$ and $\Es_\Bbbk$, which are better than those defining $\Sc_K(s)$ (see Definition \ref{def:Skset}). It then remains to control $(\Vc_2(s), \cdots, \Vc_\ell(s))$. This means that the problem is reduced to the control of a finite dimensional function $(\Vc_2(s), \cdots, \Vc_\ell(s))$ and then get the conclusion $(i)$ of Proposition \ref{prop:redu}.

- \textit{Part 2: Transverse crossing.} We aim at proving that if $(\Vc_2(s), \cdots, \Vc_\ell(s))$ touches 
$$\partial \hat \Sc _K(s):= \partial\left(-\frac{K}{s^{\frac{\eta}{2}}(1 - \delta)}, \frac{K}{s^{\frac{\eta}{2}}(1 - \delta)}\right)^{\ell - 1}$$
at $s = s_1$, it actually leaves $\partial \hat \Sc_K(s)$ at $s=s_1$ for $s_1 \geq s_0$, provided that $s_0$ is large enough. We then get the conclusion $(ii)$ of Proposition \ref{prop:redu}.

\paragraph{$\bullet\,$ Reduction to a finite dimensional problem.} We give the proof of item $(i)$ of Proposition \ref{prop:redu} in this part. Given $K > 0$, $s_0 \geq 1$ and the initial data at $s = s_0$ as in Definition \ref{def:1}, we assume for all $s \in [s_0, s_1]$, $(b_1(s), \cdots, b_L(s), \vec q(s)) \in \Sc_K(s)$ for some $s_1 \geq s_0$. We claim that for all $s \in [s_0, s_1]$,
\begin{align}
&|\Vc_1(s)| \leq s^{-\frac \eta 2(1- \delta)},\label{est:impV1}\\
&|b_k(s)| \lesssim s^{-(k + \eta(1 - \delta))} \quad \text{for}\;\; \ell + 1 \leq k \leq L,\label{est:impbk}\\
&\Es_\sigma \leq \frac{K}{2}s^{-\frac{2\ell(\sigma - \frac{d}{2})}{\ell - \gamma}}\label{est:impEsigma}\\
&\Es_{\Bbbk} \leq \frac{K}{2} s^{-(2L + 2(1 - \delta)(1 + \eta))},\label{est:impE2k}
\end{align}
Once these estimates are proved, it immediately follows from Definition \ref{def:Skset} of $\Sc_K$ that if 
$$(b_1(s_1), \cdots, b_L(s), \vec q(s_1)) \in \partial\Sc_K(s_1),$$
then $(\Vc_2, \cdots, \Vc_\ell))(s_1)$ must be in $\partial \hat \Sc_K(s_1)$, which concludes the proof of item $(i)$ of Proposition \ref{prop:redu}.\\

Before going to the proof of \eqref{est:impV1}-\eqref{est:impE2k}, let us compute explicitly the scaling parameter $\lambda$. To do so, let us note from \eqref{eq:Ukbke} and the a priori bound on $\Uc_1$ given in Definition \ref{def:Skset} that 
$b_1(s) = \frac{c_1}{s} + \frac{\Uc_1}{s} = \frac{\ell}{(\ell - \gamma)s} + \Oc\left(\frac{1}{s^{1 + c\eta}}\right).$ 
From the modulation equation \eqref{eq:ODEbkl}, we have
\begin{equation}\label{eq:lam10}
-\frac{\lambda_s}{\lambda} = \frac{\ell}{(\ell - \gamma)s} + \Oc\left(\frac{1}{s^{1 + c\eta}}\right),
\end{equation}
from which we write
\begin{equation*}
\left|\frac{d}{ds}\left\{\log \Big(s^{\frac{\ell}{\ell - \gamma}}\lambda(s)\Big)\right\}\right| \lesssim \frac{1}{s^{1 + c\eta}}.
\end{equation*}
We now integrate by using the initial data value $\lambda(s_0) = 1$ to get 
\begin{equation}\label{eq:Lamdas}
\lambda(s) = \left(\frac{s_0}{s}\right)^\frac{\ell}{\ell - \gamma}\left[1 + \Oc\left(s^{-c\eta}\right)\right] \quad \text{for}\;\; s_0 \gg 1.
\end{equation}
This implies that 
\begin{equation}\label{est:b1so_s}
s_0^{-\frac{\ell}{\ell - \gamma}} \lesssim \frac{s^{-\frac{\ell}{\ell - \gamma}} }{\lambda(s)}\lesssim s_0^{-\frac{\ell}{\ell - \gamma}}.
\end{equation}

\noindent - \textit{Improved control of $\Es_{\Bbbk}$}: We aim at using \eqref{eq:Es2sLya} to derive the improved bound \eqref{est:impE2k}. From the Morawetz formula \eqref{eq:Mcon}, we have 
$$\frac{b_1 \Es_{\Bbbk, \textup{loc}}}{\lambda^{2\Bbbk - d}} \leq C(N)b_1 \frac{d}{ds}\left[\frac{\Mc}{\lambda^{2\Bbbk - d}}\right] + \frac{C(M,N)}{A^\nu}\frac{b_1\Es_\Bbbk}{\lambda^{2\Bbbk - d}} + \frac{C(A, M, N)}{\lambda^{2\Bbbk - d}} \sqrt{\Es_\Bbbk}b_1^{L + 2 + (1 - \delta)(1 + \eta)},$$
from which and the monotonicity formula \eqref{eq:Es2sLya}, we write
\begin{align*}
&\frac{d}{ds} \left\{\frac{\Es_\Bbbk}{\lambda^{2\Bbbk - d}} \left[1 + \Oc(b_1^{\eta(1 - \delta)})\right]\right\} \leq C(N,M)b_1 \frac{d}{ds}\left[\frac{\Mc}{\lambda^{2\Bbbk - d}}\right]\\
&\quad + \frac{b_1}{\lambda^{2\Bbbk - d}}\left\{C(A, M, N)\sqrt{\Es_\Bbbk} b_1^{L + (1 - \delta)(1 + \eta)} + C(M) \left[K b_1^{\frac{\ell}{\ell - \gamma}(2\sigma - d) + \eta(1 - \delta) + \Oc\left(\frac{|2\sigma  - d|}{L}\right)} + \frac{1}{N^{2\gamma - 1}} + \frac{C(N)}{A^\nu} \right]\Es_\Bbbk \right\} \\
& \quad \leq \frac{C(N,M)}{s} \frac{d}{ds}\left[\frac{\Mc}{\lambda^{2\Bbbk - d}}\right] + \frac{1}{\lambda^{2\Bbbk - d}} \left[ \left(C(M)(\sqrt{K} + 1)s^{-(2L  +1 + 2(1 - \delta)(1 + \eta))} \right) \right],
\end{align*}
where we used the bootstrap bounds given in Definition \ref{def:Skset}, $b_1 \sim \frac{1}{s} \leq \frac{1}{s_0}$ and the constants $s_0, A, N, M, L$ is fixed large enough. Integrating in time by using $\lambda(s_0) = 1$ and $|\Mc| \leq C(A,M)\Es_\Bbbk$ yields for all $s \in [s_0, s_1)$,
\begin{align*}
\Es_{\Bbbk}(s) &\leq C(M) \lambda(s)^{2\Bbbk - d}\left[\Es_{\Bbbk}(s_0) + \left(\sqrt K +  1\right)\int_{s_0}^s\frac{\tau^{-(2L + 1 + 2(1 - \delta)(1 + \eta))}}{\lambda(\tau)^{2\Bbbk - d}}d\tau\right]\\
& \quad + C(N,M)\left[\frac{\Mc(s)}{s} - \frac{\Mc(s_0)}{s_0} \lambda(s)^{2\Bbbk - d} \right] + C(M,N)\lambda(s)^{2\Bbbk - d}\int_{s_0}^s \frac{|\Mc(\tau)|}{\tau^2 \lambda(\tau)^{2\Bbbk - d}}\\
& \quad \leq  C(M) \lambda(s)^{2\Bbbk - d}\left[\Es_{\Bbbk}(s_0) + \left(\sqrt K +  1\right)\int_{s_0}^s\frac{\tau^{-(2L + 1 + 2(1 - \delta)(1 + \eta))}}{\lambda(\tau)^{2\Bbbk - d}}d\tau\right],
\end{align*}
for $s_0$ large enough. Using \eqref{est:b1so_s}, we estimate
\begin{align*}
&\lambda(s)^{2\Bbbk - d}\int_{s_0}^s\frac{\tau^{-(2L + 1 + 2(1 - \delta)(1 + \eta))}}{\lambda(\tau)^{2\Bbbk - d}}d\tau\\
& \quad \lesssim s^{- \frac{\ell(2\Bbbk - d)}{\ell - \gamma}}\int_{s_0}^s \tau^{\frac{\ell(2\Bbbk - d)}{\ell - \gamma} - (2L + 1 + 2(1 - \delta)(1 + \eta))}d\tau \lesssim s^{-(2L + 2(1 - \delta)(1 + \eta))}.
\end{align*}
Here we used the fact that the integral is divergent because
$$\frac{\ell(2\Bbbk - d)}{\ell - \gamma} - [2L + 1 + 2(1 - \delta)(1 + \eta)] = \frac{2\gamma L}{\ell - \gamma} + \Oc_{L \to +\infty}(1) \gg -1.$$
Using again \eqref{est:b1so_s} and the initial bound \eqref{eq:intialbounE2m}, we estimate
$$\lambda(s)^{2\Bbbk - d}\Es_{\Bbbk}(s_0) \leq \left(\frac{s_0}{s}\right)^\frac{\ell(2\Bbbk - d)}{\ell - \gamma} s_0^{-\frac{10L\ell}{\ell - \gamma}} \lesssim s^{-(2L + 2(1 - \delta)(1 + \eta))},$$
for $L$ large enough. Therefore, we obtain 
$$\Es_{\Bbbk}(s) \leq C\left(\sqrt K + 1\right)s^{-(2L + 2(1 - \delta)(1 + \eta))} \leq \frac{K}{2}s^{-(2L + 2(1 - \delta)(1 + \eta))}, $$
for $K = K(M)$ large enough. This concludes the proof of \eqref{est:impE2k}.\\

\noindent - \textit{Improved control of $\Es_{\sigma}$.} We can improve the control of $\Es_{\sigma}$ by using the monotonicity formula \eqref{eq:Esigma}. Indeed, we see from \eqref{eq:Esigma} that there exists a small constant $0 < \epsilon \ll 1$ such that 
$$\frac{d}{ds} \left[\frac{\Es_\sigma}{\lambda^{2\sigma - d}}\right] \leq \frac{b_1\sqrt{\Es_\sigma}}{\lambda^{2\sigma - d }}b_1^{\frac{\ell}{\ell - \gamma}\left(\sigma - \frac d2\right) + \epsilon},$$
from which  we obtain
$$\Es_\sigma(s) \leq \lambda^{2\sigma - d}\left[\Es_\sigma(s_0) + \sqrt{K}\int_{s_0}^s \frac{b_1}{\lambda^{2\sigma -d}}b_1^{\frac{2\ell}{\ell - \gamma}\left(\sigma - \frac d2\right) + \epsilon}\right].$$
We estimate from the initial bound \eqref{eq:interBound} on $\Es_\sigma(0)$ and \eqref{eq:lam10},
$$\Es_\sigma(s_0)\lambda^{2\sigma - d}(s) \leq C s_0^{-\frac{10L \ell}{\ell - \gamma}} \left(\frac{s_0}{s}\right)^{\frac{\ell}{\ell - \gamma} \left(2\sigma - d\right)} \leq s^{-\frac{2\ell}{\ell - \gamma}\left(\sigma - \frac d2\right)}.$$
Note that $\frac{b_1}{\lambda^{2\sigma -d}}b_1^{\frac{2\ell}{\ell - \gamma}\left(\sigma - \frac d2\right) + \epsilon} \leq \frac{1}{s^{1 + \epsilon}}$, the integral is convergent. Thus, we have the estimate
$$\sqrt{K}\int_{s_0}^s \frac{b_1}{\lambda^{2\sigma -d}}b_1^{\frac{2\ell}{\ell - \gamma}\left(\sigma - \frac d2\right) + \epsilon} \leq C\sqrt{K} s^{-\frac{2\ell}{\ell - \gamma}\left(\sigma - \frac d2\right)}.$$
Hence, by choosing $K$ large enough such that $1  +C\sqrt{K} \leq \frac{K}{2}$, we deduce the improve bound \eqref{est:impEsigma}.\\
  
\noindent - \textit{Control of the stable modes $b_k$'s.} We now close the control of the stable modes $(b_{\ell + 1}, \cdots, b_L)$, in particular, we prove \eqref{est:impbk}. We first treat the case when $k = L$. Let 
$$\tilde{b}_L = b_L + \frac{\left<\Hs^L \vec q,\chi_{B_0}\Lambda \vec Q \right>}{\left<\chi_{B_0}\Lambda Q, \Lambda Q  + (-1)^{\frac{L - 1}{2}}\Ls^{\frac{L - 1}{2}} \left(\frac{\partial S_{L+2}}{\partial b_L}\right)\right>},$$
then from \eqref{def:Gs}, \eqref{eq:HsLchiB0} and \eqref{est:impE2k}, 
$$|\tilde b_L - b_L| \lesssim b_1^{-(1 - \delta)}\sqrt{\Es_{\Bbbk}} \lesssim b_1^{L + \eta(1 - \delta)}.$$
Hence, we have from the improved modulation equation \eqref{eq:ODEbLimproved}, 
\begin{align*}
|(\tilde{b}_L)_s + (L - \gamma)b_1\tilde{b}_L|&\lesssim b_1|\tilde b_L - b_L| +b_1^\delta\left[C(M)\sqrt{\Es_{\Bbbk}} + b_1^{L +(1 - \delta)}\right] \lesssim b_1^{L + 1 + \eta(1 -\delta)}.
\end{align*}
This follows
$$\left|\frac{d}{ds} \left\{\frac{\tilde b_L}{\lambda^{L - \gamma}} \right\}\right| \lesssim \frac{b_1^{L + 1 + \eta(1 - \delta)}}{\lambda^{L - \gamma}}.$$
Integrating this identity in time from $s_0$ and recalling that $\lambda(s_0) = 1$ yields
\begin{align*}
\tilde{b}_L(s) \lesssim C\lambda(s)^{L - \gamma}\left(\tilde{b}_L(s_0) + \int_{s_0}^s\frac{b_1(\tau)^{L + 1 + \eta(1 - \delta)}}{\lambda(\tau)^{L - \gamma}} d\tau\right).
\end{align*}
Using the fact that $b_1(s) \sim \frac{1}{s}$, the initial bounds \eqref{eq:initbk} and \eqref{eq:intialbounE2m} together with \eqref{est:b1so_s}, we estimate
$$\lambda(s)^{L - \gamma}\tilde{b}_L(s_0) \lesssim \left(\frac{s_0}{s}\right)^{\frac{\ell(L - \gamma)}{\ell - \gamma}}\left(s_0^{-\frac{5\ell(L - \gamma)}{\ell - \gamma}} + s_0^{\eta(1 - \delta)}s_0^{-\frac{5L\ell}{\ell - \gamma}} \right) \lesssim s^{- L - \eta(1 - \delta)},$$
and 
\begin{align*}
\lambda(s)^{L - \gamma}\int_{s_0}^s\frac{b_1(\tau)^{L + 1 + \eta(1 - \delta)}}{\lambda(\tau)^{L - \gamma}} d\tau &\lesssim s^{-\frac{\ell(L - \gamma)}{\ell - \gamma}}\int_{s_0}^s \tau ^{ \frac{\ell(L - \gamma)}{\ell - \gamma} - L - 1 - \eta(1 - \delta)}d\tau\\
&\quad \lesssim s^{-L - \eta(1 - \delta)}.
\end{align*}
Therefore, 
$$b_L(s) \lesssim |\tilde{b}_L(s)| + |\tilde{b}_L(s) - b_L(s)| \lesssim s^{-L - \eta(1 - \delta)},$$
which concludes the proof of \eqref{est:impbk} for $k = L$.
Now we will propagate this improvement that we found for the bound of $b_L$ to all $b_k$ for all $\ell + 1 \leq k \leq L - 1$. To do so we do a descending induction where the initialization is for $k=L$. Let assume  the bound 
$$|b_k|\lesssim  b_1^{k + \eta(1 - \delta)},$$
for $k+1$ and let's prove it for $k$. 
Indeed, from \eqref{eq:ODEbkl} and the induction bound, we have 
$$\left|(b_k)_s - (k - \gamma)\frac{\lambda_s}{\lambda}b_k\right| \lesssim b_1^{L + 1} + |b_{k + 1}| \lesssim b_1^{k + 1 + \eta(1 - \delta)},$$
which follows
$$\left|\frac{d}{ds}\left\{\frac{b_k}{\lambda^{k - \gamma}}\right\} \right| \lesssim \frac{b_1^{k + 1 + \eta(1 - \delta)}}{\lambda^{k - \gamma}}.$$
Integrating this identity in time as for the case $k = L$, we end-up with 
\begin{align*}
b_k(s) &\lesssim C\lambda(s)^{k - \gamma}\left(b_k(s_0) + \int_{s_0}^s\frac{b_1(\tau)^{k + 1 + \eta(1 - \delta)}}{\lambda(\tau)^{k - \gamma}} d\tau\right)\\
&\quad \lesssim s^{-k - \eta(1 - \delta)},
\end{align*}
where we used the initial bound \eqref{eq:initbk}, \eqref{est:b1so_s} and $k \geq \ell + 1$. This concludes the proof of \eqref{est:impbk}.

\noindent - \textit{Control of the stable mode $\Vc_1$.} We recall from \eqref{eq:Ukbke} and \eqref{def:VctoUc} that 
$$b_k = b_k^e + \frac{\Uc_k}{s^k}, \quad 1 \leq k \leq \ell, \quad \Vc = P_\ell \Uc,$$
where $P_\ell$ diagonalize the matrix $A_\ell$ with spectrum \eqref{eq:diagAlPl}. From \eqref{eq:bkk1}, and \eqref{eq:ODEbkl}, we estimate for $1 \leq k \leq \ell - 1$, 
$$|s(\Uc_k)_s - (A_\ell \Uc)_k| \lesssim s^{k + 1}|(b_k)_s + (k - \gamma)b_1 b_k - b_{k + 1}| + |\Uc|^2 \lesssim s^{-L + k} + |\Uc|^2,$$
From \eqref{eq:bell}, \eqref{eq:ODEbkl} and the improved bound \eqref{est:impbk}, we have 
$$|s(\Uc_\ell)_s - (A_\ell \Uc)_\ell| \lesssim s^{\ell + 1}\left(|(b_k)_s + (k - \gamma)b_1 b_\ell - b_{\ell + 1}| + |b_{\ell + 1}|\right) + |\Uc|^2 \lesssim s^{-\eta(1 - \delta)} + |\Uc|^2.$$
Using the diagonalization \eqref{eq:diagAlPl}, we obtain 
\begin{equation}\label{eq:sVs}
s\Vc_s = D_\ell \Vc + \Oc(s^{-\eta(1 - \delta)}).
\end{equation}
Using \eqref{eq:diagAlPl} again yields the control of the stable mode $\Vc_1$:
$$|(s\Vc_1)_s| \lesssim s^{-\eta(1 - \delta)}.$$
Thus from the initial bound \eqref{eq:initbk}, 
$$|s^{\eta(1 - \delta)}\Vc_1(s)| \leq \left(\frac{s_0}{s}\right)^{1 - \eta(1 - \delta)}s_0^{\eta(1 - \eta)}\Vc_1(s_0) + 1\lesssim s_0^{\eta(1 - \delta)},$$
which yields \eqref{est:impV1} for $s_0 \geq s_0(\eta)$ large enough.\\

\paragraph{$\bullet \,$ Transverse crossing.} We give the proof of item $(ii)$ of Proposition \ref{prop:redu} in this part. We compute from \eqref{eq:sVs} and \eqref{eq:diagAlPl} at the exit time $s = s_1$:
\begin{align*}
&\frac{1}{2}\frac{d}{ds}\left(\sum_{k = 2}^\ell|s^{\frac{\eta}{2}(1 - \delta)}\Vc_k(s)|^2 \right)_{\big|_{s= s_1}} \\
&\quad = \left(s^{\eta(1 - \delta) - 1} \sum_{k = 2}^\ell \left[\frac{\eta}{2}(1 - \delta)\Vc_k^2(s) + s\Vc_k(\Vc_k)_s \right]\right)_{\big|_{s= s_1}}\\
&\qquad = \left(s^{\eta(1 - \delta) - 1} \Bigg[\sum_{k = 2}^\ell \left[\frac{k\gamma}{k - \gamma} + \frac{\eta}{2}(1 - \delta)\right]\Vc_k^2(s) + \Oc\left(\frac{1}{s^{\frac{3}{2}\eta(1 - \delta)}}\right)\Bigg] \right)_{\big|_{s= s_1}}\\
&\quad \qquad \geq \frac{1}{s_1}\left[c(d,\ell) \sum_{k = 2}^\ell |s_1^{\frac{\eta}{2}(1 - \delta)}\Vc_k(s_1)|^2 +  \Oc\left(\frac{1}{s_1^{\frac{\eta}{2}(1 - \delta)}}\right)\right]\\
& \qquad \qquad \geq \frac{1}{s_1}\left[c(d,\ell) +  \Oc\left(\frac{1}{s_1^{\frac{\eta}{2}(1 - \delta)}}\right)\right] > 0,
\end{align*}
where we used item $(i)$ of Proposition \ref{prop:redu} in the last step. This completes the proof of Proposition \ref{prop:redu}.

\appendix

\section{Coercivity of the adapted norms.}
In this section we aim at proving the coercivity of the adapted norms $\Es_\Bbbk$ defined by (see \eqref{def:normk})
$$\Es_{\Bbbk}:= \|\vec q\|^2_{\Bbbk} = \int_{\Rb^d} q_1\Ls^{\Bbbk}q_1 + \int_{\Rb^d} q_2\Ls^{\Bbbk - 1}q_2 = \int_{\Rb^d}|(q_1)_{\Bbbk}|^2 + \int_{\Rb^d}|(q_2)_{\Bbbk - 1}|^2,$$
where we exploit the notation 
$$f_0 = f, \quad f_{2k + 1} = \As f_{2k} = \As \Ls^k f, \quad f_{2k+2} = \As^*f_{2k + 1} = \Ls^{k + 1}f \quad \text{for} \quad k \in \mathbb{N}.$$
To do so, we first recall some results in \cite{GINapde18}
concerning  the coercivity estimates for the operator $\As$, $\As^*$ under some suitable orthogonality condition. As a consequence, we then obtain the coercivity of $\Es_\Bbbk$.

We recall from Lemma \ref{lemm:estPhiM} that the direction $\vec \Phi_M$ defined in \eqref{def:PhiM} is of the form 
\begin{equation}\label{def:fromPhiM}
\vec \Phi_M = \binom{\Phi_M}{0} \quad \text{with} \quad \Phi_M = \sum_{k = 0}^{\frac{L-1}2}(-1)^kc_{2k,M}\Ls^k(\chi_M \Lambda Q),
\end{equation}
where $L \gg 1$ is an odd integer. We denote by $\Dc_{rad}$ as the set of all radially symmetric functions.
For simplicity, we write
$$\int f := \int_0^{+\infty} f(y)y^{d-1}dy.$$

We have the following:
\begin{lemma}[Hardy inequalities]\label{lemm:Hardy} Let $d \geq 7$ and $f \in \Dc_{rad}$, then\\
$(i)$ (Control near the origin)
\begin{equation*}
\int_{y \leq 1} \frac{|\py f|^2}{y^{2i}} \geq \frac{(d - 2 - 2i)^2}{4}\int_{y\leq 1} \frac{f^2}{y^{2+2i}} - C(d)f^2(1), \quad i = 0, 1, 2.
\end{equation*}
$(ii)$ (Control away from the origin for the non-critical exponent) Let $\alpha > 0, \alpha \ne \frac{d-2}{2}$, then
\begin{align}
&\int_{y \geq 1} \frac{|\py f|^2}{y^{2\alpha}} \geq \left(\frac{d - (2\alpha + 2)}{2}\right)^2\int_{y \geq 1} \frac{f^2}{y^{2 + 2\alpha}} - C(\alpha,d)f^2(1),\label{eq:Hardy0}\\
&\int_{y \geq 1} \frac{|\py f|^2}{y^{2\alpha}} \geq \left(\frac{d - (2\alpha + 2)}{2}\right)^2\left\| \frac{f}{y^{ \alpha + 1 - \frac d2}}\right\|^2_{L^\infty(y \geq 1)} - C(\alpha,d)f^2(1),\label{eq:Hardy1}
\end{align}
$(iii)$ (Control away from the origin for the critical exponent) Let $\alpha = \frac{d-2}{2}$, then
$$
\int_{y \geq 1} \frac{|\py f|^2}{y^{2\alpha}} \geq \frac 14 \int_{y \geq 1} \frac{f^2}{y^{2 + 2\alpha} (1 + \log y)^2} - C(d)f^2(1).$$
$(iv)$ (Weighted Hardy inequality) For any $\mu > 0$, $k \geq 2$ be an integer and $1 \leq j \leq k-1$,
\begin{equation}\label{eq:genHardy}
\int \frac{|\py^jf|^2}{1 + y^{\mu + 2(k - j)}} \lesssim_{j,\mu} \int \frac{|\py^k f|^2}{1 + y^\mu} + \int \frac{f^2}{1 + y^{\mu + 2k}}.
\end{equation}
\end{lemma}
\begin{proof} See Lemma B.1 in \cite{MRRcjm15}.
\end{proof}

We have the following coercivity of $\As^*$ and $\As$: 
\begin{lemma}[Weight coercivity of $\As^*$]\label{lemm:coerAst} Let $\alpha \geq 0$, and $f \in \Dc_{rad}$ satisfying 
$$i = 0, 1,2, \quad \int \frac{|\py f|^2}{y^{2i}(1 + y^{2\alpha})} + \int \frac{f^2}{y^{2i + 2}(1 + y^{2\alpha})} < +\infty,
$$
then 
\begin{equation}\label{eq:coerAst}
i = 0, 1,2, \quad \int \frac{|\As^* f|^2}{y^{2i}(1 + y^{2\alpha})} \geq c_\alpha\left(\int \frac{|\py f|^2}{y^{2i}(1 + y^{2\alpha})} + \int \frac{f^2}{y^{2i + 2}(1 + y^{2\alpha})}\right),
\end{equation}
for some $c_\alpha > 0$.
\end{lemma}
\begin{proof} See Lemma A.2 in \cite{GINapde18}.
\end{proof}
We also have the following coercivity of $\As$. 
\begin{lemma}[Weight coercivity of $\As$] \label{lemm:coerA} Let $p \geq 0$ and $i = 0, 1, 2$ such that $|2p + 2i - (d - 2 - 2\gamma)| \ne 0$, where $\gamma \in (1,2]$ is defined by \eqref{def:gamome}. 
For all $f \in \mathcal{D}_{rad}$ with 
$$\int \frac{|\py f|^2}{y^{2i}(1 + y^{2p})}  + \int \frac{f^2}{y^{2i + 2}(1 + y^{2p})} < +\infty,$$
and
\begin{equation}\label{cond:A3}
\left<f, \Phi_M\right> = 0 \quad \text{if}\quad 2i + 2p > d - 2\gamma - 2,
\end{equation}
where $\Phi_M$ is defined in \eqref{def:PhiM}, we have 
\begin{equation}\label{eq:coerA}
\int \frac{|\As f|^2}{y^{2i}(1 + y^{2p})} \gtrsim \int \frac{|\py f|^2}{y^{2i}(1 + y^{2p})} + \int \frac{f^2}{y^{2i + 2}(1 + y^{2p})}.
\end{equation}
\end{lemma}
\begin{proof} See Lemma A.3 in \cite{GINapde18}.
\end{proof}
From the coercivity estimates of $\As$ and $\As^*$ given in Lemmas \ref{lemm:coerA} and \ref{lemm:coerAst}, we can turn to the core of this part: the coercivity of the adapted norm $\Es_{\Bbbk}$. In particular, we have the following.
\begin{lemma}[Coercivity of $\Es_\Bbbk$]  \label{lemm:coerEk} Let $L \gg 1$ be an odd integer and $\Bbbk$ be defined as in \eqref{def:kbb}, there exists a constant $c = c(L, M) > 0$ such that for all radially symmetric vector function $\vec q$ satisfying
\begin{equation}
\sum_{k = 0}^{\Bbbk - 1}\int \frac{|(q_1)_k|^2}{y^2(1 + y^{2\Bbbk -2 - 2k})} + \sum_{k = 0}^{\Bbbk - 2}\int \frac{|(q_2)_k|^2}{y^2(1 + y^{2\Bbbk - 4 - 2k)}} < +\infty,
\end{equation}
and 
\begin{equation}\label{ap:ortHk}
\big \langle \vec q, \Hs^{*i}\vec \Phi_M \big \rangle = 0 \quad \text{for} \quad 0 \leq i \leq L,
\end{equation}
there holds:
\begin{equation}
\sum_{k = 0}^{\Bbbk - 1}\int \frac{|(q_1)_k|^2}{y^2(1 + y^{2\Bbbk -2 - 2k})} + \sum_{k = 0}^{\Bbbk - 2}\int \frac{|(q_2)_k|^2}{y^2(1 + y^{2\Bbbk - 4 - 2k)}} \leq c\Es_\Bbbk.
\end{equation}
\end{lemma}
\begin{proof} By \eqref{def:H2k1adj}, we see that the condition \eqref{ap:ortHk} is equivalent to 
\begin{equation}\label{eq:orap}
\big \langle \Ls^iq_1,  \Phi_M \big \rangle = 0 \quad \text{and} \quad \big \langle \Ls^{i}q_2,  \Phi_M \big \rangle = 0 \quad \text{for} \;\; 0 \leq i \leq \frac{L-1}{2}.
\end{equation}

Recall that 
\begin{equation}
\Es_\Bbbk = \int |(q_1)_{\Bbbk}|^2 + \int |(q_2)_{\Bbbk - 1}|^2.
\end{equation}
We will write indifferently $q$ to denote $q_1$ and $q_2$, and try to control the term of the form 
$\int |q_{k}|^2$ with $k = \Bbbk$ or $k = \Bbbk - 1$. Let us rewrite
$$q_k = \As q_{k - 1} \; \text{or}\; q_k = \As^*q_{k-1},$$
and apply Lemma \ref{lemm:coerA} or Lemma \ref{lemm:coerAst} with $i = p = 0$ to find that
$$\int |q_k|^2 \gtrsim \int \frac{|q_{k - 1}|^2}{y^2}.$$
If $k - 1 = 0$, we are done, if not, we repeat this step again by writing 
$$q_{k-1} = \As^* q_{k - 2} \; \text{or}\; q_{k-1} = \As q_{k-2},$$
and so forth. Note that $\frac{1}{y^2} \gtrsim \frac{1}{1 + y^2}$ and that the orthogonal condition in Lemma \ref{lemm:coerA} is fulfilled thanks to \eqref{eq:orap}. Then, applying Lemma \ref{lemm:coerA} or Lemma \ref{lemm:coerAst} with the appropriate values of $i$ and $p$ would give
$$\int |q_k|^2 \gtrsim \int \frac{|q_{k - 1}|^2}{y^2} \gtrsim \cdots \gtrsim \int \frac{|q_1|^2}{y^2(1 + y^{2k - 2})} \gtrsim \int \frac{|q|^2}{y^2(1 + y^{2k - 1})}.$$
This concludes the proof of Lemma \ref{lemm:coerEk}.
\end{proof}

\section{Interpolation bounds.}
We derive in this section interpolation bounds on $\vec q$ which are the consequence of the coercivity property given in Lemma \ref{lemm:coerEk}. We have the following:
\begin{lemma}[Interpolation bounds]\label{lemm:interbounds} Suppose that $\Es_\Bbbk$ and $\Es_\sigma$ satisfy the bootstrap bounds in Definition \ref{def:Skset} and that $\vec q$ satisfies the orthogonal condition \eqref{eq:orthqPhiM}, there holds:\\
$(i)$ Weighted bounds for $\vec q$:
\begin{equation}\label{eq:qmbyE2k}
\int |(q_1)_{\Bbbk}|^2 + \int|(q_2)_{\Bbbk - 1}|^2  + \sum_{i = 0}^{\Bbbk - 1}\int \frac{|(q_1)_i|^2}{y^2(1 + y^{2\Bbbk - 2i - 2})} + \sum_{i = 0}^{\Bbbk - 1}\int \frac{|(q_2)_i|^2}{y^2(1 + y^{2\Bbbk - 2i - 4})} \leq c(M) \Es_{\Bbbk}.
\end{equation}
$(ii)$ Expansion near the origin: for $y < 1$,
\begin{equation}\label{eq:expqat0}
\vec q = \sum_{i = 1}^{\Bbbk}c_i\vec T_{\Bbbk - i} + \vec r, \quad |c_i| \lesssim \sqrt{\Es_\Bbbk},
\end{equation}
where $\vec T_{k}$ is defined as in \eqref{def:Tk} for all $k \in \mathbb{N}$, and $\vec r$ satisfies the bounds
$$\sum_{i = 0}^{\Bbbk - 1}|y^i\py^i r_1|  \lesssim y^{\Bbbk - \frac{d}{2}}\sqrt{\Es_\Bbbk}, \quad  \sum_{i = 0}^{\Bbbk-2}|y^i\py^i r_2|  \lesssim y^{\Bbbk - 1 - \frac{d}{2}}\sqrt{\Es_\Bbbk}.$$
$(iii)$ Weighted bounds for $\py^i \vec q$: 
\begin{equation}\label{eq:weipykq}
\sum_{i = 0}^{\Bbbk}\int \frac{|\py^i q_1|^2}{1 + y^{2\Bbbk - 2i}} + \sum_{i = 0}^{\Bbbk - 1}\int \frac{|\py^i q_2|^2}{1 + y^{2\Bbbk - 2i - 2}}   \leq c(M) \Es_{\Bbbk},
\end{equation}
hence,
\begin{equation}\label{eq:Hkbound}
\|\vec q \,\|_{\dot{H}^\Bbbk \times \dot{H}^{\Bbbk - 1}}^2 \leq c(M)\Es_\Bbbk \quad \text{and} \quad \|\vec q \,\|_{\dot{H}^\beta \times \dot{H}^{\beta - 1}}^2 \leq c(M)\Es_\sigma^\frac{\Bbbk - \beta}{\Bbbk - \sigma}\Es_\Bbbk^\frac{\beta - \sigma}{\Bbbk - \sigma} \quad \text{for} \quad \sigma \leq \beta \leq \Bbbk.
\end{equation}
Moreover, for $j \in \mathbb{N}$ and $p > 0$ satisfying $\sigma \leq j + p \leq \Bbbk$, we have
\begin{equation}\label{eq:interBound}
\int_{y \geq 1} \frac{|\py^j q_1|^2}{y^{2p}} \leq c(M)\Es_\sigma^{\frac{\Bbbk - (j + p)}{\Bbbk - \sigma}}\Es_\Bbbk^{\frac{(j+p) -\sigma}{\Bbbk - \sigma}}.
\end{equation}

\noindent $(iv)$ Weighted $L^\infty$ control: Let $a > 0$ satisfying $\sigma - \frac{d}{2} \leq p \leq \Bbbk -\frac{d}{2}$, then 
\begin{equation}\label{eq:boundLinfq}
\left\|\frac{q_1}{y^p}\right\|^2_{L^\infty(y \geq 1)} \lesssim \Es_\sigma^{\frac{\Bbbk - \frac{d}{2} - p}{\Bbbk - \sigma}}\Es_\Bbbk^{\frac{p - \sigma + \frac{d}{2}}{\Bbbk - \sigma}}.
\end{equation}
Let $j \in \mathbb{N}^*$ and $p \geq 0$ such that $1 \leq j \leq \Bbbk - \frac{d}{2}$ and $0 \leq p \leq \Bbbk - j - \frac{d}{2}$, then 
\begin{equation}\label{eq:boundLinfq2}
\left\|\frac{\py^jq_1}{y^p}\right\|^2_{L^\infty(y \geq 1)} \lesssim \Es_\sigma^{\frac{\Bbbk - j - p - \frac{d}{2}}{\Bbbk - \sigma}}\Es_\Bbbk^{\frac{j + \frac{d}{2} -  \sigma}{\Bbbk - \sigma}\left(1 - \frac{p}{\Bbbk - j - \frac d2}\right) + \frac{p}{\Bbbk - j - \frac d2}}.
\end{equation}
Moreover, if $1\leq j + p \ll L$, then
\begin{equation}\label{eq:boundLinfq1}
\left\|\frac{\py^j q_1}{y^p}\right\|^2_{L^\infty(y\geq 1)} \lesssim  b_1^{2(j + p) + \frac{2\gamma(j + p)}{L} + \Oc\left(\frac{1}{L^2}\right)}\Es_\sigma^{1 + \Oc\left( \frac{1}{L}\right)}.
\end{equation}
\end{lemma}
\begin{proof} $(i)$ The estimate \eqref{eq:qmbyE2k} directly follows from Lemma \ref{lemm:coerEk}. 

$(ii)$ Without loss of generality, we assume that $\hbar$ is an even integer so that $\Bbbk = L + \hbar + 1$ is also an even integer. By \eqref{eq:formTk}, the expansion \eqref{eq:expqat0} is equivalent to

\begin{equation}\label{eq:expanq1q2at0}
q_1 = \sum_{i = 1}^{\frac{\Bbbk}{2}}c_{2i}T_{\Bbbk - 2i} + r_1= \sum_{i = 1}^{\frac{\Bbbk}{2}}c_{2i}\phi_{\frac{\Bbbk}{2} - i} + r_1 , \quad q_2 = \sum_{i = 1}^{\frac{\Bbbk}{2} - 1}c_{2i+1}T_{\Bbbk - 2i - 1} + r_2=\sum_{i = 1}^{\frac{\Bbbk}{2} - 1}c_{2i+1}\phi_{\frac \Bbbk 2 - i - 1} + r_2,
\end{equation}
where we recall that $T_{2k} = T_{2k + 1} = \phi_k$ with $\phi_k$ being defined as in \eqref{def:phik}. We only deal with the expansion of $q_1$ because the same proof holds for $q_2$.  We claim that for $1 \leq m \leq \frac{\Bbbk}{2}$, $(q_1)_{\Bbbk - 2m}$ admits the Taylor expansion at the origin
\begin{equation}\label{eq:expandq2k}
(q_1)_{\Bbbk - 2m} = \sum_{i = 1}^{m} c_{i,m}\phi_{m - i} + (r_1)_{m},
\end{equation}
with the bounds 
$$|c_{i,m}| \lesssim \sqrt{\Es_{\Bbbk}},$$
$$\sum_{j = 0}^{2m-1}|\py^j (r_1)_{m}| \lesssim y^{2m - \frac{d}{2} - j}\sqrt{\Es_{\Bbbk}}, \quad \text{for}  \quad y < 1.$$
The expansion \eqref{eq:expqat0} for $q_1$ then follows from \eqref{eq:expandq2k} with $m = \frac{\Bbbk}{2}$.

We proceed by induction in $m$ for the proof of \eqref{eq:expandq2k}. For $m = 1$, we write from the definition \eqref{def:Astar} of $\As^*$,
$$(e_1)_1(y) =  (q_1)_{\Bbbk - 1}(y) = \frac{1}{y^{d-1}\Lambda Q}\int_0^y (q_1)_{\Bbbk} \Lambda Q x^{d-1}dx + \frac{d_1}{y^{d-1}\Lambda Q}.$$
Note from \eqref{eq:qmbyE2k} that $\int \frac{|(q_1)_{\Bbbk - 1}|^2}{y^2} \lesssim \Es_{\Bbbk}$ and from \eqref{eq:asymLamQ} that $\Lambda Q \sim y$ as $y \to 0$, we deduce that $d_1 = 0$. Using the Cauchy-Schwartz inequality, we derive the pointwise estimate
$$|(e_1)_1(y)| \leq \frac{1}{y^d} \left(\int_0^y |(q_1)_{\Bbbk}|^2 x^{d-1}dx\right)^\frac{1}{2}\left(\int_0^y x^2x^{d-1} dx\right)^\frac{1}{2} \lesssim y^{-\frac{d}{2} + 1}\sqrt{\Es_{\Bbbk}} \quad \text{for}\; y < 1.$$ 
We remark from \eqref{eq:qmbyE2k} that there exists $a \in (1/2,1)$ such that 
$$|(e_1)_1(a)| = |(q_1)_{\Bbbk - 1}(a)|^2 \lesssim \int_{y < 1} |(q_1)_{\Bbbk - 1}|^2 \Es_{\Bbbk}.$$
We then define 
$$(r_1)_1(y) = -\Lambda Q \int_a^y \frac{(e_1)_1}{\Lambda Q}dx,$$
and obtain from the pointwise estimate of $(e_1)_1$, 
$$|(r_1)_1(y)| \lesssim  y^{-\frac{d}{2} + 2}\sqrt{\Es_{\Bbbk}} \quad \text{for}\;\; y < 1.$$
By construction and the definition \eqref{def:As} of $\As$, we have 
$$\Ls (r_1)_1 = \As^*(q_1)_{\Bbbk-1} = (q_1)_{\Bbbk} = \Ls (q_1)_{\Bbbk - 2}.$$
Recall that $\text{span}(\Ls) = \{\Lambda Q, \Gamma\}$ where $\Gamma$ admits the singular behavior \eqref{eq:asymGamma}. From \eqref{eq:qmbyE2k}, we have $\int\frac{|(q_1)_{\Bbbk - 2}|^2}{y^2} \lesssim \Es_{\Bbbk} < +\infty$. This implies that there exists $c_{1,1} \in \Rb$ such that 
$$(q_1)_{\Bbbk - 2} = c_{1,1} \Lambda Q + (r_1)_1 = c_{1,1}\phi_0 + (r_1)_1.$$
Moreover, from \eqref{eq:qmbyE2k} there exists $a \in (1/2,1)$ such that
$$|(q_1)_{\Bbbk - 2}(a)|^2 \lesssim \int_{y < 1} |(q_1)_{\Bbbk - 2}|^2\lesssim \Es_{\Bbbk},$$
which follows
$$|c_{1,1}| \lesssim \sqrt{\Es_{2\Bbbk}}, \quad |(q_1)_{\Bbbk - 2}| \lesssim y^{-\frac{d}{2} + 2}\sqrt{\Es_{2\Bbbk}} \quad \text{for}\; y < 1.$$
Since $\As (r_1)_1 = (e_1)_1$, we then write from the definition \eqref{def:As} of $\As$, 
$$|\py (r_1)_1| \lesssim |(e_1)_1| + \left|\frac{(r_1)_1}{y}\right| \lesssim y^{-\frac{d}{2} + 1}\sqrt{\Es_{\Bbbk}} \quad \text{for}\; y < 1.$$
This concludes the proof of \eqref{eq:expandq2k} for $m = 1$.   

We now assume that \eqref{eq:expandq2k} holds for $j = 1, \cdots, m$  for some $m \geq 1$, and prove that \eqref{eq:expandq2k} holds for $j = 1, \cdots, m + 1$. The term $(r_1)_{m + 1}$ is built as follows:
$$(r_1)_{m + 1} = - \Lambda Q\int_a^y \frac{(e_1)_{m + 1}}{\Lambda Q} dx,$$
where $a \in (1/2, 1)$ and 
$$(e_1)_{m + 1} = \frac{1}{y^{d-1}\Lambda Q}\int_0^y (r_1)_{m} \Lambda Q x^{d-1}dx.$$
We now use the induction hypothesis to estimate
\begin{align*}
|(e_1)_{m + 1}| &= \left|\frac{1}{y^{d-1} \Lambda Q}\int_0^y (r_1)_{m} \Lambda Q x^{d-1}dx \right| \\
&\quad \lesssim \frac{1}{y^d}\sqrt{\Es_{\Bbbk}}\int_0^y x^{2m + \frac{d}{2}} dx \lesssim y^{2m - \frac{d}{2} + 1} \sqrt{\Es_{\Bbbk}}.
\end{align*}
Then, we have
\begin{align*}
|(r_1)_{m + 1}| = \left|\Lambda Q\int_a^y \frac{(e_1)_{m +1}}{\Lambda Q} dx\right| &\lesssim y^{2m - \frac{d}{2} + 2} \sqrt{\Es_\Bbbk}.
\end{align*}
By construction, we have 
$$\Ls (r_1)_{m+1} = r_{m}.$$
From the induction hypothesis and the definition \eqref{def:Tk} of $T_k$, we write
$$\Ls (q_1)_{\Bbbk - 2(m + 1)} = (q_1)_{\Bbbk - 2m} = \sum_{i = 1}^m c_{i,m}\phi_{m - i} + (r_1)_{m} = \sum_{i = 1}^mc_{i,m}\Ls \phi_{m + 1 - i} + \Ls r_{m + 1}.$$
The singularity \eqref{eq:asymGamma} of $\Gamma$ at the origin and the bound $\int_{y < 1} \frac{|(q_1)_{\Bbbk - 2(m+1)}|^2}{y^2} \lesssim \Es_{\Bbbk}$ implies that
$$(q_1)_{\Bbbk - 2(m + 1)} = \sum_{i=1}^{m} c_{i,m}\phi_{m +1 - i}+ c_{m + 1, m}\Lambda Q +  (r_1)_{m + 1}.$$
From \eqref{eq:qmbyE2k}, we see that there exists $a \in (1/2,1)$ such that 
$$|(q_1)_{\Bbbk - 2(m+1)}(a)|^2 \lesssim \int_{y < 1}|(q_1)_{\Bbbk - 2(m+1)}|^2 \lesssim \Es_{\Bbbk},$$
from which we derive the bound $|c_{m+1, m}| \lesssim \sqrt{\Es_{\Bbbk}}$. For the estimate on $\py^j (r_1)_{m+1}$, we note that by construction 
$$(r_1)_{m + 2}:= \As (r_1)_{m + 1} = (e_1)_{m+1}, \quad (r_1)_{m + 3}:= \Ls (r_1)_{m + 1} = (r_1)_{m},$$
$$(r_1)_{m + 4}:= \As \Ls (r_1)_{m + 1} = \As (r_1)_m = (e_1)_{m}, \quad (r_1)_{m + 5}:=\Ls^{2}(r_1)_{m + 1} = (r_1)_{m - 1}, \cdots$$
$$(r_1)_{3m + 1}:=\Ls^m (r_1)_{m + 1} = (r_1)_1, \quad  (r_1)_{3m + 2}:=\As\Ls^{m}(r_1)_{m + 1}= (e_1)_1.$$
A brute force computation using the definitions of $\As$ and $\As^*$ and the asymptotic behavior \eqref{eq:asympV}  ensure that for any function $f$, we have
\begin{equation}\label{eq:pyjf}
\py^j f = \sum_{i = 0}^j P_{i,j}f_i \quad \text{with} \quad |P_{i,j}| \lesssim \frac{1}{y^{j - i}}.
\end{equation}
Hence, we have for $0 \leq j \leq 2m + 1$ and $y < 1$,
\begin{align*}
|\py^j r_{m + 1}|&\lesssim \sum_{i = 0}^j \frac{|r_{m + 1 + i}|}{y^{j-i}}\lesssim \sqrt{\Es_{\Bbbk}}\sum_{i = 0}^j \frac{y^{2m + 2 - i - \frac{d}{2}}}{y^{j - i}} \lesssim y^{2m + 2 -\frac{d}{2} - j}\sqrt{\Es_{\Bbbk}}.
\end{align*}
This concludes the proof of \eqref{eq:expandq2k} as well as \eqref{eq:expqat0}.

$(iii)$ We use \eqref{eq:pyjf}, \eqref{eq:qmbyE2k} and the expansion \eqref{eq:expqat0} to estimate
\begin{align*}
&\sum_{i = 0}^{\Bbbk}\int \frac{|\py^i q_1|^2}{1 + y^{2\Bbbk - 2i}} + \sum_{i = 0}^{\Bbbk - 1}\int \frac{|\py^i q_2|^2}{1 + y^{2\Bbbk - 2i - 2}}\\
& \lesssim \Es_{\Bbbk} + \sum_{i = 0}^{\Bbbk - 1} \int_{y < 1} |\py^i q_1|^2 +  \sum_{i = 0}^{\Bbbk - 2} \int_{y < 1} |\py^i q_2|^2  + \sum_{i = 0}^{\Bbbk - 1}\int_{y > 1}\frac{|\py^i(q_1)|^2}{y^{2\Bbbk - 2i}} +\sum_{i = 0}^{\Bbbk - 2}\int_{y > 1}\frac{|\py^i(q_2)|^2}{y^{2\Bbbk - 2i - 2}}\\
&\lesssim \Es_{\Bbbk} + \sum_{i = 0}^{\Bbbk - 1} \sum_{j = 0}^i \int_{y > 1}\frac{|(q_1)_j|^2}{y^{2\Bbbk - 2j}} + \sum_{i = 0}^{\Bbbk - 2} \sum_{j = 0}^i \int_{y > 1}\frac{|(q_2)_j|^2}{y^{2\Bbbk - 2j - 2}} \lesssim \Es_{\Bbbk}, 
\end{align*}
which concludes the proof of \eqref{eq:weipykq}. The estimate \eqref{eq:Hkbound} simply follows from an interpolation. The estimate \eqref{eq:interBound} follows from \eqref{eq:Hardy0} and the interpolation \eqref{eq:Hkbound}. 

$(iv)$ We apply the Hardy inequality \eqref{eq:Hardy1} to $q_1$ with $\alpha = \sigma - 1$, the bound \eqref{eq:interBound} with $j = 1$ and $p  = \sigma  -1$, and the expansion \eqref{eq:expqat0} to find that 
$$\left\|\frac{q_1}{y^{\sigma - \frac{d}{2}}} \right\|^2_{L^\infty(y \geq 1)} \lesssim \int_{y \geq 1}\frac{|\py q_1|^2}{y^{2(\sigma - 1)}} + |q_1(1)|^2 \lesssim \Es_\sigma.$$
Similarly, we have 
$$\left\|\frac{q_1}{y^{\Bbbk - \frac{d}{2}}} \right\|^2_{L^\infty(y \geq 1)} \lesssim \int_{y \geq 1}\frac{|\py q_1|^2}{y^{2(\Bbbk - 1)}} + |q_1(1)|^2 \lesssim \Es_\Bbbk.$$
An interpolation of the two estimates yields the bound \eqref{eq:boundLinfq}. 

For $1 \leq j \leq \Bbbk - \frac{d}{2}$, we have from Sobolev and the bound \eqref{eq:Hkbound},
\begin{equation}\label{eq:boudLinq1k}
\|\nabla^j q_1\|^2_{L^\infty} + \|\nabla^{\frac d2 + j} q_1\|^2_{L^2} \lesssim  \Es_\sigma^{\frac{\Bbbk - j - \frac{d}{2}}{\Bbbk - \sigma}}\Es_\Bbbk^\frac{j + \frac{d}{2} - \sigma}{\Bbbk - \sigma}.
\end{equation}
We apply \eqref{eq:Hardy1} to $\py^{j}q_1$ with $\alpha = \Bbbk - j - 1$, then use \eqref{eq:interBound} and \eqref{eq:expqat0} to estimate
\begin{equation*}
\left\|\frac{\py^j q_1}{y^{\Bbbk - j - \frac{d}{2}}} \right\|^2_{L^\infty(y \geq 1)} \lesssim \int_{y \geq 1} \frac{|\py^{j + 1} q_1|^2}{y^{2\Bbbk - 2j - 2}} + |\py^{j}q_1(1)|^2 \lesssim \Es_\Bbbk.
\end{equation*}
We interpolate for $0 \leq p \leq \Bbbk - j - \frac{d}{2}$, 
\begin{align*}
\left\|\frac{\py^j q_1}{y^{p}} \right\|_{L^\infty(y \geq 1)} &\lesssim \Es_\sigma^{\frac{\Bbbk - j - p - \frac{d}{2}}{\Bbbk - \sigma}}\Es_\Bbbk^{\frac{j + \frac{d}{2} -  \sigma}{\Bbbk - \sigma}\left(1 - \frac{p}{\Bbbk - j - \frac d2}\right) + \frac{p}{\Bbbk - j - \frac d2}}.
\end{align*}
If $1 \leq j  + p\ll L$, then we have 
$$\frac{\Bbbk - j - p - \frac{d}{2}}{\Bbbk - \sigma} = 1 + \Oc\left(\frac{1}{L}\right).$$
Recall from \eqref{def:sigma} and \eqref{def:B0B1} that $|\sigma - d/2| + \eta = \Oc\left(\frac{1}{L^2}\right)$, we compute the exponent
\begin{align*}
&\left[\frac{j + \frac{d}{2} -  \sigma}{\Bbbk - \sigma}\left(1 - \frac{p}{\Bbbk - j - \frac d2}\right) + \frac{p}{\Bbbk - j - \frac d2}\right](2L + 2(1 - \delta)(1 + \eta))\\
& \quad = 2(j + p) + \frac{2\gamma(j + p)}{L} +  \Oc\left(\frac{1}{L^2}\right).
\end{align*}
This yields
$$\left\|\frac{\py^j q_1}{y^{\Bbbk - j - \frac{d}{2}}} \right\|^2_{L^\infty(y \geq 1)} \lesssim \Es_\sigma^{1 + \Oc\left(\frac 1L\right)}b_1^{2(j+p) + \frac{2\gamma(j+p)}{L} + \Oc\left(\frac{1}{L^2}\right)}.$$
This concludes the proof of \eqref{eq:boundLinfq2} and \eqref{eq:boundLinfq1} as well as the proof of Lemma \ref{lemm:interbounds}.
\end{proof}

For the estimates of the linear and commutator terms in derivation of the monotonicity formula \eqref{eq:Es2sLya}, we need the following Leibniz rule for $\Ls^k$ whose proof can be found in \cite{GINapde18}, Lemma C.1:

\begin{lemma}[Leibniz rule for $\Ls^k$] \label{lemm:LeibnizLk}  Let $\phi$ be a smooth function and $k \in \mathbb{N}$, we have 
\begin{equation}\label{eq:LeibnizLk}
\Ls^{k + 1}(\phi f) = \sum_{m = 0}^{k+1}f_{2m}\phi_{2k+2, 2m} + \sum_{m = 0}^k f_{2m + 1}\phi_{2k + 2, 2m + 1},
\end{equation}
and 
\begin{equation}\label{eq:LeibnizALk}
\As\Ls^{k}(\phi f) = \sum_{m = 0}^{k}f_{2m + 1}\phi_{2k+1, 2m + 1} + \sum_{m = 0}^k f_{2m}\phi_{2k+1, 2m},
\end{equation}
where\\
- for $k = 0$,
\begin{align*}
&\phi_{1,0} = -\py \phi, \quad \phi_{1,1} = \phi,\\
&\phi_{2,0} = -\py^2 \phi - \frac{d-1 +2V}{y}\py \phi, \quad \phi_{2,1}= 2\py \phi, \quad \phi_{2,2} = \phi,
\end{align*}
- for $k \geq 1$
\begin{align*}
&\phi_{2k + 1, 0} = -\py \phi_{2k,0},\\
&\phi_{2k + 1, 2i} = - \py \phi_{2k, 2i} - \phi_{2k, 2i - 1}, \quad 1 \leq i \leq k,\\
&\phi_{2k+1, 2i + 1}= \phi_{2k, 2i} + \frac{d-1 + 2V}{y}\phi_{2k, 2i + 1} - \py\phi_{2k, 2i+1}, \quad 0 \leq i \leq k-1,\\
&\phi_{2k + 1, 2k + 1} = \phi_{2k, 2k} = \phi,\\
& \quad\\
&\phi_{2k + 2, 0} = \py \phi_{2k+1,0} + \frac{d-1 + 2V}{y}\phi_{2k+1, 0},\\
&\phi_{2k + 2, 2i} = \phi_{2k + 1, 2i - 1} + \py \phi_{2k + 1, 2i} + \frac{d-1 + 2V}{y}\phi_{2k + 1, 2i},\quad 1 \leq i \leq k,\\
&\phi_{2k+2, 2i + 1}= -\phi_{2k + 1,2i} + \py\phi_{2k+1, 2i + 1}, \quad 0 \leq i \leq k,\\
&\phi_{2k + 2, 2k + 2} = \phi_{2k + 1, 2k + 1} = \phi.
\end{align*}
\end{lemma}

\def\cprime{$'$}


\end{document}